\documentclass{amsart}

\usepackage[a4paper, left=1in, right=1in, top=1in, bottom=1in]{geometry}

\usepackage[dvips]{epsfig}
\usepackage{graphicx}
\usepackage{float}

\usepackage{ulem}

\usepackage{amsmath}
\usepackage{amssymb}
\usepackage{amsbsy}
\usepackage{bbm}
\usepackage{eucal}
\usepackage{mathrsfs}
\usepackage{mathabx}
\usepackage{empheq, esdiff}

\usepackage{enumitem}

\usepackage[colorlinks=true,citecolor=blue, linkcolor=red]{hyperref}

\usepackage{dsfont}
\usepackage{palatino}
\usepackage{xcolor}
\definecolor{darkgreen}{RGB}{0,100,0}
\definecolor{pink}{RGB}{255,105,180}

\usepackage{multirow}

\usepackage{tikz}
\usepackage{pgfplots}
\usepackage{changepage}

\usepackage{amsthm}

\newtheorem{theorem}{Theorem}

\newtheorem{assumption}{Assumption}
\newtheorem{lemma}{Lemma}

\newtheorem{corollary}{Corollary}
\newtheorem{definition}{Definition}
\newtheorem{proposition}{Proposition}
\newtheorem{remark}{Remark}
\newtheorem{notation}{Notation}

\usepackage{amscd}

\usepackage{tensor}

\newcommand{\opnorm}[1]{{\vert\kern-0.25ex\vert\kern-0.25ex\vert #1 
  \vert\kern-0.25ex\vert\kern-0.25ex\vert}}

\newcommand{\1}{\mathbbm{1}}

\DeclareMathOperator{\Var}{Var}

\newcommand{\R}{\mathbb{R}}
\newcommand{\E}{\mathbb{E}}

\newcommand{\bs}[1]{\boldsymbol{#1}}

\newcommand{\calG}{{\mathcal{G}}}
\newcommand{\calF}{{\mathcal{F}}}

\newcommand{\wt}{\widetilde}
\newcommand{\wh}{\widehat}
\newcommand{\wb}{\widebar}

\newcommand{\bfgm}{{\boldsymbol{\gamma}}}
\newcommand{\bfu}{{\mathbf{u}}}
\newcommand{\bfv}{{\mathbf{v}}}
\newcommand{\bfw}{{\mathbf{w}}}
\newcommand{\bfj}{{\mathbf{j}}}
\newcommand{\bfe}{{\mathbf{e}}}
\newcommand{\bfp}{{\mathbf{p}}}
\newcommand{\bfJ}{{\mathbf{J}}}
\newcommand{\bfE}{{\mathbf{E}}}
\newcommand{\bfP}{{\mathbf{P}}}

\def\barroman#1{\sbox0{#1}\dimen0=\dimexpr\wd0+1pt\relax
  \makebox[\dimen0]{\rlap{\vrule width\dimen0 height 0.06ex depth 0.06ex}%
    \rlap{\vrule width\dimen0 height\dimexpr\ht0+0.03ex\relax 
            depth\dimexpr-\ht0+0.09ex\relax}%
    \kern.5pt#1\kern.5pt}}

\newcommand{\Cbias}{C_{\mathrm{B}}}
\newcommand{\Ctbias}{\widetilde{C}_{\mathrm{B}}}
\newcommand{\Cvar}{C_{\mathrm{V}}}

\makeatletter
\newcommand{\settheoremtag}[1]{
	\let\oldtheassumption\theassumption
	\renewcommand{\theassumption}{#1}
	\g@addto@macro\endassumption{
        \global\let\theassumption\oldtheassumption
    }
}
\makeatother

\makeatletter
\newcommand{\leqnomode}{\tagsleft@true}
\newcommand{\reqnomode}{\tagsleft@false}
\makeatother

\allowdisplaybreaks

\usepackage{setspace}

\newcommand{\invar}{\overline{\mu}^b}

\def\bbI{\mathbb{I}}

\def\bbR{\mathbb{R}}

\begin{document}

\date{\today}

\title{Estimation of the invariant measure of a multidimensional diffusion from noisy observations}

\author[R. Maillet; G. Szymanski]
{Rapha\"el Maillet and Gr\'egoire Szymanski
}

\address{Rapha\"el Maillet, Capital Fund Management (CFM)
 }
\email{maillet@ceremade.dauphine.fr}
\address{Gr\'egoire Szymanski, DMATH, Université du Luxembourg}
\email{gregoire.szymanski@uni.lu}

\begin{abstract} 
We introduce a new approach for estimating the invariant density of a multidimensional diffusion when dealing with high-frequency observations blurred by independent noise. We consider the intermediate regime, where observations occur at discrete time instances $k\Delta_n$ for $k=0,\dots,n$, under the conditions $\Delta_n\to 0$ and $n\Delta_n\to\infty$. We construct a kernel density estimator based on preaveraged observations to reduce the effect of the noise and involves a two-step bias-correction procedure to appropriately account for the bias introduced by the pre-averaging. The rate of convergence of our estimator depends on both the anisotropic regularity of the density and the intensity of the noise. We establish conditions on the intensity of the noise that ensure the recovery of convergence rates similar to those achievable without any noise. Furthermore, we prove a Bernstein concentration inequality for our estimator, from which we derive an adaptive procedure for the kernel bandwidth selection.
\end{abstract}

\maketitle



\tableofcontents

\section{Introduction}

\subsection{Setting}
We revisit the classical statistical problem of recovering the distribution of a latent signal observed with additive noise.  
More precisely, we consider the $d$-dimensional deconvolution model
\begin{equation}
\label{eq:model}
Y_{i,n} = X_{i,n} + \tau_n \xi_{i,n}, \qquad i=0,\dots,n,
\end{equation}
where $(X_{i,n})_{0\le i\le n}$ is a stationary sequence with common distribution $\mu$ on $\mathbb{R}^d$, the variables $(\xi_{i,n})_{i,n}$ form an i.i.d.\ sequence of standard Gaussian vectors independent of $X$, and $(\tau_n)_n$ is a known sequence governing the noise intensity.\\

In the purely i.i.d.\ setting, when the variables $X_{i,n}$ are independent, the noise is Gaussian and the target density is $\alpha$-H\"older, the best achievable pointwise quadratic rate for estimating $\mu$ at a fixed point is of order
\[
\log(n)^{-\alpha/2}\,\tau_n^{\alpha},
\]
see~\cite{fan1991optimal, comte2013anisotropic}. When $\tau_n$ is of order one, the rate becomes logarithmic and essentially unusable in practice: the reconstruction of the underlying density $\mu$ is statistically ill-posed.\\

The aim of this work is to show that this picture changes drastically when the signal $(X_{i,n})$ is no longer arbitrary but arises from a Markovian diffusion. In that case the temporal structure of the latent process can be exploited to "average out" the noise and recover polynomial convergence rates, even when the noise level does not vanish. Throughout the paper, we assume that $X$ is the solution of a $d$-dimensional stochastic differential equation
\begin{equation} 
\label{eq:sde}
\mathrm{d}X_t = b(X_t) \, \mathrm{d}t + \sigma(X_t) \, \mathrm{d}W_t, 
\end{equation}
defined on a filtered probability space $(\Omega, \mathcal{A}, \mathbb{P})$, where $W$ is a $d$-dimensional Brownian motion, and $b : \mathbb{R}^d \to \mathbb{R}^d$ and $\sigma : \mathbb{R}^d \to \mathbb{R}^d\otimes\mathbb{R}^d$ are the drift and diffusion coefficients, respectively. We will mostly work in the simplest case $\sigma \equiv I_d$. Under mild conditions on $b$ and $\sigma$, the process $X$ is ergodic and admits a unique invariant probability measure $\overline{\mu}^b$ on $\mathbb{R}^d$, which is absolutely continuous with respect to Lebesgue measure; we denote its density again by $\overline{\mu}^b$.\\

Our observation scheme is discrete and noisy. For a sequence $(\Delta_n)_n$ with $\Delta_n>0$, we define
\begin{equation}
\label{eq:observation}
X_{i,n} = X_{i\Delta_n}, \qquad i = 0,\dots, n.
\end{equation}
We work in the high-frequency--long-time regime
\[
\Delta_n \to 0, \qquad T_n := n\Delta_n \to \infty,
\]
and observe only the corrupted data $(Y_{i,n})_{0\le i\le n}$ from~\eqref{eq:model}.  
Our goal is to construct a nonparametric estimator of the invariant density $\overline{\mu}^b$ and to describe its pointwise quadratic risk in terms of the regularity of $\overline{\mu}^b$, the dimension $d$ and the sequences $(\Delta_n)$ and $(\tau_n)$.\\

If $\Delta_n$ is of order $1$ or larger, the discrete-time chain $(X_{i,n})$ is exponentially $\beta$-mixing and the problem reduces to the i.i.d.\ framework, with logarithmic deconvolution rates in the presence of Gaussian noise. When $\Delta_n \to 0$, however, the Markov and ergodic properties of~\eqref{eq:sde} can be exploited much more efficiently. In particular, we show that when $\tau_n$ is not too large and the sampling is sufficiently frequent, one can achieve polynomial convergence rates in $\Delta_n$ even when $\tau_n$ is constant, which is in sharp contrast with the i.i.d.\ case.

\subsection{Motivation and related work}

Diffusion processes of the form~\eqref{eq:sde} arise as continuous limits of Markov chains and have become a central modelling tool in many areas of applied mathematics, including physics, biology, economics and finance; see~\cite{papanicolaou1995diffusion, bergstrom1993history, hull2003options, bailey1957mathematical, ricciardi2013diffusion} among many others. Statistical inference for such processes has been extensively studied: parametric and nonparametric estimation of $b$ and $\sigma$ from discrete or continuous observations, as well as estimation of functionals of the invariant measure, play a key role in, for instance, Markov Chain Monte Carlo and stochastic approximation algorithms~\cite{lamberton2002recursive, panloup2008recursive}.\\

The estimation of the invariant density $\overline{\mu}^b$ itself has attracted considerable attention; see, e.g.,~\cite{nguyen1979ergodic, delecroix1980estimation, bosq2012nonparametric}. It is closely related to the estimation of the drift: in the reversible case $b=-\nabla V$, one has $\overline{\mu}^b \propto e^{-2V}$, and more generally nonparametric estimation of $\overline{\mu}^b$ and $b$ are intertwined~\cite{schmisser2013penalized}.\\

On the other hand, nonparametric density estimation from i.i.d.\ data has a long history starting with kernel estimators~\cite{rosenblatt1956remarks, parzen1962estimation}; see~\cite{tsybakov2009nonparametric} for a modern account. In the i.i.d.\ case, the kernel estimator
\[
\widehat{\mu}_n(x) = \frac{1}{n} \sum_{k=0}^{n-1} \bs{K}_{\bs{h}}(x-X_k)
\]
achieves the minimax rate $n^{-\alpha/(2\alpha+d)}$ over H\"older classes, when the bandwidth vector $\bs{h}$ is chosen appropriately, possibly in a data-driven way~\cite{goldenshluger2008universal, goldenshluger2009structural, goldenshluger2011bandwidth}. Extensions to ergodic processes are nontrivial; early work focused either on continuous observations $(X_t)_{t\le T}$, or on low-frequency discrete observations $(X_{i\Delta})_{0\le i\le n}$, where the mixing properties allow one to reuse i.i.d.\ techniques. In the continuous case, fine probabilistic tools for Markov semigroups and transition densities~\cite{bakry2008rate, cattiaux2008deviation, lezaud2001chernoff, paulin2015concentration, chen1997estimation, qian2003comparison, qian2004representation} yield sharp results: the seminal papers~\cite{dalalyan2006asymptotic, dalalyan2007asymptotic} derive optimal rates of order
\[
T^{-\alpha/(2\alpha + d-2)}
\]
for kernel estimators when $d\ge3$ and $\overline{\mu}^b$ is $\alpha$-H\"older, with extensions to anisotropic and minimax settings in~\cite{strauch2018adaptive, amorino2021minimax}.\\

The case of discrete high-frequency observations, i.e.\ $(X_{i\Delta_n})_{0\le i\le n}$ with $\Delta_n\to0$ and $T_n=n\Delta_n\to\infty$, has only recently been understood in a systematic way, despite its practical importance (e.g.\ in high-frequency econometrics~\cite{ait2014high}). The question is to determine precisely when the ``continuous-time'' rates and when the ``i.i.d.-like'' rates are relevant, depending on the relation between $\Delta_n$ and $T_n$. This has been clarified in~\cite{amorino2023estimation}, which identifies the breakeven point between the two regimes and shows that the continuous-time rate~\cite{amorino2021minimax} remains minimax-optimal for all dimensions, while the optimality of the low-frequency rate is known only in dimension one~\cite{amorino2022malliavin}.\\

A second, and in many applications unavoidable, issue is measurement noise. For i.i.d.\ data with additive noise, deconvolution methods based on Fourier inversion and kernel estimators have been extensively developed~\cite{devroye1989consistent, liu1989consistent, stefanski1990deconvolving}. Under the assumption that the noise distribution is known and its characteristic function does not vanish, minimax-optimality of these procedures is established in~\cite{carroll1988optimal, fan1991optimal, fan1993adaptively}, and refined in the anisotropic setting in~\cite{comte2013anisotropic}. When the noise is ``super-smooth'' (e.g.\ Gaussian), the resulting rates are notoriously slow: for an $\alpha$-H\"older target and standard Gaussian noise, any estimator can converge at most at rate $\log(n)^{-\alpha/2}$. \\

In contrast, very little is known about density estimation in noisy ergodic setups. A few notable contributions include~\cite{schmisser2011nondrift, schmisser2012nonvolatility}, which consider drift and volatility estimation for diffusions with noisy observations. However, for invariant density estimation, the interplay between the Markov structure of the latent process and the noise has not been exploited in full generality. \\

A natural idea for mitigating noise in high-frequency settings is to use preaveraging of the data
\[
\frac{1}{p}\sum_{\ell=0}^{p-1} Y_{kp+\ell,n}, \qquad 0\le k\le \lfloor n/p\rfloor-1.
\]
Under~\eqref{eq:model}--\eqref{eq:observation}, this block average decomposes as
\[
\frac{1}{p}\sum_{\ell=0}^{p-1} X_{(kp+\ell)\Delta_n}
\;+\; \frac{\tau_n}{\sqrt{p}}\,\widetilde\xi_{k,n},
\]
where $(\widetilde\xi_{k,n})_k$ is i.i.d.\ standard Gaussian. Thus preaveraging reduces the noise variance by a factor $1/p$, while the first term is a short-time average of the diffusion. Its invariant law is a smooth perturbation of $\overline{\mu}^b$, so preaveraging should in principle improve the signal-to-noise ratio without fundamentally altering the target distribution. A major difficulty is that the sequence of block averages is not Markovian. Following~\cite{gloter2008lamn}, we recover Markovianity by augmenting the state with the block endpoint:
\[
Z_k := \Big( \tfrac{1}{p}\sum_{\ell=0}^{p-1} X_{(kp+\ell)\Delta_n},\; X_{(k+1)p\Delta_n} \Big),
\qquad k\ge0.
\]
This yields a time-homogeneous Markov chain in $\mathbb{R}^{2d}$. But this chain is technically challenging: its diffusion part lives in a $d$-dimensional Brownian motion whereas the state space is $2d$-dimensional, so standard ellipticity arguments do not apply, and the averaging operator is nonlocal in time.

A central probabilistic contribution of this paper is to prove that, under our standing assumptions on $b$, the chain $(Z_k)$ is $\psi$-irreducible, aperiodic, and geometrically ergodic, with unique invariant law $\overline{\pi}^{p,b}$. We additionally derive an expansion of the invariant density $m^h(\overline{\mu}^{p,b})$ in powers of $p\Delta_n$, with remainder of order $(p\Delta_n)^{\beta/2}$ when $\overline{\mu}^b$ is $\beta$-times differentiable. This allows for a Richardson extrapolation step that removes the leading preaveraging bias. Moreover, using Gaussian-type bounds for diffusion transition densities~\cite{menozzi2020density}, we obtain analogous bounds for the preaveraged process. Combined with geometric ergodicity, these yield sharp variance controls and a Bernstein-type inequality for kernel estimators applied to the preaveraged data. To our knowledge, the only prior analysis of a general preaveraged diffusion-based model is~\cite{gloter2008lamn}, restricted to dimension~$1$.

\subsection{Contributions}

We now summarize our main contributions, combining the statistical and probabilistic aspects of the problem.

\begin{itemize}
    \item \textbf{Noisy invariant density estimation in the diffusion setting.}  
    We construct a kernel-type estimator of the invariant density $\overline{\mu}^b$ based on high-frequency noisy observations~\eqref{eq:model}--\eqref{eq:observation}. The estimator uses a preaveraging step to reduce the noise, a Richardson extrapolation to remove the leading preaveraging bias, and a finite-order deconvolution correction that exploits the H\"older regularity of $\overline{\mu}^b$.\\

    \item \textbf{Polynomial rates in the presence of Gaussian noise.}  
    When the noise level $\tau_n$ is not too large (relative to $\Delta_n$), we show that the estimator attains the same rates as in the noise-free case studied in~\cite{amorino2023estimation}. In particular, in the high-frequency regime the rate coincides with the continuous-time rate of~\cite{dalalyan2007asymptotic, amorino2021minimax}. When the noise is larger, we prove that one can still achieve a rate of order
    \[
    \big(\tau_n^2\Delta_n\big)^{\frac{\alpha_1\beta}{\alpha_1+\beta}},
    \]
    which remains polynomial in $\Delta_n$ even when $\tau_n$ is constant. This is a substantial improvement over the logarithmic rates known in the i.i.d.\ deconvolution setting. Remarkably, this rate does not depend on the dimension $d$.\\

    \item \textbf{Probabilistic analysis of the preaveraged latent process.}  
    We develop a general framework for the preaveraged process $m^h_{kh,(k+1)h}(X)$ associated with a diffusion, including:
    \begin{itemize}
        \item the Markov and geometric ergodic properties of the augmented chain $(Z^h_k)$,
        \item an expansion of the invariant density $m^h(\overline{\mu}^b)$ in powers of $h$,
        \item Gaussian-type bounds on the transition densities of the preaveraged process, derived from the bounds in~\cite{menozzi2020density}.
    \end{itemize}
    These results may be of independent interest for the study of hidden Markov models built from continuous-time diffusions.\\

    \item \textbf{Adaptive bandwidth selection via a Bernstein inequality.}  
    We derive a Bernstein-type concentration inequality for our estimator at fixed bandwidth and preaveraging level. This inequality is obtained by combining exponential ergodicity of the latent Markov chain with the transition density bounds mentioned above. Following the Goldenshluger--Lepski methodology~\cite{goldenshluger2008universal, goldenshluger2009structural, goldenshluger2011bandwidth}, we then construct a fully data-driven bandwidth selector and establish an oracle inequality for the resulting adaptive estimator.
\end{itemize}

\subsection{Organization of the paper}

Section~\ref{sec:model} introduces the statistical and probabilistic framework. We specify the assumptions on the drift $b$ and the diffusion $X$, describe the high-frequency noisy observation scheme, and recall the anisotropic H\"older classes used to quantify the smoothness of $\overline{\mu}^b$. In Section~\ref{sec:hidden}, we study the probabilistic structure of the hidden preaveraged process. We introduce the general averaging operator $m^h$, construct the augmented Markov chain $(Z^h_k)$, prove its geometric ergodicity, and derive an expansion of the invariant density $m^h(\overline{\mu}^b)$ as well as bounds on the associated transition densities. Section~\ref{sec:estimation} is devoted to the construction of our estimator. We first define the preaveraged kernel estimator, then correct its preaveraging bias via Richardson extrapolation and its noise-induced bias via a finite-order deconvolution step, leading to the final estimator $\widehat{\mu}_{n,\bs{h},p}$. We then study this estimator and give some bounds on its variance. In Section~\ref{sec:upper_bounds} we analyse the pointwise quadratic risk of the estimator. We obtain separate bias and variance bounds, identify the optimal choices of preaveraging parameter $p$ and bandwidth vector $\bs{h}$, and derive the resulting rates in the different asymptotic regimes (small/large noise, high/low sampling frequency). In Section~\ref{sec:comments}, we give a few concluding remarks and observations.
Lastly, all proofs are gathered in the Appendix.

\section{Statistical and Probabilistic framework}
\label{sec:model}

\subsection{Notation} 

For any $x \in \mathbb{R}^d$, we denote by $|x|^2 := x \cdot x$ the squared Euclidean norm.  
The space of probability measures on $\mathbb{R}^d$ is denoted by $\mathcal{P}(\mathbb{R}^d)$.  
For a differentiable function $f : \mathbb{R}^d \to \mathbb{R}$, $\nabla f$ stands for its gradient, and for $k \ge 1$ and $1 \le i \le d$, the notation $\partial_i^k f$ refers to the $k$-th partial derivative with respect to the $i$-th coordinate.  
For $g : \mathbb{R}^d \to \mathbb{R}^d$, $\nabla \cdot g$ denotes the divergence of $g$.\\

Let $\nu$ be a $\sigma$-finite measure on $\mathbb{R}^d$ and $q \ge 1$.  
We say that $f : \mathbb{R}^d \to \mathbb{R}$ belongs to $L^q(\nu)$ if
\[
\| f \|_{L^q(\nu)}^q := \int_{\mathbb{R}^d} |f(x)|^q \, \nu(\mathrm{d}x) < +\infty.
\]
When $\nu$ is the Lebesgue measure, we simply write $L^q$ and denote the corresponding norm by $\|\cdot\|_q$.  
For $q = \infty$, we use
\[
\|f\|_\infty := \sup_{x\in\mathbb{R}^d} |f(x)|.
\]

\subsection{The probabilistic diffusion model}

Let $(\Omega,\mathcal{A},\mathbb{P})$ be a sufficiently rich probability space.
For a Lipschitz continuous drift $b : \mathbb{R}^d \to \mathbb{R}^d$ and an initial law $\mu \in \mathcal{P}(\mathbb{R}^d)$, we denote by $\mathbb{P}^b_\mu$ the law under which $X$ solves the stochastic differential equation
\begin{equation}\label{eq:diff:X}
    \mathrm{d}X_t = b(X_t)\,\mathrm{d}t + \mathrm{d}W_t, 
    \qquad \mathcal{L}(X_0) = \mu,
\end{equation}
where $W$ is a $d$-dimensional $\mathbb{P}^b_\mu$-Brownian motion and $\mathbb{F}=(\mathcal{F}_t)_{t\ge0}$ denotes its natural filtration.
Expectation with respect to $\mathbb{P}^b_\mu$ is written $\mathbb{E}^b_\mu$, with the shorthand $\mathbb{P}^b_x := \mathbb{P}^b_{\delta_x}$ and $\mathbb{E}^b_x := \mathbb{E}^b_{\delta_x}$.

We now list the structural assumptions imposed on $b$.

\begin{assumption}\label{assumption:boundedness}
The function $b$ is differentiable and satisfies
\[
|b(0)| \le b_0,
\qquad 
\text{and} \qquad
\|\partial_i b\|_\infty \le b_1/d,\; 1\le i\le d,
\]
for some constants $b_0,b_1>0$.  
In particular, $\|\nabla \cdot b\|_\infty \le b_1$.
\end{assumption}

\begin{assumption}\label{assumption:ergodicity}
There exist constants $b_2, b_3 > 0$ such that for all $x$, we have
\[
\langle b(x),x\rangle \le - b_2 |x| + b_3.
\]
\end{assumption}

Assumptions~\ref{assumption:boundedness}--\ref{assumption:ergodicity} ensure that the diffusion~\eqref{eq:diff:X} admits a unique invariant probability measure, denoted by $\overline{\mu}^b$.  
Writing
\[
\overline{\mu}^b(x) = \frac{1}{Z_b}\exp(-2V_0(x)),
\]
for a suitable normalizing constant $Z_b$, defines a function $V_0$ (unique up to an additive constant) such that $b_{\rm rev}:=-\nabla V_0$ shares the same invariant measure as $b$.
The residual drift is defined by
\begin{equation}
    \label{eq:residual_drift}
\widetilde b(x) := b(x) + \nabla V_0(x) = b(x)-b_{\rm rev}(x).
\end{equation}
When $\widetilde b(x) = 0$ for all $x$, the drift function derives from a potential, which ensures that the diffusion is reversible. In the present work, we do not assume reversibility, but we require that $\widetilde b$ has bounded derivatives. Therefore, we impose that $\widetilde b$ satisfies Assumption~\ref{assumption:boundedness} with possibly larger constants $b_0$ and $b_1$.

The solution to~\eqref{eq:diff:X} is a time-homogeneous Markov process.  
Assumption~\ref{assumption:ergodicity} ensures quantitative convergence to equilibrium.  
Indeed, under these assumptions, one obtains a Poincaré inequality for the invariant measure $\overline{\mu}^b$ (see, e.g.,~\cite{bakry2014analysis}), yielding exponential contraction of the semigroup $(P_t^b)_{t\ge0}$ defined by
\begin{equation}\label{eq:defSG}
P_t^b f(x) = \mathbb{E}^b_x[f(X_t)] .
\end{equation}
More precisely, there exists $C_{PI}>0$ depending only on $\widetilde C_b$ such that, for all $f\in L^2(\overline{\mu}^b)$ with mean zero,
\begin{equation}\label{eq:ergodicity}
\mathrm{Var}^b_{\overline{\mu}^b}[P_t^b f(X_0)]
    \le e^{-2t/C_{PI}}\,\mathbb{E}^b_{\overline{\mu}^b}[f(X_0)^2],\qquad t\ge0.
\end{equation}
In what follows, inequality~\eqref{eq:ergodicity} will serve as our main quantitative ergodicity assumption.

\medskip

Since the accuracy of density estimation crucially depends on the smoothness of $\overline{\mu}^b$ (see, e.g.,~\cite{dalalyan2007asymptotic,strauch2018adaptive,amorino2021minimax,amorino2023estimation}), we assume that $\overline{\mu}^b$ lies in an anisotropic Hölder class.

\begin{definition}
Let $\bs{\alpha}=\left(\alpha_1, \ldots, \alpha_d\right) \in (0, \infty)^d$ and $\bs{\mathcal{L}}=\left(\mathcal{L}_1, \ldots, \mathcal{L}_d\right) \in (0, \infty)^d$. A function $g: \mathbb{R}^d \rightarrow \mathbb{R}$ is said to belong to the anisotropic Hölder class $\mathcal{H}_d(\bs{\alpha}, \bs{\mathcal{L}})$ if, for all $1 \leq i \leq d$, the function $g$ is $\lfloor \alpha_i \rfloor$ times differentiable with respect to the $i$-th variable and the partial derivatives satisfy, for $0 \leq k \leq \lfloor \alpha_i \rfloor$,
\begin{equation*}
\|\partial_i^k g\|_{\infty} \leq \mathcal{L}_i \,\, \text{ and }  \,\,
 \forall t \in \mathbb{R},\,
\big\|\partial_i^{\lfloor \alpha_i \rfloor} g(\cdot + te_i) - \partial_i^{\lfloor \alpha_i \rfloor}g(\cdot)\big\|_{\infty} \leq  \mathcal{L}_i |t|^{\alpha_i - \lfloor \alpha_i \rfloor},
\end{equation*}
where $(e_1, \ldots, e_d)$ is the canonical basis of $\mathbb{R}^d$.
\end{definition}

More precisely, we assume the following.

\begin{assumption}\label{assumption:isotropic}
The invariant measure $\wb \mu^b$ belongs to the anisotropic Hölder class $\mathcal{H}_d(\boldsymbol{\alpha},\boldsymbol{\mathcal{L}})$ for some $\bs{\alpha}=\left(\alpha_1, \ldots, \alpha_d\right) \in (0, \infty)^d$ and $\bs{\mathcal{L}}=\left(\mathcal{L}_1, \ldots, \mathcal{L}_d\right) \in (0, \infty)^d$. Moreover, there exists $\beta \geq 2$ such that $b$ is of class $\mathcal{C}^{\beta-1}$ and the invariant density $\overline{\mu}^b$ is of class $\mathcal{C}^{\beta}$, in the sense that all partial derivatives of order up to $\lfloor \beta-1 \rfloor$ (resp.\ $ \lfloor \beta \rfloor$) exist and are continuous, and the partial derivatives of order $\lfloor \beta-1 \rfloor$ (resp.\ $ \lfloor \beta \rfloor$) are locally $(\beta-1-\lfloor \beta-1 \rfloor)$-Hölder continuous (resp.\ locally $(\beta-\lfloor \beta \rfloor)$-Hölder continuous). Moreover, we assume that all the derivatives of $b$ have polynomial growth, in the sense that for every multi-index $\bfgm$ with $|\bfgm|\le \lfloor \beta-1 \rfloor$, there exist constants $C_\bfgm,q_\bfgm<\infty$ such that
\[
|\partial^\bfgm b(y)|\le C_\bfgm(1+|y|^{q_\bfgm}),\qquad y\in\mathbb R^d,
\]
and we assume that for each compact $K \subset \bbR^d$ there exist constants $C_{K,\gamma},q_{K,\gamma}<\infty$ such that
\begin{equation*}
    \begin{cases*}
        |\partial^\gamma b(y+z)-\partial^\gamma b(y)|
\le C_{K,\gamma}(1+|z|^{q_{K,\gamma}})|z|^{\beta-\lfloor \beta \rfloor},
        \\
        |\partial^\gamma\overline\mu^b(y+z)-\partial^\gamma\overline\mu^b(y)|
\le C_{K,\gamma}(1+|z|^{q_{K,\gamma}})|z|^{\beta-\lfloor \beta \rfloor},
    \end{cases*}
    \qquad y\in K,\ z\in\mathbb R^d .
\end{equation*}
\end{assumption}

\begin{remark}
Assumption~\ref{assumption:isotropic} reflects the regularity regime in which our minimax analysis is conducted. In the reversible case $b=-\nabla V$, classical elliptic regularity implies that the invariant density is one degree smoother than $b$. In the non-reversible case, the relationship between the smoothness of $b$ and that of $\overline{\mu}^b$ is more subtle and depends on hypoellipticity properties of the generator. Our assumption should therefore be interpreted as a regularity condition directly imposed on the pair $(b,\overline{\mu}^b)$. Note also that Assumption~\ref{assumption:boundedness} already ensures that $b$ is of class $\mathcal{C}^{1}$. In Assumption~\ref{assumption:isotropic}, we assume that $b$ is of class $\mathcal{C}^{\beta-1}$ with $\beta - 1 \geq 1$. This additional assumption is used to improve the convergence rate of the estimator built in Section~\ref{sec:construction} using a Richardson extrapolation derived from Proposition~\ref{proposition:expansion_hidden_measure}. The polynomial weights in the last part of Assumption~\ref{assumption:isotropic} are only used in local-in-$x$ estimates; Brownian Gaussian tails then make all weighted moments finite on each compact set of evaluation points.
\end{remark}

\begin{definition}[Model class]
Let $\bs{\alpha}=\left(\alpha_1, \ldots, \alpha_d\right) \in (0, \infty)^d$, $\bs{\mathcal{L}}=\left(\mathcal{L}_1, \ldots, \mathcal{L}_d\right) \in (0, \infty)^d$, $b_0, b_1, b_2, b_3 > 0$ and $2 \leq \beta \leq \min_i \alpha_i$. We write $\bs{\mathfrak{b}} = (\bs{\alpha}, \bs{\mathcal{L}}, \beta, b_0, b_1, b_2, b_3)$ and $\Sigma(\bs{\mathfrak{b}})$ for the set of functions $b: \mathbb{R}^d \to \mathbb{R}^d$ satisfying Assumption~\ref{assumption:boundedness} with constants $b_0$ and $b_1$, Assumption~\ref{assumption:ergodicity} with constants $b_2$ and $b_3$ and Assumption~\ref{assumption:isotropic} with constants $\bs{\alpha}$, $\bs{\mathcal{L}}$ and $\beta$, and such that the residual drift $\wt b$ defined in~\eqref{eq:residual_drift} satisfies Assumption~\ref{assumption:boundedness} with constants $b_0$ and $b_1$.
\end{definition}

\medskip

We handle the low-dimensional cases $d = 1$ and $d = 2$ separately. For $d\ge3$, it is convenient to classify anisotropic smoothness vectors $\boldsymbol{\alpha}$ according to the multiplicities of their smallest components. 
Let $k_0=k_0(\boldsymbol{\alpha})$ be such that
\[
\alpha_1=\cdots=\alpha_{k_0} < \alpha_{k_0+1}\le\cdots\le\alpha_d .
\]
We then partition the anisotropic regimes into
\begin{equation}
\begin{aligned}\label{eq:D}
D_1 &= \{\boldsymbol{\alpha}: (k_0=1\ \text{and}\ \alpha_2<\alpha_3)\ \text{or}\ k_0=2\},\\
D_2 &= \{\boldsymbol{\alpha}: k_0\ge3\},\\
D_3 &= \{\boldsymbol{\alpha}: k_0=1\ \text{and}\ \alpha_2=\alpha_3\}.
\end{aligned}
\end{equation}
This partition will later be used to express the different minimax regimes and to guide the construction of optimal estimators.
Note that with this definition, we have
\begin{equation*}
\begin{aligned}
	& D_1 = \left\{ \bs{\alpha} , \: \alpha_1 \leq \alpha_2 < \alpha_3 \right\}; \\
	& D_2 = \left\{ \bs{\alpha} , \: \alpha_1 = \alpha_2 = \alpha_3 \right\};\\
	& D_3 = \left\{ \bs{\alpha} , \: \alpha_1 < \alpha_2 = \alpha_3\right\}.
\end{aligned}
\end{equation*}

\subsection{The statistical model}

We now describe the observation scheme.
Throughout the paper, the drift $b$ is assumed to belong to $\Sigma(\boldsymbol{\mathfrak{b}})$ for some parameters $\boldsymbol{\mathfrak{b}} = (\boldsymbol{\alpha},\boldsymbol{\mathcal{L}},b_0,b_1,b_2,b_3)$ with $\alpha_1\le\cdots\le\alpha_d$.

The diffusion $X$ is observed at discrete times $\{i\Delta_n\}_{0\le i\le n}$ and is corrupted by additive Gaussian noise:
\[
Y_{i,n} = X_{i\Delta_n} + \tau_n \xi_{i,n},
\]
where $(\xi_{i,n})_{i,n}$ is an i.i.d.\ family of standard Gaussian variables independent of both $W$ and $X_0$, and $\tau_n$ is a sequence bounded by $\tau^*$.  
We work in the high-frequency–-long-time asymptotic regime
\[
\Delta_n \to 0,
\qquad 
T_n := n\Delta_n \to \infty .
\]
This observation scheme is classical in high-frequency statistics for diffusions and forms the basis for the minimax analysis carried out in the sequel.

\section{Probabilistic Structure of the Hidden Process}
\label{sec:hidden}

\subsection{Setting}
\label{sec:hidden:setting}

In this section, we consider a probability measure $m$ defined on $[0,1]$. For any process $Y$ and any $h > 0$, we define
\begin{equation*}
m^h(Y) = \int_0^1 Y_{hs} \, m(\mathrm{d}s).
\end{equation*}
By a simple change of variables and an integration by parts argument applied to the semimartingale $Y$, we can rewrite $m^h(Y)$ as
\begin{equation*}
m^h(Y) = Y_0 + \int_0^1 \int_0^{hs} 1 \, \mathrm{d}Y_u \, m(\mathrm{d}s) = Y_0 + \int_0^{h} \chi(u/h) \, \mathrm{d}Y_u,
\end{equation*}
where
\begin{equation*}
\chi(u) = m([u,1]),
\end{equation*}
for all $u \geq 0$. Using this definition, we define
\begin{equation*}
m^h_{s,t}(Y) = Y_s + \int_s^{t} \chi((u-s)/h) \, \mathrm{d}Y_u,
\end{equation*}
for all $0 \leq s \leq t$. When $s=0$, we write $m^h_{t}(Y) = m^h_{0,t}(Y)$ for simplicity.\\

We are particularly interested in the distribution of the joint vector
\begin{equation}
\label{eq:hidden_process_framework}
\Big( m^h_{kh, (k+1)h}(X) \Big)_{0 \leq k < N},
\end{equation}
for some integer $N > 0$. Additionally, we aim to determine the stationary distribution of this vector, if it exists. Of special interest in our statistical study of Section~\ref{sec:estimation} is the case
\begin{equation}
\label{eq:measure_preav}
m = \frac{1}{p} \sum_{k=0}^{p-1} \delta_{k/p},
\end{equation}
for $p \geq 2$, representing the discrete preaveraging measure. Here, we consider a general measure $m$ for several reasons. First, the sequence in~\eqref{eq:hidden_process_framework} fits into the general hidden process framework studied in~\cite{gloter2008lamn}, which makes the present section of independent probabilistic interest. Second, several proofs are more naturally expressed for a generic averaging measure $m$ rather than for the specific discrete measure in~\eqref{eq:measure_preav}; the latter can then be recovered as a particular case by specialization of $m$.

\subsection{Markovian Structure of Hidden Processes}

Throughout this section, we fix $b\in\Sigma(\bs{\mathfrak{b}})$ and $0<h\le1$. We also suppose that $m$ is given by~\eqref{eq:measure_preav} for some $p \geq 2$. \\

The goal of this section is to make precise the Markovian structure underlying the hidden process~\eqref{eq:hidden_process_framework} and its augmented version~\eqref{eq:augmented_framework}. A key difficulty comes from the fact that the sequence
\[
\big(m^h_{kh,(k+1)h}(X)\big)_{0\le k<N}
\]
is not Markovian, so that many standard tools for Markov chains cannot be applied directly. Following~\cite{gloter2008lamn}, the crucial idea is to work instead with the augmented process $(Z^h_k)_k$ defined by
\begin{equation}
\label{eq:augmented_framework}
Z^h_k = \big(U^h_k, X_{(k+1)h}\big) = \Big( m^h_{kh,(k+1)h}(X),\, X_{(k+1)h} \Big).
\end{equation}

Using that $(X_{t})_t$ is a Markov process, we can show that the distribution of $Z^h_{k+1}$ given its past only depends on $X_{kh}$, which ensures that $(Z^h_k)_k$ is Markovian, which leads to the following result.

\begin{lemma}[Markov property of the augmented chain]
\label{lem:Z-Markov}
For every $h>0$ the process $(Z^h_k)_{k\ge0}$ is a time-homogeneous Markov chain on $\R^{2d}$.
Moreover, its transition kernel can be written as
\[
\mathfrak{P}_h^b\big((u,x),A\big)
=
\mathbb{P}^b_x\Big(
\big(m^h_{0,h}(X), X_h\big)\in A
\Big),
\qquad (u,x)\in\R^{2d},\ A\in\mathcal{B}(\R^{2d}),
\]
that is, the next state depends on the past only through $X_{kh}$.
\end{lemma}

By Assumptions~\ref{assumption:boundedness} and~\ref{assumption:ergodicity}, the diffusion $X$ is non-degenerate and its 
transition density $p_t^b(x;y)$ is strictly positive and continuous on $(0,\infty)\times\R^d\times\R^d$; see, e.g.,
\cite{menozzi2020density}. This means that conditional on $X_0$ the random vector $(X_{h/p}, X_{2h/p}, \dots, X_{h})$ admits a density that is strictly positive and continuous on $\R^{dp}$. A direct change of variables then ensures that the Markov chain $(Z^h_k)_k$ admits a transition density $\mathfrak{p}_{p, h}^b(x;\cdot, \cdot)$ strictly positive and continuous.\\

We now state the main properties of the Markov chain $(Z^h_k)_{k\ge0}$. Recall the following definition.
\begin{definition}
A Markov chain with transition kernel $P$ on $(\mathsf{X},\mathcal{X})$ is said to be
{$\psi$-irreducible} if there exists a $\sigma$-finite measure $\psi$ on $(\mathsf{X},\mathcal{X})$ such that
for every $A\in\mathcal{X}$ with $\psi(A)>0$ and every $x\in\mathsf{X}$,
there exists $n\ge1$ with $P^n(x,A)>0$. 
It is {aperiodic} if it does not decompose into cyclic classes.
\end{definition}

We then have the following properties.

\begin{proposition}[Geometric ergodicity of the augmented chain]
\label{prop:Z-geom-ergodic}
Suppose $b\in\Sigma(\bs{\mathfrak{b}})$ and let $0<h\le h_0$. Then the Markov chain
$(Z^h_k)_{k\ge0}$ with transition kernel $\mathfrak{P}_h^b$ is $\psi$-irreducible, aperiodic and geometrically
ergodic. In particular:
\begin{itemize}
\item There exists a unique invariant probability measure $\overline{\pi}^{h,b}$ on $\R^{2d}$ such that
\[
\overline{\pi}^{h,b}\mathfrak{P}_h^b = \overline{\pi}^{h,b}.
\]
\item 
There exist constants $a>0$, $R<\infty$ and $c>0$, depending only on $\bs{\mathfrak b}$ and $h_0$, such that, with
\[
V(u,x)=V_X(x):=\exp\!\big(a(1+|x|^2)^{1/2}\big),
\]
one has, uniformly for all $b\in\Sigma(\bs{\mathfrak b})$, all $0<h\le h_0$, all $(u,x)\in\R^{2d}$ and all $k\ge0$,
\[
\big\|\big(\mathfrak P_h^b\big)^k\big((u,x),\cdot\big)-\overline\pi^{h,b}\big\|_{\mathrm{TV}}
\le R\,V_X(x)\,e^{-c h k}.
\]
Equivalently, the discrete-step geometric rate is $\rho_h=e^{-ch}$. Thus $\rho_h$ depends on $h$ and satisfies $\rho_h\uparrow1$ as $h\downarrow0$, while the physical-time constants $R$ and $c$ are uniform in $b$ and $h$.
\item The marginal of $\overline{\pi}^{h,b}$ on the first $d$ coordinates coincides with the invariant density
$m^h(\overline{\mu}^b)$ of the hidden process~\eqref{eq:hidden_process_framework} introduced in
Notation~\ref{notation:measures}.
\end{itemize}
\end{proposition}

\subsection{Invariant measure of the pre-averaged process}
\label{sec:hidden_measure}

We first summarize the notations used in the following.
\begin{notation}
\label{notation:measures}
	We use the following notation for invariant densities associated with the hidden and augmented processes:
	\begin{itemize}
		\item $m^{h}(\wb \mu^b)$ is the invariant density of the process~\eqref{eq:hidden_process_framework};
		\item $\wb\pi^{h,b}(\cdot, \,\cdot)$ is the invariant density of the Markov process~\eqref{eq:augmented_framework}.
	\end{itemize}
\end{notation}

\begin{proposition}
\label{proposition:expansion_hidden_measure}
Suppose that $b \in \Sigma(\bs{\mathfrak{b}})$. Recall from Assumption~\ref{assumption:isotropic} that $\wb \mu^b \in \mathcal{C}^{\beta}$ and $b \in \mathcal{C}^{\beta-1}$ for some $\beta \geq 2$ and set $k = \lfloor \beta / 2 \rfloor$. Then there exist deterministic functions $u_1, \dots, u_k$ depending on $b$ and on the averaging measure $m$, but not on $h\in(0,1)$, and functions $(r_h)_{0<h<1}$ locally bounded in $x$ uniformly in $h$ such that
\begin{equation*}
m^h(\wb\mu^b)(x) = \wb\mu^b(x) + \sum_{i=1}^{k} u_i(x) h^i + r_h(x) h^{\beta / 2}
\end{equation*}
for all $x\in\bbR^d$. More precisely, for every compact set $K\subset\mathbb R^d$, one has $\sup_{0<h<1}\sup_{x\in K}|r_h(x)|<\infty$; the bound may depend on $K$, $b$, and $m$, but not on $h$.
\end{proposition}

\subsection{Bounds on the transition probability}

In this section, we study the transition probabilities associated with the Markov chain~\eqref{eq:augmented_framework}. To streamline the analysis, we first introduce the following notations:
\begin{notation}
For any $t > 0$ and $x, y \in \mathbb{R}^d$:
\begin{itemize}
    \item $p^b_t(x; \cdot)$ denotes the transition density of $X$, i.e., the density of $X_t$ conditionally on $X_0 = x$.
    \item $\mathfrak{p}^b_{h,t}(x; \cdot)$ denotes the density of $m^{h}_{t,t+h}(X)$ conditionally on $X_0 = x$, with $\mathfrak{p}^b_{h}(x; \cdot) = \mathfrak{p}^b_{h,0}(x; \cdot)$.
    \item For $t > h$, $\mathfrak{q}_{h, t}(x; \cdot)$ denotes the density of $m_{t, t+h}^h(X)$ conditionally on $m_{0, h}^h(X) = x$ under $\mathbb{P}^{b}_{\wb \mu^b}$.
\end{itemize}
\end{notation}

Our results build on the findings of~\cite{menozzi2020density}. Let $g_{\lambda,d}(t, x)$ be defined as
\begin{equation*}
g_{\lambda,d}(t,x) = t^{-d/2} \exp(-\lambda |x|^2 / t),
\end{equation*}
and let $\theta_t(x)$ denote the deterministic forward flow defined by
\begin{equation}
\label{eq:theta}
\theta_0(x) = x, \quad \partial_t \theta_t(x) = b(\theta_t(x)), \quad t \geq 0.
\end{equation}

It is clear that the process $X$ satisfies the assumptions $(H^\sigma_\infty)$ and $(H^b_1)$ of~\cite{menozzi2020density} since $b$ is assumed to be of class $\mathcal{C}^1$ with bounded derivative and the diffusion coefficient of $X$ is constant in~\eqref{eq:diff:X}. 
Therefore, Theorem 1.2 of~\cite{menozzi2020density} applies and we have the following result.
\begin{proposition}\label{prop:bounds:p}
Suppose that $b \in \Sigma(\bs{\mathfrak{b}})$ and fix $T_0>0$. Then there exist constants $\lambda_i \in (0, 1]$ and $C_i > 0$ for $i \in \{0, 1, 2, 3\}$, depending only on $T_0$ and on $\bs{\mathfrak{b}}$, such that for any $0<t\le T_0$ and $x, y \in \mathbb{R}^d$, we have
\begin{align*}
p_t^b(x; y) &\leq C_0 g_{\lambda_0, d}(t, \theta_t(x) - y),\\
| \nabla_x p_t^b(x; y) | &\leq C_1 t^{-1/2} g_{\lambda_1, d}(t, \theta_t(x) - y), \\
| \nabla^2_x p_t^b(x; y) | &\leq C_2 t^{-1} g_{\lambda_2, d}(t, \theta_t(x) - y), \\
| \nabla_y p_t^b(x; y) | &\leq C_3 t^{-1/2} g_{\lambda_3, d}(t, \theta_t(x) - y).
\end{align*}
\end{proposition}

We now extend these bounds to $\mathfrak{p}$. 

\begin{proposition}\label{prop:bounds:frakp_p}
Suppose that $b \in \Sigma(\bs{\mathfrak{b}})$, suppose in addition that $m$ is given by~\eqref{eq:measure_preav}, and fix $T_0>0$. Then for every $c_0 > 0$, there exist constants $\wt c_0 >0$, $\wt \lambda_i \in (0, 1]$ and $\wt C_i > 0$ for $i \in \{0, 1\}$, depending only on $c_0,T_0$ and $\bs{\mathfrak{b}}$, such that if $ph \leq \wt c_0$, then for any $0<h\le t\le T_0$ and $x, y \in \mathbb{R}^d$, we have
    \begin{align*}
    \mathfrak{p}_{h,t}(x; y) &\leq \widetilde{C}_0 g_{\widetilde{\lambda}_0, d}(t, \theta_t(x) - y) \exp(c_0 |y|^2),
    \\
    | \nabla_x \mathfrak{p}_{h,t}(x; y) | &\leq \widetilde{C}_1 t^{-1/2} g_{\widetilde{\lambda}_1, d}(t, \theta_t(x) - y) \exp(c_0 |y|^2).
    \end{align*}
\end{proposition}

\begin{remark}[Heuristic bounds on the transition density of the augmented process]
Note that $U_t = m_{t}^h(X)$ satisfies
\begin{equation*}
\mathrm{d}U_t = \chi(t/h)b(X_t) \, \mathrm{d}t + \chi(t/h) \, \mathrm{d}W_t.
\end{equation*}
Thus, $(m_{t}^h(X), X_t)$ can be viewed as the solution $Z_t = (U_t, X_t)$ of the stochastic differential equation:
\begin{equation*}
\mathrm{d}Z_t = \widetilde{b}(t, Z_t) \, \mathrm{d}t + \widetilde{\sigma}(t) \, \mathrm{d}W_t,
\end{equation*}
where
\begin{equation*}
\begin{aligned}
\widetilde{b}(t, (u, x)) &= (\chi(t/h)b(x), b(x)), \qquad
\widetilde{\sigma}(t) &= (\chi(t/h), 1).
\end{aligned}
\end{equation*}
The deterministic flow $\psi_t(u, x)$ associated with $\widetilde{b}(t, (u, x))$ satisfies $\psi_t(u, x) = (\vartheta_t(u, x), \theta_t(x))$, where $\theta_t(x)$ is defined in~\eqref{eq:theta}, and $\vartheta_t(u, x)$ satisfies
\begin{equation*}
\vartheta_0(u, x) = u, \quad \partial_t \vartheta_t(u, x) = \chi(t/h)b(\theta_t(x)).
\end{equation*}
Integrating this equation gives
\begin{equation*}
\vartheta_t(u, x) = u + \int_0^t \chi(s/h)b(\theta_s(x)) \, \mathrm{d}s.
\end{equation*}
When $Z_t = (m_{t}^h(X), X_t)$, we see that the initial condition $(u,x)$ must satisfy $u=x$. We write $\vartheta_t(x) = \vartheta_t(x, x)$ for the flow for this case. Therefore, it is natural to conjecture that the transition density $\mathfrak{p}^b_t(x;y,z)$ of the augmented process 
$(m^h_t(X),X_t)$ satisfies Gaussian-type bounds analogous to those of Proposition~\ref{prop:bounds:p}. 
More precisely, one expects the existence of constants $\wt \lambda_i \in (0,1]$ and $\wt C_i > 0$, 
$i \in \{0,1,2,3\}$, such that for all $t>0$ and $x,y,z \in \R^d$,
\begin{align*}
\mathfrak{p}^b_t(x; y, z) 
&\leq \widetilde{C}_0 \, g_{\widetilde{\lambda}_0, 2d}\big(t, (\vartheta_t(x) - y, \theta_t(x) - z)\big),\\
\big| \nabla_x \mathfrak{p}^b_t(x; y, z) \big| 
&\leq \widetilde{C}_1 \, t^{-1/2} \, g_{\widetilde{\lambda}_1, 2d}\big(t, (\vartheta_t(x) - y, \theta_t(x) - z)\big),
\end{align*}
where $\theta_t$ and $\vartheta_t$ are the deterministic flows defined above. However, we do not have a rigorous proof of these inequalities. In particular, 
Theorem~1.2 of~\cite{menozzi2020density}, which we used to obtain Proposition~\ref{prop:bounds:p}, 
cannot be applied directly to the augmented process $Z_t = (m^h_t(X),X_t)$. There are two main obstacles:
\begin{itemize}
    \item The diffusion matrix $\widetilde{\sigma}(t)$ associated with $Z$ is degenerate: the coefficient 
    $\chi(t/h)$ vanishes near $t=h$, so the non-degeneracy condition $(H^\sigma_\alpha)$ in~\cite{menozzi2020density} fails.
    \item The process $Z$ is $2d$-dimensional but is driven by a $d$-dimensional Brownian motion, 
    which again falls outside the uniformly elliptic framework of~\cite{menozzi2020density}.
\end{itemize}
For specific choices of the averaging measure $m$ (such as the discrete pre-averaging measure 
in~\eqref{eq:measure_preav}) and under suitable quantitative non-degeneracy assumptions on $m$, 
it is plausible that an extension of the arguments of~\cite{menozzi2020density} would yield bounds of the above form in a hypoelliptic setting. Developing such an extension would require substantial additional work and lies beyond the scope of the present paper. For this reason, we do not use the conjectured bounds on $\mathfrak{p}^b_t$ in our proofs, and
the rigorous Gaussian estimates below use Proposition~\ref{prop:bounds:p} only on bounded short-time intervals; when a large-time estimate is needed, we use the invariant density and the explicit $L^2$ ergodicity bound~\eqref{eq:ergodicity} instead of a Gaussian upper bound.
\end{remark}

Proposition~\ref{prop:bounds:frakp_p} has a series of useful corollaries that are given below and proved in Appendix~\ref{appendix:proof:corollaries}.

\begin{corollary}
\label{corrolary:bounds:1}
Suppose that $b \in \Sigma(\bs{\mathfrak{b}})$ and suppose in addition that $m$ is given by~\eqref{eq:measure_preav}. Then for every $c_0 > 0$, there exist constants $\wt c_0 >0$, 
$\wt \lambda_i \in (0, 1]$ and $\wt C_i > 0$ for $i \in \{0, 1\}$ and $t_0^*$ such that 
if $ph \leq \wt c_0$, then for any $h\le t \leq t_0^*$ and $x, y \in \mathbb{R}^d$,
we have
\begin{align*}
    \mathfrak{p}_{h,t}(x; y) 
    &\leq \widetilde{C}_0 \, g_{\widetilde{\lambda}_0, d}(t, x - y) 
    \exp(c_0 |x|^2), \\
    | \nabla_x \mathfrak{p}_{h,t}(x; y) | 
    &\leq \widetilde{C}_1 \, t^{-1/2} \, g_{\widetilde{\lambda}_1, d}(t, x - y) 
    \exp(c_0 |x|^2).
\end{align*}
\end{corollary}

\begin{corollary}
\label{corrolary:bounds:2}
Suppose that $b \in \Sigma(\bs{\mathfrak{b}})$ and suppose in addition that $m$ is given by~\eqref{eq:measure_preav}. Then for every $c_0 > 0$, there exist constants $\wt c_0 >0$, 
$\wt \lambda_i \in (0, 1]$ and $\wt C_i > 0$ for $i \in \{0, 1\}$ and $t_0^*$ such that 
if $ph \leq \wt c_0$, then for any $2h\le t \leq t_0^*$ and $x, y \in \mathbb{R}^d$,
we have
\begin{equation*}
    \mathfrak{q}_{h, t + h}(x; y)
    \leq \widetilde{C}_0 \, g_{\widetilde{\lambda}_0, d}(t, x - y) 
    \frac{\exp\!\big(c_0|x|^2\big)}{m^{h}(\wb \mu^{b})(x)}.
\end{equation*}
\end{corollary}

\begin{corollary}
\label{corrolary:bounds:3}
Suppose that $b \in \Sigma(\bs{\mathfrak{b}})$ and suppose in addition that $m$ is given by~\eqref{eq:measure_preav}. Then for every $c_0 > 0$ and every $t_0^* > 0$, there exist constants $\wt c_0>0$ and $\widetilde{C}_0$ such that, if $ph\le\wt c_0$, then for all $t \geq t_0^*$, we have
\begin{equation*}
    \mathfrak{q}_{h, t + h}(x; y)
    \leq \widetilde{C}_0 \exp\!\big(c_0|y|^2\big).
\end{equation*}
\end{corollary}

This implies in particular the following on the invariant distribution of $m_{0,h}^h(X)$.

\begin{corollary}
    \label{corrolary:bounds:4}
Suppose that $b \in \Sigma(\bs{\mathfrak{b}})$ and suppose in addition that $m$ is given by~\eqref{eq:measure_preav}. Then for every $c_0 > 0$, there exist constants $\wt c_0>0$ and $\widetilde{C}_0$ such that, if $ph\le\wt c_0$, we have
\begin{equation*}
    m^{h}(\wb \mu^b)(y)
    \leq \widetilde{C}_0 \exp\!\big(c_0|y|^2\big).
\end{equation*}
\end{corollary}

\section{Estimation of the invariant measure}
\label{sec:estimation}

Throughout this section, we assume that $b \in \Sigma(\bs{\mathfrak{b}})$, with $\bs{\alpha}$ such that $\alpha_1 \leq \dots \leq \alpha_d$.

\subsection{Construction of the estimator}
\label{sec:construction}

In this section, we construct an estimator of the invariant density $\overline{\mu}^b$ of $X$ in the statistical model described in Section~\ref{sec:model}. To this end, we adopt a kernel-type estimation procedure. \\

Let $K : \mathbb{R} \to \mathbb{R}$ be a bounded kernel, that is, a bounded measurable function satisfying
\[
\int_{\mathbb{R}} K(y) \, \mathrm{d}y = 1.
\]
We assume that $K$ is of order $l \geq 1$, i.e.\ for all $1 \leq k \leq l-1$,
\begin{equation}
\label{eq:def:K}
\int_{\mathbb{R}} y^k K(y) \, \mathrm{d}y = 0 .
\end{equation}
Moreover, we also assume that $K$ is sub-Gaussian, in the sense that there exists $c_K, C_K > 0$ such that
\begin{equation*}
    |K(x)| \leq C_K \exp(-c_K |x|^2),
\end{equation*}
for all $x \in \mathbb{R}$. For any bandwidth vector $\bs{h} = (h_1, \dots, h_d) \in (0, +\infty)^d$ and $y = (y_1, \dots, y_d) \in \mathbb{R}^d$, we define the product kernel
\begin{equation*}
\bs{K}_{\bs{h}}(y) = \prod_{i=1}^d h_i^{-1} K(y_i h_i^{-1}).
\end{equation*}
When the target density $\overline{\mu}^b$ belongs to the anisotropic Hölder class $\mathcal{H}_d(\bs{\alpha}, \bs{\mathcal{L}})$, we always assume that $l \geq \lceil \alpha_d \rceil$.\medskip

A first natural idea is to use the standard kernel estimator based directly on the noisy observations $(Y_{k,n})_{0\le k\le n-1}$. Since $(Y_{k,n})$ is stationary, the estimator
\begin{equation}
\label{eq:naiveestimator}
\widehat{\mu}_{KB, n,\bs{h},p}(x) 
= \frac{1}{n } \sum_{k=0}^{n - 1} \bs{K}_{\bs{h}}( x-Y_{k,n})
\end{equation}
targets the density of $Y_{1,n}$, which is given by
\[
\overline{\mu}^b * \varphi_{\tau_n} (x) 
= \int_{\mathbb{R}^d} 
\overline{\mu}^b (x-y) \varphi_{\tau_n} (y)\,\mathrm{d}y
= \int_{\mathbb{R}^d} 
\overline{\mu}^b (x-\tau_n y) \varphi_{1} (\tau_n y)\,\mathrm{d}y,
\]
where $\varphi_{\tau}(x) = (2\pi \tau^2)^{-d/2} \exp(-|x|^2/(2\tau^2))$ is the density of a centred Gaussian with variance $\tau^2 I_d$. Thus, the presence of noise induces an additional bias
\begin{equation*}
\big|
    \overline{\mu}^b * \varphi_{\tau_n} (x)
    - \overline{\mu}^b(x)
\big|
\leq 
    \int_{\mathbb{R}^d} 
    \big|\overline{\mu}^b (x-\tau_n y) - \overline{\mu}^b(x)\big| \varphi_{1} (\tau_n y)
    \,\mathrm{d}y,
\end{equation*}
which is of order $\tau_n$. When $\tau_n$ is not small, this noise-induced bias dominates the usual kernel bias. In particular, if $\tau_n$ is of order $1$, the estimator in~\eqref{eq:naiveestimator} is not consistent for $\overline{\mu}^b$. This motivates a noise-reduction step.
\medskip

We implement a preaveraging approach by computing local averages of the observations over blocks of size $p$, that is, we consider the observations
\begin{equation*}
	\Big( \frac{1}{p} \sum_{\ell=0}^{p-1} Y_{kp + \ell,n} \Big)_{1\leq k \leq \lfloor n/p\rfloor}.
\end{equation*}
This scheme fits into the hidden process framework introduced in Section~\ref{sec:hidden}. Let $m_p$ denote the probability measure
\begin{equation*}
m_p = \frac{1}{p} \sum_{k=0}^{p-1} \delta_{k/p},
\end{equation*}
where $\delta_u$ stands for the Dirac measure at $u$. Then
\begin{equation*}
\frac{1}{p} \sum_{\ell=0}^{p-1} Y_{kp + \ell,n} = \frac{1}{p} \sum_{\ell=0}^{p-1} X_{(kp + \ell)\Delta_n} + 
    \frac{\tau_n}{p}\sum_{\ell=0}^{p-1}\xi_{kp + \ell,n} = 
	m^{p\Delta_n}_{p, kp\Delta_n, (k+1)p\Delta_n} (X) + 
    \frac{\tau_n}{p}\sum_{\ell=0}^{p-1}\xi_{kp + \ell,n}.
\end{equation*}
As shown in Section~\ref{sec:hidden_measure}, the preaveraged process $m^{p\Delta_n}_{p, kp\Delta_n, (k+1)p\Delta_n} (X)$ has an invariant density given by $m^{p\Delta_n}_p(\overline{\mu}^b)$ which is close to $\overline{\mu}^b$; see in particular Proposition~\ref{proposition:expansion_hidden_measure}. Therefore, preaveraging distorts the invariant measure only mildly, while it significantly reduces the effect of the noise.

Since $(\xi_{i,n})_{i,n}$ are i.i.d.\ standard Gaussian, we can write
\[
\frac{1}{p} \sum_{\ell=0}^{p-1}\xi_{kp + \ell,n} 
= \frac{1}{\sqrt{p}} \widetilde \xi_{k,n},
\]
where each $\widetilde \xi_{k,n}$ is again standard Gaussian. This leads to the following natural estimator of the invariant density of the preaveraged process:
\begin{equation}
\label{eq:preliminary}
\widehat{\nu}_{n,\bs{h}, p}(x) 
= \frac{1}{\lfloor n/p \rfloor} \sum_{k=0}^{\lfloor n/p \rfloor-1} 
\bs{K}_{\bs{h}}\Big( x-\frac{1}{p}\sum_{\ell=0}^{p-1}Y_{kp + \ell,n} \Big).
\end{equation}

It is important to emphasize that $\widehat{\nu}_{n,\bs{h}, p}(x)$ estimates the invariant density of
\[
\frac{1}{p}\sum_{\ell=0}^{p-1}Y_{kp + \ell,n} 
= m^{p\Delta_n}_{p, kp\Delta_n, (k+1)p\Delta_n} (X) 
+ \frac{\tau_n}{\sqrt{p}}\widetilde \xi_{k,n},
\]
rather than directly that of $X$. Denoting this invariant density by $\overline{\mu}^b_{p, n}$, and using the independence between $(\widetilde \xi_{k,n})$ and $X$, we obtain
\[
\overline{\mu}^b_{p, n} 
= m^{p\Delta_n}_p(\overline{\mu}^b) * \varphi_{\tau_n / p^{1/2}},
\]
where $*$ denotes convolution. In general, the difference $\overline{\mu}^b_{p, n} - \overline{\mu}^b$ may be non-negligible. Furthermore,
\begin{equation}
\label{eq:expectation:nu}
    \mathbb{E}_{\overline{\mu}^b}^b [\widehat{\nu}_{n,\bs{h},p}(x)] 
    =
    \int_{\mathbb{R}^d}
    \overline{\mu}^b_{p, n}(z)
    \bs{K}_{\bs{h}}(x-z)\,\mathrm{d}z
    =
    \big(m^{p\Delta_n}_p(\overline{\mu}^b) * \varphi_{\tau_n / p^{1/2}} * \bs{K}_{\bs{h}}\big)(x).
\end{equation}
Thus, the bias of $\widehat{\nu}_{n,\bs{h},p}(x)$ decomposes into three contributions:
\begin{itemize}
    \item[(i)] \emph{Preaveraging bias:} we estimate $m^{p\Delta_n}_p(\overline{\mu}^b)$ instead of $\overline{\mu}^b$, and the difference $m^{p\Delta_n}_p(\overline{\mu}^b) - \overline{\mu}^b$ is of order $p\Delta_n$ by Proposition~\ref{proposition:expansion_hidden_measure}.
    \item[(ii)] \emph{Noise bias:} even after preaveraging, $\overline{\mu}^b$ is convolved with $\varphi_{\tau_n / p^{1/2}}$, which induces an error of order $\tau_n / p^{1/2}$.
    \item[(iii)] \emph{Kernel bias:} the standard smoothing bias due to convolution with $\bs{K}_{\bs{h}}$, which in our anisotropic H\"older setting is of order $\sum_{i=1}^d h_i^{\alpha_i}$.
\end{itemize}

While the kernel bias (iii) cannot be improved beyond the usual minimax rate for density estimation, we can significantly reduce the preaveraging bias (i) and the noise bias (ii). For (i), we exploit the expansion of Proposition~\ref{proposition:expansion_hidden_measure} to construct a Richardson extrapolation, improving the order $p\Delta_n$ to $(p\Delta_n)^{\beta / 2}$. For (ii), an appropriate deconvolution procedure based on the H\"older regularity of $\overline{\mu}^b$ improves the order $\tau_n / p^{1/2}$ to $\tau_n^{\alpha_1} / p^{\alpha_1/2}$.
\medskip

\textbf{Richardson extrapolation.} We first address the preaveraging bias term involving $m^{p\Delta_n}_p(\overline{\mu}^b)(x)$. Writing $h = p\Delta_n$, Proposition~\ref{proposition:expansion_hidden_measure} yields
\[
m^{p\Delta_n}_p(\overline{\mu}^b)(x) 
= \overline{\mu}^b(x) + \sum_{i=1}^{\lfloor \beta / 2 \rfloor} u_{i,p}(x) (p\Delta_n)^i + O((p\Delta_n)^{\beta / 2}),
\]
where the coefficients $u_{i,p}(x)$ do not depend on $\Delta_n$, but can depend on $p$ and $\bs{K}_{\bs{h}}$. The classical Richardson extrapolation method can be summarized as follows: suppose we wish to estimate a quantity $\mu$ and we have approximations $\widehat{\mu}_k(h)$ such that
\[
\widehat{\mu}_k(h) 
= \mu + \sum_{i = k}^N u_i h^i + O(h^\beta),
\]
with $k$ known, $u_i$ independent of $h$ and $\beta$ unknown. Then
\[
\widehat{\mu}_{k+1}(h) 
= \frac{2^{k} \widehat{\mu}_k(h) - \widehat{\mu}_k(2h) }{2^k - 1} 
= \mu + \sum_{i = k+1}^N u_i' h^{i} + O(h^{\beta}),
\]
for some new coefficients $u_i'$. Iterating this construction up to $k = N$, we eventually obtain an estimator with bias of order $O(h^\beta)$. To apply this to $m^{h}_p(\overline{\mu}^b)(x)$, let $a_0 = (1)$ and, for $k \geq 1$, define the $(k+1)$-dimensional vector $a_k = (a_{k,i})_{0\le i\le k}$ recursively by
\begin{align*}
a_{k,k} &= \frac{-1}{2^{k} - 1} a_{k-1, k - 1},\\
a_{k,i} &= \frac{1}{2^{k} - 1} \big( 2^k a_{k-1, i} - a_{k-1, i-1} \big), \qquad 1 \le i \le k-1,\\
a_{k,0} &= \frac{2^k}{2^{k} - 1} a_{k-1, 0}.
\end{align*}
Then, for any integer $l \le \lfloor \beta/2 \rfloor - 1$,
\begin{equation}
\label{eq:richardson}
\sum_{i=0}^{l}
a_{l,i}
\,m^{2^i h}_p(\overline{\mu}^b)(x)
=
\overline{\mu}^b(x) + O(h^{\beta / 2}).
\end{equation}

We now transfer this extrapolation scheme to the estimators. For $k \geq 0$, denote by
\[
\widehat{\nu}_{n,\bs{h},p}^{(k)}(x)
\]
the estimator $\widehat{\nu}_{n,\bs{h},p}(x)$ defined in~\eqref{eq:preliminary}, but constructed from the sub-sampled data $(Y_{ik, n})_{i \leq \lfloor n/k \rfloor}$ instead of $(Y_{i,n})_{i \leq n}$. We then define the Richardson-corrected estimator
\[
\widehat{\nu}_{n,\bs{h},p}^{R}(x)
=
\sum_{i=0}^{l}
a_{l,i}
\,\widehat{\nu}_{n,\bs{h},p}^{(2^i)}(x).
\]
For $p=1$, we use the convention $\widehat{\nu}^{R}_{n,\bs{h},1}=\widehat{\nu}_{n,\bs{h},1}$. Combining~\eqref{eq:expectation:nu} and~\eqref{eq:richardson} yields the following bias bound.

\begin{lemma}
\label{lemma:bias:Richardson}
Suppose that $b \in \Sigma(\bs{\mathfrak{b}})$. Then there exists a constant $\Cbias^{(R)}$, depending only on $\bs{\mathfrak{b}}$, such that for all $x \in \mathbb{R}^d$, all $p \geq 2$ with $p\Delta_n \leq 1$ and all $\bs{h} \in (0, \infty)^d$, we have
\begin{equation*}
\Big|\mathbb{E}_{\overline{\mu}^b}^b [\widehat{\nu}^{R}_{n,\bs{h},p}(x)] - \overline{\mu}^b * \varphi_{\tau_n / p^{1/2}} * \bs{K}_{\bs{h}}(x)\Big|
\leq \Cbias^{(R)}
(p\Delta_n)^{\beta/2}.
\end{equation*}
\end{lemma}

\textbf{Deconvolution of the noise.} We now handle the bias created by convolution with the Gaussian kernel $\varphi_{\tau_n/p^{1/2}}$. Since $\tau_n$ is bounded, the convolved density $\overline{\mu}^b * \varphi_{\tau_n / p^{1/2}}$ belongs to the same anisotropic Hölder class $\mathcal{H}_d(\bs{\alpha},\bs{\mathcal{L}})$ as $\overline{\mu}^b$, with the same parameters. Standard kernel arguments, detailed in Appendix~\ref{sec:proof:eq:bound:bias:kernel}, then give
\begin{equation}
\label{eq:bound:bias:kernel}
    \Big| \overline{\mu}^b * \varphi_{\tau_n / p^{1/2}} * \bs{K}_{\bs{h}}(x) - \overline{\mu}^b * \varphi_{\tau_n / p^{1/2}}(x) \Big| \leq \Cbias^{(K)} \sum_{i=1}^{d} h_i^{\alpha_i},
\end{equation}
for some constant $\Cbias^{(K)} > 0$ depending only on $\bs{\mathfrak{b}}$. To reduce the noise bias itself, we exploit the smoothness of $\overline{\mu}^b$. We can write
\[
\overline{\mu}^b * \varphi_{\tau_n / p^{1/2}}(x)
=
\mathbb{E}\big[\overline{\mu}^b(x-\widetilde{\tau}_n p^{-1/2}\zeta)\big],
\]
where $\zeta$ is a standard $d$-dimensional Gaussian vector and $\widetilde{\tau}_n$ is a suitable scalar multiple of $\tau_n$. Using the H\"older regularity of $\overline{\mu}^b$, we perform a Taylor expansion of $\overline{\mu}^b(x + \widetilde{\tau}_n p^{-1/2}(\bs{\gamma}-\zeta))$ around $x$, for $\bs{\gamma} = (\gamma_1,\dots,\gamma_d)\in\{0,\dots,l\}^d$. By explicitly computing the moments of $(\bs{\gamma}-\zeta)$, we obtain an expansion of $m^{p\Delta_n}_p(\overline{\mu}^b) * \varphi_{\widetilde \tau_{n,p}}(x+\tau_n\bs{\gamma}/p^{1/2})$ in terms of $m^{p\Delta_n}_p(\overline{\mu}^b)(x)$ and higher-order terms. This allows us to isolate $m^{p\Delta_n}_p(\overline{\mu}^b)(x)$ and to construct a deconvolution-based correction. \\

More precisely, introduce the $(l+1)\times(l+1)$ matrix $A = (a_{k,i})_{0 \leq k,i \leq l}$ with entries
\begin{equation}
\label{eq:def:a}
a_{k,i} = \sum_{j=0}^k {k \choose j} (-1)^j m_j i^{k-j},
\end{equation}
where $m_j$ denotes the $j$-th moment of a standard Gaussian variable. As shown in Appendix~\ref{sec:appendix:A}, the matrix $A$ is invertible. We denote its inverse by $A^{-1}$ and its first column by $\bs{u} = (u_0, \dots, u_l)$. Then, for all $k \in \{ 0, \dots, l\}$,
\begin{align}
\label{eq:def:u}
\sum_{i=0}^l 
u_i \Big( \sum_{j=0}^k \frac{ (-1)^j m_j i^{k-j}}{j! (k-j)!} \Big)
=
\begin{cases}
1 & \text{ if } k=0,\\
0 & \text{ otherwise.}
\end{cases}
\end{align}
For any multi-index $\bs{\gamma} = (\gamma_1, \dots, \gamma_d) \in \{0, \dots, l\}^d$, set $\bs{u}_{\bs{\gamma}} = \prod_{i=1}^d u_{\gamma_i}$. We then define the final pointwise estimator by
\begin{equation}
\label{eq:def:est:mu}
	\widehat{\mu}_{n,\bs{h},p}(x) 
	= \sum_{\bs{\gamma} \in \{0,\dots,l\}^d} 
	\bs{u}_{\bs{\gamma}}\, \widehat{\nu}_{n,\bs{h},p}^{R}\big(x+\bs{\gamma}\widetilde{\tau}_{n} p^{-1/2}\big).
\end{equation}
The bias of $\widehat{\mu}_{n,\bs{h},p}(x)$ is stated in the following proposition, proved in Section~\ref{sec:propo:bias:mu}.

\begin{proposition}
\label{propo:small:bias:mu}
Suppose that $b \in \Sigma(\bs{\mathfrak{b}})$, with $\bs{\alpha}$ such that $\alpha_1 \leq \dots \leq \alpha_d$. Then there exists a constant $\Ctbias$ depending only on $\bs{\mathfrak{b}}$ such that for any $x \in \mathbb{R}^d$, any $p \geq 1$ with $p\Delta_n \leq 1$ and any $\bs{h} \in (0, \infty)^d$, we have
\begin{align*}
\big|\mathbb{E}_{\overline{\mu}^b}^b [\widehat{\mu}_{n,\bs{h},p}(x)] - \overline{\mu}^b(x)\big|
\leq \Ctbias
\begin{cases}
\tau_{n}^{\alpha_1} +  \displaystyle\sum_{i=1}^d h_i^{\alpha_i} & \text{ if } p = 1,\\[0.4em]
\dfrac{\tau_{n}^{\alpha_1}}{p^{\alpha_1/2}} + \displaystyle\sum_{i=1}^d h_i^{\alpha_i} + (p\Delta_n)^{\beta / 2} & \text{ if } p\geq 2.
\end{cases}
\end{align*}
\end{proposition}

\subsection{Properties of the estimator}

As identified in previous work, variance bounds for kernel estimators of invariant densities of diffusions depend critically on both the dimension and the regularity of the target density, and thus differ on the sets $D_1$, $D_2$ and $D_3$ defined in~\eqref{eq:D}. In particular, the proofs rely on sharp bounds for the transition probabilities of the underlying process; see, for instance,~\cite{amorino2023estimation}. In our setting, the presence of additive noise and the use of preaveraging require new bounds on the transition densities of the preaveraged process, which were derived in Section~\ref{sec:hidden}. 

\begin{proposition}\label{prop:variance}
	Suppose that $b \in \Sigma(\bs{\mathfrak{b}})$. If $\hat{\mu}_{n, \bs{h},p}$ is the estimator proposed in~\eqref{eq:def:est:mu}, then for all $x \in \mathbb{R}^d$, there exist $\Cvar>0$, $n_0>0$ and $c_\star>0$ such that, for all $n \geq n_0$, all $\bs{h} \in (0,1]^d$, and all $p$ satisfying the short-window condition $p^2\Delta_n\le c_\star$, we have
    \begin{equation*}
        \mathrm{V}^b_{n,\bs{h},p}(x) \leq \Cvar \frac{p\Delta_n}{T_n} v^2(\alpha, n, \bs{h},p),
    \end{equation*}
    where
    \begin{equation*}
    v^2(\alpha, n, \bs{h},p) =
    \begin{cases}
    h_1^{-1} + \frac{|\log(h_1)|}{ p\Delta_n}
    & \, \text{ if } d = 1,
    \\ \\
    ({h_1h_2})^{-1} + \frac{|\log(p\Delta_n)|}{p\Delta_n} + \frac{|\log(h_1h_2)|}{p\Delta_n}
    & \, \text{ if } d = 2,
    \\ \\
        {\prod_{i=1}^dh_i^{-1}} 
        + \frac{ \sum_{i=1}^d |\log(h_i)| }{p\Delta_n\prod_{i=3}^d h_i} 
        & \, \text{ if } d \geq 3 \text{ and } (\bs{\alpha}, k_0 ) \in D_1,
    \\ \\
     {\prod_{i=1}^dh_i^{-1}}
     +
     \frac{ \prod_{i=1}^{k_0} (h_i)^{2/k_0}  }{p\Delta_n \prod_{i=1}^d h_i}  
     + \frac{\sum_{i=1}^d | \log h_i |}{p\Delta_n} & \, \text{ if } d \geq 3 \text{ and } (\bs{\alpha}, k_0 ) \in D_2,
\\ \\
     {\prod_{i=1}^d h_i^{-1}}
     +
    \frac{\sum_{i=1}^d | \log h_i | }{p\Delta_n}
    +
    \frac{1}{p\Delta_n (h_2 h_3)^{1/2} \prod_{i=4}^d h_i }& \, \text{ if } d \geq 3 \text{ and } (\bs{\alpha}, k_0 ) \in D_3.
    \end{cases}
\end{equation*}
and where the sets $D_1$, $D_2$ and $D_3$ are defined in~\eqref{eq:D}. 
\end{proposition}

The proof of Proposition~\ref{prop:variance} is postponed to Appendix~\ref{sec:appendix:prop:variance}. Our bounds match those obtained in Proposition~1 of~\cite{amorino2023estimation}, up to the presence of an extra factor $p$ in front of $\Delta_n (T_n\prod_{i=1}^d h_i)^{-1}$. This is natural: the term arises from the discretisation of the process, and preaveraging effectively subsamples the data by grouping observations into blocks of length $p$, thereby replacing $\Delta_n$ with $p\Delta_n$. As in~\cite{amorino2023estimation}, this term only becomes relevant when $\Delta_n$ is relatively large. Here, the break-even point also depends on the noise level; see Section~\ref{sec:upper_bounds} for a detailed discussion. Note also that, although Proposition~\ref{prop:variance} is stated for the final estimator $\widehat{\mu}_{n,\bs{h},p}(x)$, the same bound holds for the intermediate estimators $\widehat{\nu}_{n,\bs{h},p}(x)$ and $\widehat{\nu}_{n,\bs{h},p}^{(R)}(x)$.\\

We finally derive a Bernstein-type concentration inequality for the estimator constructed in Section~\ref{sec:estimation}. This inequality provides a key tool for adaptive, data-driven selection of the tuning parameters.

\begin{theorem}[Bernstein inequality]\label{th:bernstein:loc}
Set $h_p=p\Delta_n$ and $n_p=\lfloor n/p\rfloor$. Suppose $0<h_p\le h_0$, $n_p\ge2$ and $p^2\Delta_n\le c_\star$, where $c_\star$ is the constant in Proposition~\ref{prop:variance}. Then there exist constants $K_B<\infty$ and $C_B<\infty$, depending only on $\bs{\mathfrak b}$ and $h_0$, such that, for all $b\in\Sigma(\bs{\mathfrak b})$, all $x\in\R^d$, all $\bs h\in(0,1]^d$, and all $\varepsilon>0$,
\begin{equation}
\label{eq:bernstein:nu:explicit}
\mathbb{P}_{\overline{\mu}^b}^b\Big(
\big| \widehat{\nu}_{n,\bs{h},p}(x)
- \mathbb{E}^b_{\overline\mu^b}[\widehat{\nu}_{n,\bs{h},p}(x)] \big|
> \varepsilon\Big)
\le
K_B
\exp \Big(
-\frac{n_p^2\varepsilon^2}
{32 n_p \mathbb{V}_{\nu}(\alpha, n, \bs{h},p;x)
+C_Bh_p^{-2}\varepsilon\|\bs{K}_{\bs{h}}\|_{\infty} n_p\log n_p}
\Big),
\end{equation}
where
\begin{align*}
\mathbb{V}_{\nu}(\alpha, n, \bs{h},p;x)
=& \operatorname{Var}^b_{\overline\mu^b}\Big(\bs{K}_{\bs{h}}\big(x - p^{-1} \sum_{\ell=0}^{p-1} Y_{\ell,n}\big)\Big)\\
& +2 \sum_{k=1}^{\infty} \operatorname{Cov}^b_{\overline\mu^b}\Big(\bs{K}_{\bs{h}}\big(x - p^{-1} \sum_{\ell=0}^{p-1} Y_{\ell,n}\big), \bs{K}_{\bs{h}}\big(x - p^{-1} \sum_{\ell=0}^{p-1} Y_{kp + \ell,n}\big)\Big).
\end{align*}
Moreover,
\[
\mathbb{V}_{\nu}(\alpha, n, \bs{h},p;x) \leq C_V v^2(\alpha, n, \bs{h},p),
\]
with $C_V$ depending only on $\bs{\mathfrak b}$.
\end{theorem}

The proof of Theorem~\ref{th:bernstein:loc} adapts the results of \cite[Theorem~4.3 and the following remark]{lemanczyk2021general}. In particular, it relies on an exponential rate of convergence to equilibrium for the Markov chain associated with the preaveraged observations $(p^{-1}\sum_{\ell=0}^{p-1} Y_{kp+\ell,n})_k$, established using the quantitative ergodic properties derived in Section~\ref{sec:hidden}. Details are given in Section~\ref{sec:proof:bernstein}.

\section{Upper bounds and hyper-parameter choice}
\label{sec:upper_bounds}

In this section, we quantify the quality of this estimation procedure by deriving an upper bound for the quadratic error 
\begin{equation*}
	\mathcal{R}(\widehat\mu, b; x) = \mathbb{E}^{b}_{\overline{\mu}^b}[|\widehat\mu(x) -  \overline{\mu}^b(x)|^2].
\end{equation*}	
Using the classical Bias--Variance decomposition of the quadratic error
\begin{equation}\label{eq:biasvariancedecompo}
    \mathcal{R}(\widehat\mu_{n,\bs{h},p}(x), b; \, x)  = \mathrm{B}_{n,\bs{h},p}^b(x)^2 + \mathrm{V}_{n,\bs{h},p}^b(x),
\end{equation}
where for all $x \in \mathbb{R}^d$, we have
\begin{equation*}
	\mathrm{B}^b_{n,\bs{h},p}(x) = \big|\mathbb{E}_{\overline{\mu}^b}^b [\widehat\mu_{n,\bs{h},p}(x)] - \overline{\mu}^b(x)\big|\,\,
	\text{ and } \,\, \mathrm{V}^b_{n,\bs{h},p}(x) = \Var_{\overline{\mu}^b}^b [\widehat\mu_{n,\bs{h},p}(x)].
\end{equation*}
Upper bounds for $\mathrm{B}^b_{n,\bs{h},p}(x)$ and $\mathrm{V}^b_{n,\bs{h},p}(x)$ are given in Propositions~\ref{propo:small:bias:mu} and~\ref{prop:variance}, respectively.

\subsection{Upper bounds for the estimation of the invariant measure}

The variance obtained in Proposition~\ref{prop:variance} varies across different cases, and the selection of the hyper-parameters $p$ and $\bs{h}$ is naturally influenced by these variations. In this section, we detail each case and the convergence rate that can be obtained using the best choices $p^*$ and $\bs{h}^*$. For future use, we write 
\begin{equation*}
\overline{\alpha} = \Big(\frac{1}{d} \sum_{i=1}^{d} \alpha_i^{-1} \Big)^{-1} \;\;
\text{ and, if } d\geq 3,\;\;
\overline{\alpha}_3 = \Big(\frac{1}{d-2} \sum_{i=3}^{d} \alpha_i^{-1} \Big)^{-1},
\end{equation*}
and we define
\begin{equation*}
    \overline{\gamma} = \frac{2\overline{\alpha}}{2\overline{\alpha} + d} 
    \;\;\text{ and, if } d\geq 3,\;\;
    \overline{\gamma}_3 = \frac{2\overline{\alpha}_3}{2\overline{\alpha}_3 + d - 2}.
\end{equation*}
Before studying the quadratic risk in our model, we first recall the results from~\cite{amorino2023estimation} where the estimation of $\overline{\mu}^b$ is studied from discrete non-noisy observations. In~\cite{amorino2023estimation}, the authors distinguish two regimes depending on $\Delta_n$. The breaking point $w_n^{HF}$ between these two regimes is defined by
\begin{equation*}
    w_n^{HF} = 
    \begin{cases}
        \log(T_n) T_n^{-\frac{1}{2\alpha_1}}  & \, \text{ if } d=1,
\\
        \log(T_n) T_n^{-\frac{1}{2\alpha_1}  - \frac{1}{2\alpha_2} } & \, \text{ if } d=2,
        \\
        \log(T_n)\Big(\frac{\log(T_n)}{T_n} \Big)^{\frac{\overline{\alpha}_3}{(2\overline{\alpha}_3 + d - 2)} \big( \frac{1}{\alpha_1} + \frac{1}{\alpha_2} \big)}  & \, \text{ if } d \geq 3 \text{ and } (\bs{\alpha}, k_0 ) \in D_1,
        \\
        T_n^{\frac{-\overline{\alpha}_3}{(2\overline{\alpha}_3 + d - 2)} \big( \frac{1}{\alpha_1} + \frac{1}{\alpha_2} \big)} & \, \text{ if } d \geq 3 \text{ and } (\bs{\alpha}, k_0 ) \in D_2\cup D_3.
    \end{cases}
\end{equation*}
They show that, when $\Delta_n \leq w^{HF}_n$, the estimator of the invariant measure behaves as in the continuous observation case and therefore exhibits the convergence rate $(v_n^{HF})^{1/2}$ defined as 
\begin{equation*}
    v_n^{HF} = 
    \begin{cases}
        \log(T_n)T_n^{-1} & \, \text{ if } d=1 \text{ or } d=2,
        \\
\Big(\frac{\log(T_n)}{T_n} \Big)^{\frac{2\overline{\alpha}_3}{2\overline{\alpha}_3 + d - 2}} & \, \text{ if } d \geq 3 \text{ and } (\bs{\alpha}, k_0 ) \in D_1,
        \\
        T_n^{-\frac{2\overline{\alpha}_3}{2\overline{\alpha}_3 + d - 2}} & \, \text{ if } d \geq 3 \text{ and } (\bs{\alpha}, k_0) \in D_2\cup D_3.
    \end{cases}
\end{equation*}
To get this convergence rate,~\cite{amorino2023estimation} identifies the optimal bandwidth choice given by 
\begin{equation*}
    h^{*, HF}_i = \begin{cases}
         T_n^{-1/(2\alpha_i)} &\text{if } d=1 \text{ or } d=2 \\
        \left(\frac{\log(T_n)}{T_n}\right)^{\frac{\overline{\alpha}_3}{\alpha_i(2\overline{\alpha}_3 + d - 2)}} &\text{if } d \geq 3 \text{ and } (\bs{\alpha}, k_0) \in D_1,\\
        T_n^{-\frac{\overline{\alpha}_3}{\alpha_i(2\overline{\alpha}_3 + d - 2)}} &\text{if } d \geq 3 \text{ and } (\bs{\alpha}, k_0) \in D_2\cup D_3. 
    \end{cases}
\end{equation*}
When $\Delta_n \geq w_n^{HF}$, the estimator of the invariant measure essentially behaves as in the i.i.d.\ case. In that case, the optimal bandwidth choice $\bs{h}^{*, LF}$ is given by 
\begin{equation*}
    h^{*, LF}_i = n^{\frac{-\overline{\alpha}}{\alpha_i(2\overline{\alpha}+d)}},
\end{equation*}
and
\cite{amorino2023estimation} shows that the estimator exhibits a convergence rate 
\begin{equation*}
    n^{-\frac{\overline{\alpha}}{2\overline{\alpha}+d}}.
\end{equation*} 
Here, we want to reproduces this behaviour while also incorporating noise. In our case, the effective discretization step when using preaveraging is $p \Delta_n$ and therefore, we expect that the switch between the two regimes of~\cite{amorino2023estimation} appears when $p\Delta_n \approx w_n^{HF}$. However, $p$ also needs to be chosen to minimize the quadratic error bound. Applying the same analysis as in~\cite{amorino2023estimation}, we then define $\bs{h}^{*, p}$ for each $p$ by
\begin{equation*}
    h^{*, p}_i
    = 
    \begin{cases}
        h^{*, HF}_i &\text{if } p\Delta_n \leq w_n^{HF},\\
        \big(\frac{p}{n}\big)^{\frac{\overline{\alpha}}{\alpha_i(2\overline{\alpha}+d)}} &\text{if } p\Delta_n \geq w_n^{HF}.\\
    \end{cases}
\end{equation*}
Combining Propositions~\ref{propo:small:bias:mu} and~\ref{prop:variance}, we get for each $p \geq 1$,
\begin{equation}
\label{eq:boundRp}
    \mathcal{R}(\widehat\mu_{n,\bs{h}^{*,p}, p}, b; x)
    \lesssim
    \begin{cases}
        \tau_n^{2\alpha_1} + v_n^{HF} &\text{if } \Delta_n \leq w_n^{HF} \text{ and } p = 1,\\
        \tau_n^{2\alpha_1} + n^{-\frac{2\overline{\alpha}}{2\overline{\alpha}+d}} &\text{if } \Delta_n \geq w_n^{HF} \text{ and } p = 1, \\
        (p\Delta_n)^{\beta} + \frac{\tau_n^{2\alpha_1}}{p^{\alpha_1}} + v_n^{HF}&\text{if } p\Delta_n \leq w_n^{HF} \text{ and } p \geq 2,\\
        (p\Delta_n)^{\beta} + \frac{\tau_n^{2\alpha_1}}{p^{\alpha_1}} + (\frac{p}{n})^{\frac{2\overline{\alpha}}{2\overline{\alpha}+d}} &\text{if } p\Delta_n \geq w_n^{HF}  \text{ and } p \geq 2.
    \end{cases}
\end{equation}
In order to make this minimization feasible, we first remove some of these cases. We then distinguish two main cases:
\begin{itemize}
    \item The small noise case, when $\tau_n^{2\alpha_1} \leq \Delta_n^\beta$,
    \item The large noise case, when $\tau_n^{2\alpha_1} > \Delta_n^\beta$.
\end{itemize}

\textbf{Study of~\eqref{eq:boundRp} in the small noise case.} Suppose that $\tau_n^{2\alpha_1} \leq \Delta_n^\beta$. In that case, we can check that
\begin{equation*}
    \mathcal{R}(\widehat\mu_{n,\bs{h}^{*,1}, 1}, b; x) \lesssim \inf_{p > 1} \mathcal{R}(\widehat\mu_{n,\bs{h}^{*,p}, p}, b; x),
\end{equation*}
and thus we should always take $p^* = 1$. We get
\begin{equation*}
        \mathcal{R}(\widehat\mu_{n,\bs{h}^{*,1}, 1}, b; x) \lesssim
        \begin{cases}
        \tau_n^{2\alpha_1} + v_n^{HF} &\text{if } \Delta_n \leq w_n^{HF},\\
        \tau_n^{2\alpha_1} + n^{-\frac{2\overline{\alpha}}{2\overline{\alpha}+d}} &\text{if } \Delta_n \geq w_n^{HF}. \\
        \end{cases}
\end{equation*}
\\

\textbf{Study of~\eqref{eq:boundRp} in the large noise case.}
Suppose now that $\tau_n^{2\alpha_1} > \Delta_n^\beta$. Then we see that we can simplify~\eqref{eq:boundRp} by removing the distinction between the cases $p=1$ and $p\geq 2$
\begin{equation*}
    \mathcal{R}(
        \widehat\mu_{n,\bs{h}^{*,p}, p}, b; x)
    \lesssim
    \begin{cases}
        (p\Delta_n)^{\beta} + \frac{\tau_n^{2\alpha_1}}{p^{\alpha_1}} + v_n^{HF}&\text{if } p\Delta_n \leq w_n^{HF},\\
        (p\Delta_n)^{\beta} + \frac{\tau_n^{2\alpha_1}}{p^{\alpha_1}} + (\frac{p}{n})^{\frac{2\overline{\alpha}}{2\overline{\alpha}+d}} &\text{if } p\Delta_n \geq w_n^{HF}.
    \end{cases}
\end{equation*}
This bound can be further simplified by noting that the condition $p\Delta_n \gtrsim w_n^{HF}$ implies that $(p\Delta_n)^{\beta} \gtrsim  (\frac{p}{n})^{\frac{2\overline{\alpha}}{2\overline{\alpha} + d}}$. This effectively removes the contribution $(\frac{p}{n})^{\frac{2\overline{\alpha}}{2\overline{\alpha}+d}}$ and thus
\begin{equation*}
    \mathcal{R}(\widehat\mu_{n,\bs{h}^{*,p}, p}, b; x)
    \lesssim
    \begin{cases}
        (p\Delta_n)^{\beta} + \frac{\tau_n^{2\alpha_1}}{p^{\alpha_1}} + v_n^{HF}&\text{if } p\Delta_n \leq w_n^{HF},\\
        (p\Delta_n)^{\beta} + \frac{\tau_n^{2\alpha_1}}{p^{\alpha_1}} &\text{if } p\Delta_n \geq w_n^{HF}.
    \end{cases}
\end{equation*}
In that case, it is clear that the optimal choice for $p$ is given by 
\begin{equation*}
    p^* =
    \Big\lceil \big(\tau_{n}^{2\alpha_1} \Delta_n^{-\beta} \big)^{1/(\beta+\alpha_1)} \Big\rceil \vee 1.
\end{equation*}
With this choice in the large-noise case,
\begin{equation*}
p^*\Delta_n \asymp (\tau_n^2\Delta_n)^{\alpha_1/(\beta+\alpha_1)}
\quad\text{and}\quad
(p^*\Delta_n)^\beta
\asymp
\frac{\tau_n^{2\alpha_1}}{(p^*)^{\alpha_1}}
\asymp
(\tau_n^2\Delta_n)^{\alpha_1\beta/(\beta+\alpha_1)}.
\end{equation*}
Then
\begin{equation*}
    \mathcal{R}(\widehat\mu_{n,\bs{h}^{*,p^*}, p^*}, b; x)
    \lesssim
    \begin{cases}
        (\tau_n^2 \Delta_n)^{\frac{\alpha_1\beta}{\alpha_1 + \beta}} + v_n^{HF}&\text{if } (\tau_n^2 \Delta_n)^{\frac{\alpha_1}{\alpha_1 + \beta}} \leq w_n^{HF},\\
        (\tau_n^2 \Delta_n)^{\frac{\alpha_1\beta}{\alpha_1 + \beta}}  &\text{if } (\tau_n^2 \Delta_n)^{\frac{\alpha_1}{\alpha_1 + \beta}} \geq w_n^{HF}.
    \end{cases}
\end{equation*}
\\

\textbf{Conclusion.} Combining these two cases, we see that a general choice for $p^*$ is given by 
\begin{equation*}
    p^* =
    \Big\lceil \big(\tau_{n}^{2\alpha_1} \Delta_n^{-\beta} \big)^{1/(\beta+\alpha_1)} \Big\rceil \vee 1. 
\end{equation*}
With this choice, we then take $\bs{h}^{*} = \bs{h}^{*, p^*}$. We can distinguish the following four asymptotic cases.

\begin{proposition}[Small noise intensity, high sampling frequency]
\label{prop:smallnoise:highsampling}
    Suppose that $p^{\ast}\Delta_n \leq w_n^{HF}$ and that $\tau_n^{2\alpha_1} \leq \Delta_n^{\beta}$. In that case, we have
    \begin{equation*}
    \mathcal{R}(\widehat\mu_{n,\bs{h}^{*}, p^*}, b; x)
    \lesssim
    v_n^{HF} \vee \tau_n^{2\alpha_1}.
    \end{equation*}
\end{proposition}

\begin{proposition}[Large noise intensity, high sampling frequency]
\label{prop:largenoise:highsampling}
    Suppose that $p^{\ast}\Delta_n \leq w_n^{HF}$ and that $\tau_n^{2\alpha_1} \geq \Delta_n^{\beta}$. In that case, we have
    \begin{equation*}
    \mathcal{R}(\widehat\mu_{n,\bs{h}^{*}, p^*}, b; x)
    \lesssim
    v_n^{HF} \vee \big(\tau_n^2 \Delta_n\big)^{\frac{\alpha_1 \beta}{\beta+\alpha_1}}.
    \end{equation*}
\end{proposition}

\begin{proposition}[Small noise intensity, Low sampling frequency]
\label{prop:smallnoise:lowsampling}
    Suppose that $p^*\Delta_n \geq w_n^{HF}$ and that $\tau_n^{2\alpha_1} \leq \Delta_n^{\beta}$. In that case, we have
    \begin{equation*}
    \mathcal{R}(\widehat\mu_{n,\bs{h}^{*}, p^*}, b; x)
    \lesssim
    n^{-\frac{2\overline{\alpha}}{2\overline{\alpha}+d}} \vee \tau_n^{2\alpha_1}.
    \end{equation*}
\end{proposition}

\begin{proposition}[Large noise intensity, Low sampling frequency]
\label{prop:largenoise:lowsampling}
    Suppose that $p^*\Delta_n \geq w_n^{HF}$ and that $\tau_n^{2\alpha_1} \geq \Delta_n^{\beta}$. In that case, we have
    \begin{equation*}
    \mathcal{R}(\widehat\mu_{n,\bs{h}^{*}, p^*}, b; x)
    \lesssim
     \big(\tau_n^2 \Delta_n\big)^{\frac{\alpha_1\beta}{\beta+\alpha_1}}.
    \end{equation*}
\end{proposition}

\subsection{Adaptive choice of the bandwidth}

The Bernstein inequality can be used to adaptively tune the parameter $\bs{h}$, following the ideas of~\cite{goldenshluger2008universal, goldenshluger2009structural, goldenshluger2011bandwidth}. Fix $p\geq1$. As previously discussed, for dimensions $d=1$ and $d=2$, the variance upper bound does not rely on the smoothness. Consequently, there is no advantage to adopting a data-oriented adaptive strategy for $d<3$. 
	
	For $d \geq 3$, the strategy is to consider a range of possible bandwidths and to select the one that minimizes the error. To this aim, we introduce a heuristic representation of the bias and a penalty term proportional to the variance bound of Proposition~\ref{prop:variance}. The optimal bandwidths are those minimizing their sum. Let us begin by introducing the grid
\begin{equation*}
\begin{aligned}
\mathcal{H}_n^p \subset\left\{\bs{h} \in(0,1]^d: \: h_1 \leq \dots \leq h_d, \:   \forall l=1, \ldots, d \quad\left(\frac{\log(n_p)^{3}}{n_p}\right)^{1/d} \leq h_l \leq 1 \right\},
\end{aligned}
\end{equation*}
where we recall that $n_p$ stands for $\lfloor n/p\rfloor$ for the sake of clarity. We also assume that $\# \mathcal{H}_n^p \leq T_n$. According to the set of candidate bandwidths, we can introduce the set of candidate estimators:
\begin{equation*}
\begin{aligned}
& \calF\left(\mathcal{H}_n^p\right):=\left\{\hat{\mu}_{n,\bs{h}, p}(x): \quad x \in \mathbb{R}^d, \quad \bs{h} \in \mathcal{H}_n^p\right\} . \\
\end{aligned}
\end{equation*}

The goal of this section is to choose an estimator in the family $\calF\left(\mathcal{H}_n^p\right)$, in a completely data-driven way. Following the idea in~\cite{goldenshluger2008universal, goldenshluger2009structural, goldenshluger2011bandwidth}, our selection procedure relies on the introduction of auxiliary convolution estimators, $\widehat \mu_{n, (\bs{h}, \bs{\eta}), p}$ for $(\bs{h}, \bs{\eta}) \in \left( \mathcal{H}_n^p\right)^2$ which are defined in the same way as the estimator introduced in Section~\ref{sec:construction} but with kernel $\bs{K}_{\bs{h}} \ast \bs{K}_{\bs{\eta}}$. The first important remark is that when the regularity of the function of interest is unknown, the upper bound on the variance for $p$ fixed can be rewritten for any $\bs{h} \in \mathcal{H}_n^p$,
\begin{equation*}
	v^p_n(\bs{h}) = \frac{1}{T_n}  \Big(p\Delta_n \prod_{i=1}^d h_i^{-1} + \min \Big( \sum_{i=1}^d|\log(h_i)|\prod_{i=3}^d h_i^{-1},  (h_2h_3)^{-1/2}\prod_{i=4}^d h_i^{-1}, \min_{k_0\geq 3} \prod_{i=1}^{k_0} h_i^{\frac{2-k_0}{k_0}}\prod_{i=k_0+1}^d h_i^{-1}\Big)\Big).
\end{equation*}
Moreover, one can write:
\begin{equation*}
	 \sum_{i=1}^d|\log(h_i)|\prod_{i=3}^d h_i^{-1} \leq  \prod_{i=1}^{3} h_i^{\frac{-1}{3}}\prod_{i=4}^d h_i^{-1}
	\iff  h_3^2 \geq h_1h_2\sum_{i=1}^d|\log(h_i)|. 
\end{equation*}
Using the fact that $h_1 \leq \dots \leq h_d$, we obtain
\begin{equation*}
	v^p_n(\bs{h}) = \frac{1}{T_n}\left(p\Delta_n \prod_{i=1}^d h_i^{-1} + \min\left( \sum_{i=1}^d|\log(h_i)|\prod_{i=3}^d 		h_i^{-1}, (h_2h_3)^{-1/2}\prod_{i=4}^d h_i^{-1}\right)\right).
\end{equation*}
With this purpose in mind, we introduce the following penalty function 
\begin{equation*}
	V^p_n(\bs{h}) = \widebar{\omega} \log(n_p) v^p_n(\bs{h}),
\end{equation*}
for some positive constant $\widebar{\omega}$ which has to be taken large, and
\begin{equation}\label{eq:def:inflated:penalty}
	\widetilde V^p_n(\bs h)
	=
	\widebar{\omega}\Big[
	L_p v^p_n(\bs h)
	+
	h_p^{-4}\frac{L_p^2\Lambda_p^2\|\mathbb K_{\bs h}\|_\infty^2}{n_p^2}
	\Big],
\end{equation}
where
\[
	h_p=p\Delta_n,\qquad L_p=\log(n_p),\qquad
	\Lambda_p=L_p+\log(e/h_p).
\]
We assume in this adaptive subsection that $0<h_p\le1$ and $n_p\ge2$.
The intuition is that although one would want to use the penalty $V^p_n(\bs{h})$, the inflated penalty $\widetilde V^p_n(\bs{h})$ is forced by Theorem~\ref{th:bernstein:loc} to ensure that the Bernstein inequality holds uniformly over the grid $\mathcal{H}_n^p$. Now, following once again the procedure of~\cite{goldenshluger2008universal,goldenshluger2009structural, goldenshluger2011bandwidth}, we define for any $\bs{h} \in \mathcal{H}_n^p$, 
\begin{equation*}
	A^p_n(\bs{h}) =\max _{\bs{\eta} \in \mathcal{H}_n^p}\left\{\left|\widehat{\mu}_{n,(\bs{h}, \bs{\eta}),p}\left(x\right)-\widehat{\mu}_{n,\bs{\eta},p}\left(x\right)\right|^2-\widetilde V^p_n(\bs{\eta})\right\}_{+}.
\end{equation*}
Finally, we define the choice procedure
\begin{equation}\label{eq:def:hstar:inflated}
	\bs{h}^* \in \mathrm{argmin}_{\bs{h} \in \mathcal{H}_n^p} \left\{ A^p_n(\bs{h}) +\widetilde V^p_n(\bs{h}) \right\}.
\end{equation}

With this methodology, we obtain, for any preaveraging level $p$, the following oracle inequality.

\begin{proposition}[Oracle inequality]\label{prop:oracle}
	Suppose that $b \in \Sigma(\bs{\mathfrak{b}})$, $d\ge3$, and fix any $\gamma>1$. There exist $\widebar\omega_0<\infty$, $n_0\in\mathbb N$ and $c<\infty$, depending only on $\bs{\mathfrak b}$ and $\gamma$, such that, if $\widebar\omega\ge\widebar\omega_0$, then for all $n\ge n_0$ and all integers $p\ge1$ satisfying
	\[
	0<p\Delta_n\le1,\qquad n_p\ge2,\qquad p^2\Delta_n\le c_\star,
	\qquad
	p\le \left(\frac{\tau_n^2}{\Delta_n}\right)^{1/2}\vee1,
	\]
	the selector~\eqref{eq:def:hstar:inflated} satisfies
	\begin{equation}\label{eq:oracle:inflated}
		\mathbb{E}\left[ \left| \widehat\mu_{n, \bs{h}^*, p}(x) - \widebar\mu^b(x)\right|^2\right]
		\leq
		c\inf_{\bs{h} \in \mathcal{H}_n^p}
		\left\{ \mathrm{B}_{n,\bs{h},p}^b(x)^2 + \widetilde V_n^p(\bs{h})\right\}
		+c \left(\frac{p}{n}\right)^\gamma
		+c(p\Delta_n)^{\beta}\mathds{1}_{p\geq 2}
		+c\frac{\tau_n^{2\alpha_1}}{p^\alpha_1}.
	\end{equation}
\end{proposition} 

The proof of Proposition~\ref{prop:oracle} is postponed to Section~\ref{proof:prop:oracle}.

\section{Concluding remarks}
\label{sec:comments}

\subsection*{Comments on the convergence rate}

A convenient way to interpret the convergence rates obtained above is to specialize them to polynomial regimes of the form
\[
\Delta_n = n^{-\theta}
\quad\text{and}\quad
\tau_n = n^{-\kappa},
\]
for some $\theta,\kappa \ge 0$. Since $\Delta_n \to 0$ and $T_n = n\Delta_n \to \infty$, we must have $0 < \theta < 1$. For simplicity, we also restrict to $\kappa \ge 0$, although a more detailed analysis shows that consistent estimation of the invariant density is still possible as soon as $\kappa > -\theta/(2\alpha_1)$.

To keep the discussion readable, we ignore logarithmic factors in the rates $w_n^{HF}$ and $v_n^{HF}$. Recall that
\[
\overline{\alpha} = \Big(\frac{1}{d} \sum_{i=1}^{d} \alpha_i^{-1} \Big)^{-1},
\qquad
\overline{\alpha}_3 = \Big(\frac{1}{d-2} \sum_{i=3}^{d} \alpha_i^{-1} \Big)^{-1},
\qquad \overline{\gamma} = \frac{2\overline{\alpha}}{2\overline{\alpha} + d},
\quad
\text{and} \quad
\overline{\gamma}_3 = \frac{2\overline{\alpha}_3}{2\overline{\alpha}_3 + d - 2}.
\]
In terms of $(\theta,\kappa)$, the four regimes identified in Propositions~\ref{prop:smallnoise:highsampling}--\ref{prop:largenoise:lowsampling} can be summarized as follows:
\begin{itemize}
    \item Proposition~\ref{prop:smallnoise:highsampling} applies when $\kappa \geq \theta \beta / (2\alpha_1)$ and $\theta \geq \overline{\gamma} (\alpha_1^{-1} + \alpha_2^{-1})/2$,
    \item Proposition~\ref{prop:largenoise:highsampling} applies when $\kappa \leq \theta\beta / (2\alpha_1)$ and $2 \kappa \delta_0 + \theta (\delta_0 + \delta_1) \geq \delta_1$, with $\delta_0 = \alpha_1 / (\beta + \alpha_1)$ and $\delta_1 = \wb \gamma_3 (\alpha_1^{-1} + \alpha_2^{-1})/2$,
    \item Proposition~\ref{prop:smallnoise:lowsampling} applies when $\kappa \geq \theta\beta / (2\alpha_1)$ and $\theta \leq \overline{\gamma} (\alpha_1^{-1} + \alpha_2^{-1})/2$,
    \item Proposition~\ref{prop:largenoise:lowsampling} applies when $\kappa \leq \theta\beta / (2\alpha_1)$ and $2 \kappa \delta_0 + \theta (\delta_0 + \delta_1) \leq \delta_1$.
\end{itemize}

These regions are represented in Figure~\ref{fig:rate1}, together with the corresponding convergence rates for the pointwise risk $\mathcal{R}(\widehat\mu, b; x)$.

\begin{figure}[h]
    \centering
        \begin{tikzpicture}[scale = 1]
                \begin{axis}[
        axis lines=middle,
        axis line style={-},
        x axis line style = {-},
        y axis line style = {->},
        xtick=\empty,
        ytick=\empty,
        xlabel style={at=(current axis.right of origin), anchor=west},
        ylabel style={at=(current axis.above origin), anchor=south},
        xmin=-1, xmax=10,
        ymin=-1, ymax=10,
        clip=false
    ]


    \addplot[color=red!40!gray,mark=none,thick]  coordinates {(0,0) (7,8) };
    \addplot[color=orange!40!gray,mark=none,thick]  coordinates {(4.55,9) (4.55,10) };
    \addplot[color=brown!40!gray,mark=none,dotted]  coordinates {(7,8) (7,0) };
    \addplot[color=orange!40!gray,mark=none,dotted]  coordinates {(4.55,9) (4.55,0) };
    \addplot[color=green!40!gray,mark=none,thick]  coordinates {(4.55,9) (0,9) };
    \addplot[color=yellow!40!gray,mark=none,thick]  coordinates {(8.4,0) (7,8) };
    \addplot[color=yellow!40!gray,mark=none,dotted]  coordinates {(8.4,0) (8.4,-1) };
    \addplot[color=brown!40!gray,mark=none,thick]  coordinates {(4.55,9) (7,8) };
    \addplot[color=black,mark=none,dashed] coordinates {(9.7, -0.5) (9.7, 11)};

    \node at (axis cs:1.2,2.8) {\color{red!40!gray}\tiny$\kappa = \frac{\theta \beta }{2\alpha_1}$};
    \node at (axis cs:0,10.3) {\footnotesize$\kappa$};
    \node at (axis cs:10.3,0) {\footnotesize$\theta$};
    \node at (axis cs:4,-0.5) {\color{orange!40!gray}\tiny$\theta = \overline{\gamma} \frac{\alpha_1^{-1} + \alpha_2^{-1}}{2}$};
    \node at (axis cs:-1,9) {\color{green!40!gray}\tiny$\kappa = \frac{\bar\gamma}{2\alpha_1}$};
    \node at (axis cs:6.8,-0.5) {\color{brown!40!gray}\tiny$\theta = \frac{\widebar{\gamma}_3}{\widebar{\gamma}_3+\beta}$};
    \node at (axis cs:8.8,-1.5) {\color{yellow!40!gray}\tiny$\theta = \frac{(\alpha_1+\beta)\wb\gamma_3}{(\alpha_1+\beta)\wb\gamma_3+\alpha_1\beta}$};
    \node at (axis cs:2.65,9.6) {\footnotesize$n^{-\frac{2\bar\alpha}{2\bar\alpha+d}}$};
    \node at (axis cs:8.3,8.7) {\footnotesize$v_n^{HF}$};
    \node at (axis cs:5,2.8) {\tiny{$(\tau_n^2\Delta_n)^{\frac{\alpha_1 \beta}{\beta+\alpha_1}}$}};
    \node at (axis cs:2.5,6) {\footnotesize{$\tau_n^{2\alpha_1}$}};
    \node at (axis cs:9.7,-0.7) {\footnotesize{$\theta = 1$}};
    \end{axis}
        \end{tikzpicture}
        \caption{Rates of convergence of $\widebar\mu_{n, \bs{h}, p}(x)$.}
        \label{fig:rate1}
\end{figure}

\subsection*{Fully isotropic case}
A particularly transparent situation is the {fully isotropic} case, where
\[
\beta = \alpha_1 = \dots = \alpha_d = \alpha,
\]
so that the invariant density $\overline{\mu}^b$ is $\mathcal{C}^\alpha$. In this case
\[
\overline{\gamma} = \frac{2 \alpha}{2\alpha + d},
\qquad
\overline{\gamma}_3 = \frac{2\alpha}{2\alpha + d - 2},
\]
and the conditions defining the four regimes simplify considerably:
\begin{itemize}
    \item Proposition~\ref{prop:smallnoise:highsampling} applies when $\kappa \geq \theta \alpha / (2\alpha)$ and $\theta \geq \overline{\gamma} \alpha^{-1}$,
    \item Proposition~\ref{prop:largenoise:highsampling} applies when $\kappa \leq \theta\alpha / (2\alpha)$ and $2 \kappa + \theta (1 + 2 \delta) \geq 2 \delta_1$, with $\delta = \wb \gamma_3 \alpha^{-1}$,
    \item Proposition~\ref{prop:smallnoise:lowsampling} applies when $\kappa \geq \theta/2$ and $\theta \leq \overline{\gamma} \alpha^{-1}$,
    \item Proposition~\ref{prop:largenoise:lowsampling} applies when $\kappa \leq \theta / 2$ and $\kappa + \theta (1 + 2 \delta) \leq 2 \delta$.
\end{itemize}

In the isotropic regime $\beta = \alpha_1$, the choice of the preaveraging parameter $p^\ast$ becomes particularly simple. With the choice of Section~\ref{sec:upper_bounds}, we obtain
\[
p^\ast
=
\lceil \tau_{n} \Delta_n^{-1/2} \rceil \vee 1,
\]
so that $p^\ast$ automatically adapts to the relative scales of the noise level and the sampling frequency.

The resulting regions and rates are depicted in Figure~\ref{fig:rate2}.

\begin{figure}[h]
    \centering
        \begin{tikzpicture}[scale = 1]
                \begin{axis}[
        axis lines=middle,
        axis line style={-},
        x axis line style = {-},
        y axis line style = {->},
        xtick=\empty,
        ytick=\empty,
        xlabel style={at=(current axis.right of origin), anchor=west},
        ylabel style={at=(current axis.above origin), anchor=south},
        xmin=-1, xmax=10,
        ymin=-1, ymax=10,
        clip=false
    ]


    \addplot[color=red!40!gray,mark=none,thick]  coordinates {(0,0) (6,6) }; 
    \addplot[color=orange!40!gray,mark=none,thick]  coordinates {(6,6) (6,10) };
    \addplot[color=orange!40!gray,mark=none,dotted]  coordinates {(6,6) (6,0) }; 
    \addplot[color=green!40!gray,mark=none,thick]  coordinates {(6,6) (0,6) }; 
    \addplot[color=yellow!40!gray,mark=none,thick]  coordinates {(8.4,0) (6,6) }; 
    \addplot[color=yellow!40!gray,mark=none,dotted]  coordinates {(8.4,0) (8.4,-1) }; 
    \addplot[color=black,mark=none,dashed] coordinates {(9.7, -0.5) (9.7, 11)}; 

    \node at (axis cs:1.2,2.8) {\color{red!40!gray}\tiny$\kappa = \frac{\theta }{2}$}; 
    \node at (axis cs:0,10.3) {\footnotesize$\kappa$}; 
    \node at (axis cs:10.3,0) {\footnotesize$\theta$}; 
    \node at (axis cs:-1,6) {\color{green!40!gray}\tiny$\kappa = \frac{\bar\gamma}{2\alpha}$}; 
    \node at (axis cs:6,-0.5) {\color{orange!40!gray}\tiny$\theta = \frac{2}{2\alpha+d}$}; 
    \node at (axis cs:8.8,-1.5) {\color{yellow!40!gray}\tiny$\theta = \frac{4}{2\alpha+d+2}$}; 
    \node at (axis cs:2.65,9.6) {\footnotesize$n^{-\frac{2\alpha}{2\alpha+d}}$};
    \node at (axis cs:8.3,8.7) {\footnotesize$v_n^{HF}$};
    \node at (axis cs:5,2.8) {\tiny{$\tau_n^\alpha \Delta_n^{\frac{\alpha}{2}}$}};
    \node at (axis cs:2.5,4.5) {\footnotesize{$\tau_n^{2\alpha}$}};
    \node at (axis cs:9.7,-0.7) {\footnotesize{$\theta = 1$}};
    \end{axis}
        \end{tikzpicture}
        \caption{Rates of convergence of $\widebar\mu_{n, \bs{h}, p}(x)$ when $\beta = \alpha_1 = \dots = \alpha_d = \alpha$.}
        \label{fig:rate2}
\end{figure}

\subsection*{Numerical implementation}

We now briefly discuss the numerical behaviour of the estimators proposed in Section~\ref{sec:construction}. We first focus on the one-dimensional case $d=1$. According to the analysis of Section~\ref{sec:upper_bounds}, the optimal bandwidth is of order $h \simeq T_n^{-1}$, and the optimal preaveraging parameter is
\[
p^\ast = \big\lceil \tau_n \Delta_n^{-1/2} \big\rceil \vee 1.
\]

From a theoretical perspective, the most accurate estimator is the bias-corrected estimator $\widehat\mu_{n,\bs{h},p}(x)$ defined in~\eqref{eq:def:est:mu}. However, this estimator can be numerically unstable, as the variance may become large. For instance,
\[
\Var\big(\widehat{\mu}_{n,\bs{h},p}(x)\big)
\;\le\;
(l+1)^d \sup_{\bs{\gamma} \in \{0,\dots,l\}^d} 
|\bs{u}_{\bs{\gamma}}|\,
\Var \Big( \widehat{\nu}_{n,\bs{h},p}^{R}\big(x+\bs{\gamma}\widetilde{\tau}_{n} p^{-1/2}\big) \Big),
\]
and a similar inequality holds when bounding $\widehat{\nu}_{n,\bs{h},p}^{R}$ in terms of $\widehat{\nu}_{n,\bs{h},p}$. Numerical experiments show that even in dimension $d=1$, the coefficients $\bs{u}_{\bs{\gamma}}$ grow quickly with $l$, which significantly inflates the variance. To illustrate this effect, consider
$V(x) = x^2/4$ and $b(x) = -x/2$, with $n = 2^{14}$, $\Delta_n = n^{-1/2} = 2^{-7}$ and $\tau_n = 1$. We compare the bias, variance and mean squared error of $\widehat\nu_{n,\bs{h},p}(x)$ and $\widehat\mu_{n,\bs{h},p}(x)$ for $p = \lfloor (\tau_n^2 \Delta_n^{-1})^{1/2} \rfloor \vee 1$. The results are reported in Table~\ref{tab:table_bias_correction}.

\begin{table}[H]
\centering
\begin{tabular}{|c|c|c|}
\hline
 & Without bias correction & With bias correction \\ \hline
Error     &  \textbf{1.61e-3} & 8.2e-3 \\
Bias   & 1.21e-3 &  \textbf{3.48e-4}  \\
Variance & \textbf{1.67 e-3} & 9.77e-3\\ \hline
\end{tabular}
\caption{Impact of the bias correction.}
\label{tab:table_bias_correction}
\end{table}

Although bias correction reduces the bias, the associated increase in variance leads to a larger mean squared error. In practice, the uncorrected preaveraged estimator $\widehat\nu_{n,\bs{h},p}(x)$ often performs better than the fully bias-corrected estimator $\widehat\mu_{n,\bs{h},p}(x)$. For this reason, we focus on $\widehat\nu_{n,\bs{h},p}(x)$ in our simulations. In this case, the optimal choice
\[
p^\ast 
= \Big\lfloor (\tau_n^2 \Delta_n^{-1})^{1/2} \Big\rfloor \vee 1
\]
has a simple asymptotic form, which coincides with the optimal $p^\ast$ in the isotropic situation $\beta = \alpha_1$.

We now compare the preaveraged estimator with the naive estimator corresponding to $p=1$, still in the model $V(x) = x^2/4$, $b(x) = -x/2$,
with $n = 2^{14}$, $\Delta_n = 2^{-7}$ and $\tau_n = 1$. Table~\ref{tab:my_label} reports the pointwise quadratic error at several locations $x$ for different values of $p$.

\begin{table}[H]
\centering
\begin{tabular}{c|ccccc}
\hline
$p$ & {$x = 0$} & {$x = 0.25$} & {$x = 0.5$} & {$x = 1$} & {$x = 0.75$} \\ \hline
1    &  1.29e-1 & 1.20e-1 & 9.75e-2 & 6.80e-2 &  4.12e-2 \\
16   & 6.31e-2 &  5.62e-2 &  4.36e-2 &  2.57e-2 & 1.08e-2 \\
 $\bs{p^*}$  & \textbf{1.04e-2} & \textbf{1.08e-2} & \textbf{6.49e-3} & \textbf{2.83e-3} & \textbf{1.70e-3} \\
1024 & 7.07e-2 & 5.05e-2 & 4.06e-2 & 1.67e-2 & 2.02e-2 \\
4096 & 7.49e-1 & 5.39e-1 & 3.25e-1 & 7.27e-2 & 5.22e-2 \\ \hline
\end{tabular}
\caption{Pointwise quadratic error.}
\label{tab:my_label}
\end{table}
~\\

The optimal preaveraging level $p^\ast$ clearly improves the estimation error compared to the naive choice $p=1$, while excessively large values of $p$ deteriorate performance, as expected from the theoretical analysis. We also consider the case $d=2$, with $V(x) = \frac{|x|^2}{4}$, so that $\overline{\mu}^b$ is the density of a standard two-dimensional Gaussian. Using the same parameters as in the one-dimensional example, Figure~\ref{fig1} compares the target density with the estimated densities obtained with and without preaveraging. The preaveraged estimator is visibly closer to the target, illustrating the practical relevance of the proposed method.

\begin{figure}
\includegraphics[scale=0.4]{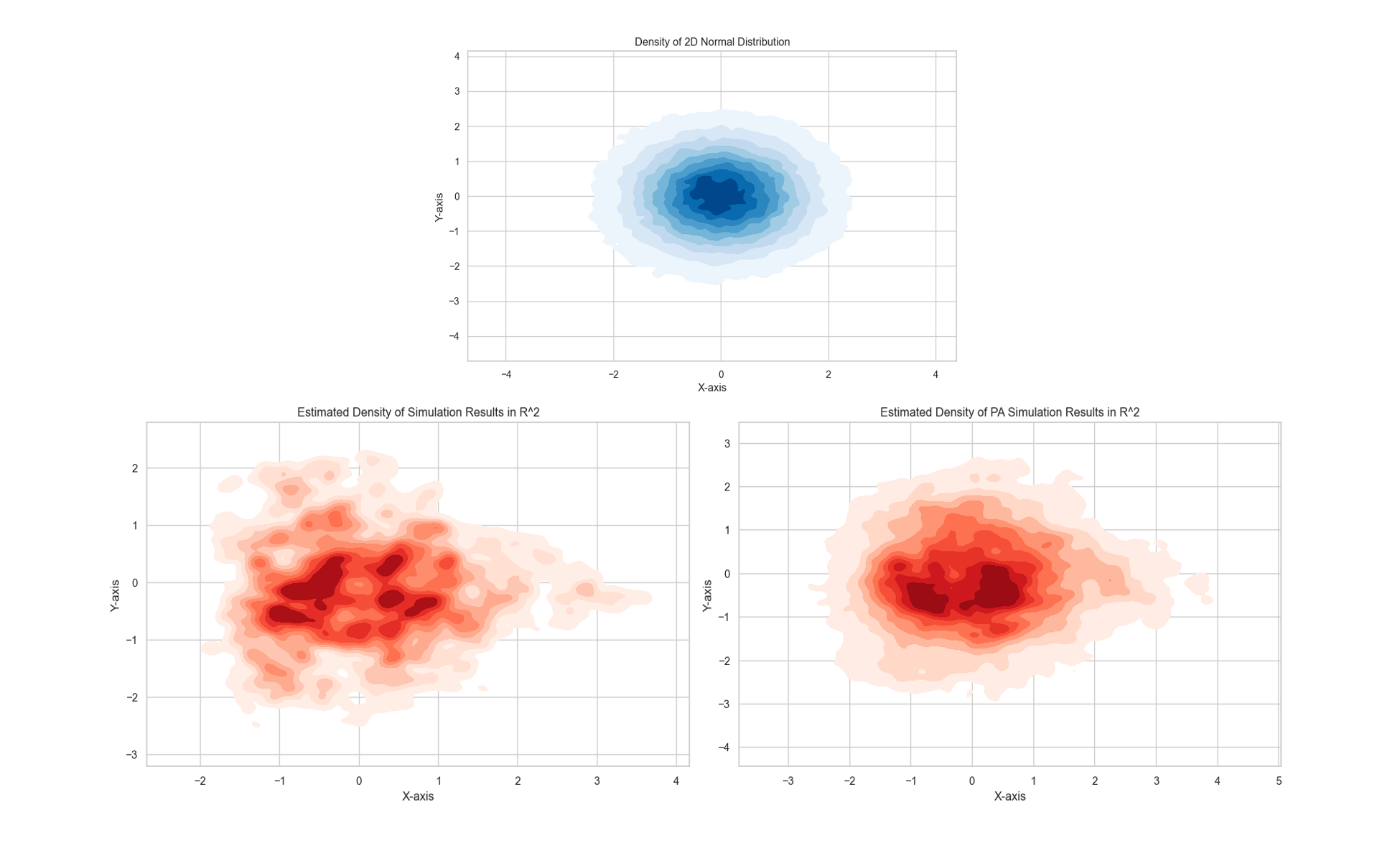}
\caption{True invariant density (top) and kernel estimators without preaveraging (bottom left) and with preaveraging (bottom right) in dimension $d=2$.} 
\label{fig1}
\end{figure}

\subsection*{Beyond Gaussian noise}

Throughout this work, we have assumed that the observation noise is Gaussian. A closer inspection of the proofs shows that this assumption can be relaxed in several directions.

For the variance bounds, it is sufficient that the noise be sub-Gaussian: if $f_\xi$ denotes the density of $\xi_0$, we require
\[
f_\xi(u) \leq C_\xi \exp(-c_\xi |u|^2),
\qquad u\in\R^d,
\]
for some constants $c_\xi, C_\xi > 0$. Under this condition, the arguments of Sections~\ref{sec:hidden} and~\ref{sec:estimation} extend with minor modifications.

Extending the bias analysis is more delicate. The construction of Section~\ref{sec:construction} uses the exact Gaussian moments of the noise to define the matrix $A$ in~\eqref{eq:def:a} and to invert it explicitly. For a general noise distribution, the same algebraic construction can be carried out provided the odd moments vanish and the corresponding matrix $A$ is invertible. However, invertibility is no longer guaranteed in general, and the resulting coefficients may be difficult to control.

On the other hand, if one discards the bias-correction step used to deconvolve the noise (that is, if one works with the preaveraged estimator $\widehat\nu_{n,\bs{h},p}$ instead of $\widehat\mu_{n,\bs{h},p}$), the bias bounds only require the noise to be centred. In this sense, the core preaveraging methodology is robust to deviations from Gaussianity, and the Gaussian assumption is mainly needed to implement the higher-order deconvolution step in a simple and explicit way.

\subsection*{Acknowledgments} The authors would like to thank Marc Hoffmann, Pierre Cardaliaguet and Mathieu Rosenbaum for the helpful discussions during the preparation of this work.

\bibliographystyle{alpha} 
\bibliography{main.bib}

\appendix

\section{Proof of Proposition~\ref{prop:Z-geom-ergodic}}

The strict positivity and continuity of the joint density of $(X_{h/p},\ldots,X_h)$, recalled before the proposition, imply that $\mathfrak P_h^b$ has a strictly positive transition density on $\R^{2d}$ for every $h>0$. Hence the augmented chain is Lebesgue-irreducible, and the convergence estimate proved below implies aperiodicity and uniqueness of the invariant law. We now give the quantitative argument used in the statement. The first point is that the diffusion itself is uniformly exponentially ergodic in total variation with an explicit physical-time rate. Let
\[
r(x)=(1+|x|^2)^{1/2},\qquad V_X(x)=\exp(a r(x)),
\]
where $a>0$ will be chosen below. Denote the generator of $X$ by
\[
\mathcal L^b f(x)=b(x)\cdot\nabla f(x)+\frac12\Delta f(x).
\]
For $V_X=e^{ar}$,
\[
\nabla V_X(x)=aV_X(x)\frac{x}{r(x)},\qquad
\Delta V_X(x)=V_X(x)\left\{a\left(\frac{d}{r(x)}-\frac{|x|^2}{r(x)^3}\right)+a^2\frac{|x|^2}{r(x)^2}\right\}.
\]
Hence
\[
\frac{\mathcal L^bV_X(x)}{V_X(x)}
=a\frac{\langle b(x),x\rangle}{r(x)}
+\frac a2\left(\frac{d}{r(x)}-\frac{|x|^2}{r(x)^3}\right)
+\frac{a^2}{2}\frac{|x|^2}{r(x)^2}.
\]
By Assumption~\ref{assumption:ergodicity},
\[
a\frac{\langle b(x),x\rangle}{r(x)}
\le a\frac{-b_2|x|+b_3}{r(x)}.
\]
Choose $a\in(0,b_2)$. Since $|x|/r(x)\to1$ and $|x|^2/r(x)^2\to1$, there exist constants $\lambda>0$ and $R_0<\infty$, depending only on $a,b_2,b_3$ and $d$, such that
\[
|x|\ge R_0\quad\Longrightarrow\quad
\mathcal L^bV_X(x)\le -\lambda V_X(x)
\]
uniformly over $b\in\Sigma(\bs{\mathfrak b})$. On the ball $B_{R_0}$, Assumption~\ref{assumption:boundedness} gives the uniform bound
\[
|b(x)|\le b_0+b_1R_0,\qquad |x|\le R_0,
\]
and therefore
\[
\sup_{b\in\Sigma(\bs{\mathfrak b})}\sup_{|x|\le R_0}\mathcal L^bV_X(x)<\infty.
\]
Increasing the constant if necessary, we obtain the uniform Foster--Lyapunov estimate
\[
\mathcal L^bV_X\le -\lambda V_X+B\1_{B_{R_0}},
\qquad b\in\Sigma(\bs{\mathfrak b}).
\]
This estimate also gives, by Dynkin's formula, $\sup_{0\le s\le1}P_s^bV_X(x)\le C_VV_X(x)$ with $C_V$ uniform in $b$.

The second point is a uniform small-set estimate. Fix a time $t_\star>0$ and a radius $R_1>R_0$. Since the diffusion matrix is the identity and, on $B_{R_1}$, the drifts in $\Sigma(\bs{\mathfrak b})$ are uniformly bounded by $b_0+b_1R_1$, the local lower Aronson bound for uniformly elliptic diffusions, equivalently the lower bound in the two-sided density estimate of~\cite{menozzi2020density} on compact space-time sets, yields a constant $\epsilon>0$ such that
\[
\inf_{b\in\Sigma(\bs{\mathfrak b})}\inf_{x,y\in B_{R_1}}p_{t_\star}^b(x;y)\ge \epsilon.
\]
If $\nu_{R_1}$ is normalized Lebesgue measure on $B_{R_1}$, then
\[
P_{t_\star}^b(x,A)\ge \epsilon |B_{R_1}|\,\nu_{R_1}(A\cap B_{R_1}),
\qquad x\in B_{R_1}.
\]
Thus $B_{R_1}$ is a small set for the $t_\star$-skeleton, with constants uniform in $b$. Combining this minorization with the preceding Lyapunov drift and applying the continuous-time Harris theorem, or equivalently the $t_\star$-skeleton theorem of Chapter~15 of~\cite{meyn2012markov} followed by the bound $\sup_{0\le s\le t_\star}P_s^bV_X\le C V_X$, gives constants $R_X<\infty$ and $c_X>0$, depending only on $\bs{\mathfrak b}$, such that
\[
\|P_t^b(x,\cdot)-\overline\mu^b\|_{\mathrm{TV}}
\le R_XV_X(x)e^{-c_Xt},
\qquad t\ge0,\quad b\in\Sigma(\bs{\mathfrak b}).
\]
This is a genuine total-variation estimate from point masses and is not a consequence of the $L^2(\overline\mu^b)$ Poincaré contraction~\eqref{eq:ergodicity}. Conversely,~\eqref{eq:ergodicity} is kept as the separate Poincaré consequence stated in Section~\ref{sec:model}; it follows uniformly from the same Lyapunov control together with local Poincaré inequalities for the invariant density, by the Lyapunov criterion for Poincaré inequalities quoted there from~\cite{bakry2014analysis}.

We now lift this estimate to the augmented chain. For $0<h\le h_0$, define the one-block Markov kernel
\[
Q_h^b(x,A)=\mathbb P_x^b\big((m^h_{0,h}(X),X_h)\in A\big),
\qquad A\in\mathcal B(\R^{2d}).
\]
By Lemma~\ref{lem:Z-Markov},
\[
\mathfrak P_h^b((u,x),A)=Q_h^b(x,A),
\qquad
\overline\pi^{h,b}=\overline\mu^bQ_h^b.
\]
Moreover, for $k\ge1$, the Markov property gives by induction
\[
\big(\mathfrak P_h^b\big)^k((u,x),\cdot)
=\big(P_{(k-1)h}^b(x,\cdot)\big)Q_h^b.
\]
Indeed, the first application of $\mathfrak P_h^b$ replaces the initial state by the law of one block started from $x$, and thereafter the only state variable transmitted to the next block is the endpoint of the preceding block. Since total variation contracts under post-composition by a Markov kernel,
\[
\|(\eta-\eta')Q_h^b\|_{\mathrm{TV}}\le \|\eta-\eta'\|_{\mathrm{TV}},
\]
we obtain, for $k\ge1$,
\[
\begin{aligned}
\big\|\big(\mathfrak P_h^b\big)^k((u,x),\cdot)-\overline\pi^{h,b}\big\|_{\mathrm{TV}}
&\le \|P_{(k-1)h}^b(x,\cdot)-\overline\mu^b\|_{\mathrm{TV}}\\
&\le R_XV_X(x)e^{-c_X(k-1)h}\\
&\le R_Xe^{c_Xh_0}V_X(x)e^{-c_Xkh}.
\end{aligned}
\]
For $k=0$,
\[
\|\delta_{(u,x)}-\overline\pi^{h,b}\|_{\mathrm{TV}}\le 2\le2V_X(x).
\]
Thus the claimed estimate holds for all $k\ge0$ after setting
\[
c=c_X,\qquad R=\max\{2,R_Xe^{c_Xh_0}\}.
\]
This also proves uniqueness of $\overline\pi^{h,b}$ and aperiodicity/geometric ergodicity of the augmented chain. 
Finally, the same lifting argument applies to the noisy lifted chain 
\begin{equation*}
	\Big(p^{-1} \sum_{\ell=0}^{p-1} Y_{kp + \ell,n}, X_{(k+1)p\Delta_n}\Big)_{k\in\mathbb{N}}.
\end{equation*}
The independent Gaussian preaveraging noise is an additional post-composition by a Markov kernel and therefore cannot increase total variation.

\section{Proof of Proposition~\ref{proposition:expansion_hidden_measure}}

Set
\[
\chi(u)=m([u,1]),\qquad 
Y_t^h=\int_0^t \chi(s/h)\,\mathrm dW_s,\qquad 0\le t\le h,
\]
so that $Y_h^h=m^h(W)$. Since $m$ is a probability measure, $|\chi|\le1$. We also write
\[
q_m=\int_0^1\chi(u)^2\,\mathrm du .
\]
Then $Y_h^h$ is a centred Gaussian vector with covariance matrix $h q_m I_d$.

Writing
\[
F_t(y)=(M_t(y)-1)\invar(y),
\]
we obtain the exact decomposition
\begin{equation*}
m^h(\invar)(x)
=
\mathbb E\left[\invar(x-Y_h^h)\right]
+
\mathbb E\left[F_h(x-Y_h^h)\right].
\end{equation*}
The two terms are treated separately below.

All remainder estimates below are understood locally in the spatial variable. Thus, for a compact set $K\subset\mathbb R^d$, the notation $O_K(h^a)$ means a quantity bounded by $C_Kh^a$, where $C_K$ may depend on $K$, $b$, $m$, and $\beta$, but not on $h\in(0,1)$. The desired conclusion is therefore
\[
\forall K\subset\mathbb R^d,\qquad
\sup_{0<h<1}\sup_{x\in K}|r_h(x)|<\infty .
\]
This is the form needed for the pointwise Richardson/deconvolution bias correction, since for a fixed evaluation point $x$ only finitely many shifted points around $x$ are used.

\subsection{Outline of the proof}

The first step of the proof is to use Girsanov's theorem to obtain an explicit representation of the density $m^{h}(\invar)$ of $m^{h}(X)$ in terms of $\invar$ and the martingale $(M_t^b(x))_{t\geq0}$ defined by 
\begin{equation*}
M_{t}^b(x)
= \exp \Big(\int_0^t  b(x +  W_s) \cdot \mathrm{d}W_s
- \frac{1}{2} \int_0^t \left| b(x +  W_s)\right|^2 \, \mathrm{d}s \Big).
\end{equation*}
Here $b=(b^{(1)},\ldots,b^{(d)})$ is vector-valued and
$
b(y+W_s)\cdot\mathrm dW_s
=\sum_{i=1}^d b^{(i)}(y+W_s)\,\mathrm dW_s^{(i)}$.
To identify the density $m^{h}(\invar)$ of $m^{h}(X)$, note that for every smooth function $\varphi$, we have
\begin{equation*}
\E_{\invar}\Big[\varphi\big(m^{h}(X)\big)\Big]
=
\int_{\R^d}
\E_y\Big[
\varphi(m^h(X))
\Big]
\, \invar(y) 
\, \mathrm{d}y.
\end{equation*}
We then use Girsanov's theorem to write
\begin{equation*}
    \E_y\Big[
\varphi(m^h(X))
\Big] = \E\Big[
\varphi(y + m^h(W))
M_{h}^{b}(y)
\Big].
\end{equation*}
This step is justified via Novikov’s criterion (see, e.g.,~\cite{karatzas1991brownian}, Lemma 5.14, p.~198) and Lemma~\ref{le:novi:1} below, which shows that the local martingale
$(\mathcal{E}^b_t(N))_t$ defined as 
\begin{equation*}
    \mathcal{E}^b_t(N) = \exp(-N^b_t - 1/2\langle N^b\rangle_t)
\end{equation*}
with $N_t^b = \int_0^tb(X_s) \cdot \mathrm{d}W_s$ is indeed a (non-local) martingale under $\mathbb{P}_x^b$. 
\begin{lemma}\label{le:novi:1}
There exists $\delta_0>0$ such that for any $T > 0$ and $x \in \R^d$, 
\begin{equation*}
\sup _{t \in[0, T-\delta]} \mathbb{E}_{x}^b\Big[\exp \Big(\frac{1}{2}\Big(\langle N^b\rangle_{t+\delta}-\langle N^b_t\rangle_t\Big)\Big)\Big] \leq C_1,
\end{equation*}
for every $0 \leq \delta \leq \delta_0$ and for some constant $C_1$ that depends on $\mathfrak{b}$, $x$ and $T$.
\end{lemma}

This lemma is proved in Section~\ref{sec:proof:le:novi:1}. Therefore, we obtain
\begin{equation*}
\E_{\invar}\Big[\varphi\big(m^{h}(X)\big)\Big]
=
\int_{\R^d}
\E\Big[
\varphi\big(y + m^h(W)\big)
M_{h}^{b}(y)
\Big]
\, \invar(y) 
\, dy
=
\E\Big[
\int_{\R^d}
\varphi\big(y + m^h(W)\big)
M_{h}^{b}(y)
\invar(y) 
\, dy
\Big].
\end{equation*}
The change of variables $x = y + m^h(W)$ then shows that
\begin{align*}
\E_{\invar}\Big[\varphi\big(m^{h}(X)\big)\Big]
&=
\E\Big[
\int_{\R^d}
\varphi(x)
M_{h}\big(x - m^h(W)\big)
\invar\big(x - m^h(W)\big) 
\, \mathrm{d}x
\Big]
\\
&=
\int_{\R^d}
\varphi(x)
\E\Big[
M_{h}\big(x - m^h(W)\big)
\invar\big(x - m^h(W)\big) 
\Big]
\, \mathrm{d}x.
\end{align*}
Thus, the distribution of $m^{h}(X)$ is absolutely continuous with respect to Lebesgue measure with density 
\begin{equation*}
m^{h}(\invar)(x) = 
\E\Big[
M_{h}\big(x - m^h(W)\big)
\invar\big(x - m^h(W)\big) 
\Big].
\end{equation*}
Note that the random variable $M_{h}\big(x - m^h(W)\big)$ is well defined since the mapping $(t, y) \mapsto M_t(y)$ is measurable (see e.g. Theorem~2 in~\cite{doleans1967integrales}). In fact, we can prove some stronger regularity results concerning $(t, y) \mapsto M_t(y)$. First, we introduce the following definition.

\begin{definition}
    For each $y$, consider a stochastic process $(X_t(y))_{t\geq 0}$. For $\gamma > 0$ and $p\geq1$, we say that the mapping $(t,y) \mapsto X_t(y)$ is of class $\mathcal{C}^{\gamma}$ in $y$ in $L^p$ if, almost surely, for each $t\geq 0$, the mapping $y \mapsto X_t(y)$ is of class $\mathcal{C}^{\lfloor\gamma\rfloor}$ in the classical sense and for all multi-indices $\bfgm$ with $|\bfgm|_1 = \lfloor\gamma\rfloor$, there exists a constant $K_p > 0$ such that for all $y,z \in \R^d$, we have
    \begin{equation*}
        \E\Big[\Big|\partial^{\bfgm}X_t(y) - \partial^{\bfgm}X_t(z)\Big|^{p}\Big] \leq K_p |y-z|^{p(\gamma - \lfloor\gamma\rfloor)}.
    \end{equation*}
\end{definition}

\begin{proposition}
\label{prop:reg:M}
    There exists a modification of $(t, y) \mapsto M_t(y)$, that we still denote $(t, y) \mapsto M_t(y)$, which is of class $\mathcal{C}^{\beta - 1}$ in $y$ in $L^p$, for $p\geq1$. Moreover, for each multi-index $\bfgm$ with $|\bfgm|_1 \geq 1$, we have
\begin{equation*}
    \partial^{\bfgm}M_t(y) = \int_0^t \sum_{\bfu  +\bfv = \bfgm} 
    {\bfgm\choose\bfu}
    \partial^{\bfu}b(y+W_s) \, \partial^{\bfv} M_s(y) \cdot \mathrm{d}W_s.
\end{equation*}
Here $\partial^{\bfu}b=(\partial^{\bfu}b^{(1)},\ldots,\partial^{\bfu}b^{(d)})$ is vector-valued, while $\partial^{\bfv}M_s(y)$ is scalar, and the dot denotes the Euclidean stochastic integral against the $d$-dimensional Brownian motion. The estimates used below are local in $y$: for every compact $K\subset\mathbb R^d$, the relevant $L^p$ and H\"older bounds are required uniformly for $y\in K$.
\end{proposition}

Proposition~\ref{prop:reg:M} is proved in Section~\ref{sec:proof:prop:reg:M}. As in Proposition~\ref{prop:reg:M}, we now consider a version of $(t, y) \mapsto M_t(y)$ that is of class $\mathcal{C}^{\beta - 1}$ in $y$ in $L^p$, for $p \geq 1$. Now fix $x \in \R^d$. The idea is then to use the regularity of the function $y \mapsto M_{h}(y) \invar(y)$ around $x$ 
and the fact that $m^h(W)$ is Gaussian to deduce Proposition~\ref{proposition:expansion_hidden_measure}. By definition, we know that $(M_t^b(y))_t$ satisfies
\begin{equation*}
M_t(y) = 1 + \int_0^t b(y+W_s) M_s(y) \cdot \mathrm{d}W_s,
\end{equation*}
for all $y \in \R^d$ and $t\geq0$. We then define the (random) function $y \mapsto F_t(y)$ by 
\begin{equation*}
F_t(y) = (M_t(y) - 1)\invar(y) = \int_0^t b(y+W_s) M_s(y) \invar(y) \cdot \mathrm{d}W_s.
\end{equation*}
With these notations, we can rewrite $m^h(\invar)(x)$ as
\begin{equation*}
m^{h}(\invar)(x) 
=
\E\big[
\invar(x - m^h(W)) 
\big]
+
\E\big[
F_h(x- m^h(W))
\big]
.
\end{equation*}
We study these two expectations separately. The first one is easier to study since $\invar$ is a deterministic $\mathcal{C}^\beta$ function. In fact, an application of Taylor's formula yields the following result, proved in Section~\ref{sec:proof:prop:expansion_hidden_measure:mu}.

\begin{proposition}
\label{prop:expansion_hidden_measure:mu}
Let $k = \lfloor \beta / 2 \rfloor$. There exist deterministic functions $u_1, \dots, u_k$ depending on $b$ and on $m$, but not on $h$, and functions $r_h$ such that, for every compact set $K\subset\mathbb R^d$, $\sup_{0<h<1}\sup_{x\in K}|r_h(x)|<\infty$, and
\begin{equation*}
\E\big[
\invar(x - m^h(W)) 
\big] = \wb\mu^b(x) + \sum_{i=1}^{\lfloor \beta / 2 \rfloor} u_i(x) h^i + r_h(x) h^{\beta / 2}.
\end{equation*}
\end{proposition}

Studying the expectation 
$\E[
F_h(x- m^h(W))
]$
is more subtle and requires some additional tools. Intuitively, the properties of $y \mapsto F_h(y)$ are similar to those of $\invar$ because $F_h(y)$ is $\mathcal{C}^{\beta - 1}$ in $L^p$. Applying Taylor's formula yields terms 
\begin{equation*}
    \E \Big[ U_t m^h(W)^\bfgm ]
\end{equation*}
for some stochastic process $U$ and some multi-index $\bfgm$. To conclude, we then need to use an iterative argument based on Itô's formula. Nevertheless, we shall obtain the following result.

\begin{proposition}
\label{prop:expansion_hidden_measure:F}
Let $k = \lfloor \beta / 2 \rfloor$. There exist deterministic functions $u_1, \dots, u_k$ depending on $b$ and on $m$, but not on $h$, and functions $r_h$ such that, for every compact set $K\subset\mathbb R^d$, $\sup_{0<h<1}\sup_{x\in K}|r_h(x)|<\infty$, and
\begin{equation*}
\E\big[
F_h(x- m^h(W))
\big]
= \sum_{i=1}^{\lfloor \beta / 2 \rfloor} u_i(x) h^i + r_h(x) h^{\beta / 2},
\end{equation*}
for all $x\in\bbR^d$.
\end{proposition}

Combining Propositions~\ref{prop:expansion_hidden_measure:mu} and~\ref{prop:expansion_hidden_measure:F} proves Proposition~\ref{proposition:expansion_hidden_measure}. The resulting coefficients are deterministic functions of $x$ and may depend on $b$ and $m$, but not on $h$ and the remainder satisfies
\begin{equation*}
\sup_{0<h<1}\sup_{x\in K}|r_h(x)|<\infty,
\end{equation*}
for every compact $K\subset\mathbb R^d$. In particular, for a fixed evaluation point $x$, the Richardson/deconvolution correction only evaluates finitely many shifted points in a compact neighborhood of $x$, so this local bound is sufficient for the pointwise bias correction. The rest of the proof is organized as follows. In Section~\ref{sec:proof:le:novi:1}, we prove Lemma~\ref{le:novi:1}. Then, in Section~\ref{sec:iterated_ito}, we show how bounds similar to those of Propositions~\ref{prop:expansion_hidden_measure:mu} and~\ref{prop:expansion_hidden_measure:F} can be obtained in a general setting for iterated Itô semimartingales. Propositions~\ref{prop:reg:M},~\ref{prop:expansion_hidden_measure:mu} and~\ref{prop:expansion_hidden_measure:F} are eventually proved in Sections~\ref{sec:proof:prop:reg:M},~\ref{sec:proof:prop:expansion_hidden_measure:mu} and~\ref{sec:proof:prop:expansion_hidden_measure:F} respectively.\\

\textbf{Additional notations.} We use the following notation. First, recall that 
\begin{equation*}
    m^h(W) = \int_0^h \chi(t/h) \, \mathrm{d}W_s,
\end{equation*}
where $\chi(u) = m([u,1])$ is defined in Section~\ref{sec:hidden:setting}. Then for $p \in \{1, 2\}$ and $j \geq 0$, we define the function $\chi^{(j)}_{(e)}$ by
\begin{equation*}
\chi^{(j)}_{(p)}(t) = \int_0^t s^j \chi(s)^{e} \, \mathrm{d}s.
\end{equation*}
We then define iteratively the function $\chi^{(\mathbf{j})}_{(\mathbf{e})}(t)$ for $\bfp = (p_1, \dots, p_m) \in \{1, 2\}^m$, $\bfj = (j_1, \dots, j_m) \in \mathbb{N}^m$ (with $m \geq 1$) by
\begin{equation*}
\chi^{(\mathbf{j})}_{(\mathbf{p})}(t) = \int_0^t t^{j_m} \chi^{(j_1, \dots, j_{m-1})}_{(p_1, \dots, p_{m-1})}(t)^{p_m} \, \mathrm{d}t.
\end{equation*}
Hiding the dependence on $h$, we write 
\begin{equation*}
\wb \chi(t) = \chi(t/h),
\end{equation*}
and similarly recursively define $\wb \chi^{(\mathbf{j})}_{(\mathbf{p})}$ using $\wb \chi$ in place of $\chi$. We now introduce the process $(Y_t)_{t\geq 0}$ by
\begin{equation*}
    Y_t = \int_0^t \wb \chi(s) \, \mathrm{d}W_s.
\end{equation*}
For $m \geq 1$ and $k \geq 1$, denote by $\bfJ_{m, r}$ the set of multi-indices $\bfj = (j_1, \dots, j_m)$ with $j_1, \dots, j_m \geq 0$ and $m + |\bfj|_1 \leq r$. We also denote by $\bfE_m$ the set of multi-indices $\bfp = (p_1, \dots, p_m)$ with $p_1, \dots, p_m \in \{1, 2\}$. 
Additionally, we denote by $(\bfe_1, \dots, \bfe_d)$ the canonical basis of $\bbR^d$.

\subsection{Proof of Lemma~\ref{le:novi:1}}
\label{sec:proof:le:novi:1}

Using Assumption~\ref{assumption:boundedness}, we get that for any $t\in [0,T]$,
	\begin{equation*}
		|X_t| \leq |X_0| + b_0T + |b|_{\mathrm{lip}}\int_0^t |X_s| \, \mathrm{d}s + \sup_{s\in [0,T]}|B_s|.,
	\end{equation*}
so that by Gronwall's lemma, for all $p \geq 1$ and $0 \leq t \leq T$, one has
\begin{equation}\label{eq:momentbound}
		\E_x^b\big[|X_t|^{2p}\big] \leq 3^{2p-1}e^{2pt|b|_{\mathrm{lip}}}(|x|^{2p} + T^p + (b_0T)^{2p}).
	\end{equation}
Moreover, for any $\delta > 0$ and $t \in [0,T-\delta]$, using Jensen's inequality,
	\begin{equation*}
		\E_x^{b}\Big[ \exp\Big(\frac{1}{2}\int_{t}^{t+\delta}|b(X_s)|^2 \, \mathrm{d}s \Big)\Big] \leq \frac{1}{\delta} \int_t^{t+\delta} \E_x^b\Big[\exp\Big(\frac{\delta}{2} |b(X_s)|^2 \Big)\Big] \, \mathrm{d}s.
	\end{equation*}
Using Assumption~\ref{assumption:boundedness}, we get the following bound
	\begin{align*}
		\E_x^{b}\Big[ \exp\Big(\frac{1}{2}\int_{t}^{t+\delta}|b(X_s)|^2 \, \mathrm{d}s \Big)\Big] & \leq \frac{1}{\delta} \int_t^{t+\delta} \E_x^b\Big[ \exp\Big( \frac{\delta}{2}(|b_0|^2 + |b|_{\mathrm{lip}}^2|X_s|^2 )\Big)\Big] \, \mathrm{d}s \\
		& \leq \frac{e^{\delta|b_0|^2}}{\delta} \int_t^{t+\delta}\E_x^b\Big[ \exp\Big( \frac{\delta}{2}|b|_{\mathrm{lip}}|X_s|^2\Big)\Big] \, \mathrm{d}s.
	\end{align*}
	Now, for $s \in [t, t+\delta]$, we have 
	\begin{equation*}
		\E_x^b\Big[ \exp\Big( \frac{\delta}{2}|b|_{\mathrm{lip}}|X_s|^2\Big)\Big]  = 1 + \sum_{k\geq 1}\Big(\frac{\delta |b|_{\mathrm{lip}}}{2}\Big)^{k}\frac{\E_x^b[|X_s|^{2k}]}{k!}. 
	\end{equation*}
	From~\eqref{eq:momentbound}, we get that $\E_x^b[ \exp( \frac{\delta}{2}|b(X_s)|^2)]$ is finite for all $\delta$ small enough, which proves Lemma~\ref{le:novi:1}.

\subsection{Expansion of the moments of Iterated Itô semimartingales}
\label{sec:iterated_ito}

In this section, we introduce some recursive classes of processes which are used in our computations. We consider a general $\beta > 0$, and not necessarily the one used in the statement of Proposition~\ref{proposition:expansion_hidden_measure}. We now define the following iterated semimartingale space.

\begin{definition}
Let $\beta > 0$. We define iteratively the class $\bbI^{\beta}$ of stochastic processes $U$ satisfying
\begin{equation*}
U_t = U_0 + \int_0^t A_s \, \mathrm{d}s + \int_0^t B_s \, \mathrm{d}W_s + R_s t^{\beta/2},
\end{equation*}
for all $t \geq 0$ and for some processes $A \in \bbI^{\beta - 2}$ and $B = (B^{(1)}, \dots, B^{(d)})$ with $B^{(i)} \in \bbI^{\beta - 1}$ for all $1 \leq i \leq d$ and for some uniformly bounded process $R$. Here, use the convention that when $\beta \leq 0$, then $\bbI^{\beta}$ contains all bounded processes.
If $U=U(y)$ is a random field and $K\subset\mathbb R^d$, we write $U(y)\in\bbI_K^\beta$ when the same recursive representation and all $L^p$ bounds hold uniformly for $y\in K$; constants may depend on $K$ but not on $h$. The class $\bbI_K^\beta$ is stable under the sums, products, stochastic integrations, time integrations, and spatial differentiations used below.
\end{definition}

We now state the following key lemma.

\begin{lemma}
\label{lem:k:iterated_ito}
        Let $\beta \geq 0$ and let $\bfgm = (\gamma_1, \dots, \gamma_d)$ be a multi-index with $|\bfgm|_1 \geq 0$. If $U \in \bbI^{\beta}$, then there exists some coefficients $\{u_j\}$ and $\{u_{\bfj, \bfp}\}$, which are $\calF_0$-measurable such that
        \begin{equation*}
            \E[U_t Y_t^{\bfgm}] = 
            \sum_{j=0}^{r} u_j t^j + 
        \sum_{m \geq 0} \sum_{\bfj \in \wb \bfJ_{m, r}} \sum_{\bfp \in \wb \bfP_{m}} u_{\bfj, \bfp} \wb \chi^{(\bfj)}_{(\bfp)}(t) + O(t^{(\beta + |\bfgm|_1)/2}),
        \end{equation*}
        where $r = \lfloor (\beta + |\bfgm|_1)/2 \rfloor$. Moreover, if $|\bfgm|_1 > 0$, then $u_0=0$. 
        If $U=U(y)\in\bbI_K^\beta$, the same expansion holds uniformly for $y\in K$, with $O_K(t^{(\beta+|\bfgm|_1)/2})$ remainders and deterministic coefficients depending on $y$ but not on $h$.
\end{lemma}

The proof of Lemma~\ref{lem:k:iterated_ito} relies on an induction on both $| \bfgm |_1$ and $\lfloor \beta \rfloor$. More precisely, the proof of Lemma~\ref{lem:k:iterated_ito} proceeds in two steps: 
\begin{itemize}
    \item \textbf{Step 1.} When $| \bfgm |_1 = 0$ (so that $Y_t^\bfgm = 1$), one shows that for all $\beta \geq 0$ and $U \in \bbI^{\beta}$, there exist coefficients $u_0, \dots, u_{k}$ (with $k = \lfloor \beta/2 \rfloor$) which are $\calF_0$-measurable such that
    \begin{equation*}
        \E[U_t] = \sum_{i=0}^{k} u_i t^i + O\big(t^{\beta/2}\big).
    \end{equation*}
    
    \item \textbf{Step 2.} When $|\bfgm|_1 \geq 1$, one uses induction on $|\bfgm|_1$ (assuming the result for multi-indices of lower order) to deduce the claim for $|\bfgm|_1$. 
\end{itemize}

Each of these steps is proved via an induction on $\lfloor \beta \rfloor$. We also need the following bound, which can be proved by induction on $m$. 
\begin{lemma}
\label{lem:control:chibar}
    For $m \geq 1$, if $\bfj = (j_1, \dots, j_m)$ with $j_1, \dots, j_m \geq 0$ and $\bfp = (p_1, \dots, p_m)$ with $p_1, \dots, p_m \in \{1, 2\}$, then for $t \geq 0$, we have
    \begin{equation*}
        \wb \chi^{(\bfj)}_{(\bfp)}(t) = h^{m + |\bfj|_1} \chi^{(\mathbf{j})}_{(\mathbf{p})}(t / h).
    \end{equation*}
\end{lemma}

\begin{proof}[Proof of {Step 1.}]
As mentioned previously, the proof of the first step relies on an induction on $\lfloor \beta \rfloor$.  By definition, we can write 
\begin{equation*}
U_t = U_0 + \int_0^t A_s \, \mathrm{d}s + \int_0^t B_s \, \mathrm{d}W_s + R_t t^{\beta/2},
\end{equation*}
for all $t \geq 0$, where $A \in \bbI^{\beta - 2}$, $B = (B^{(1)}, \dots, B^{(d)})$ with $B^{(i)} \in \bbI^{\beta - 1}$ for all $1 \leq i \leq d$, and $R_t$ is uniformly bounded. Hence, 
\begin{equation*}
    \E[U_t] = U_0 + \int_0^t \E[A_s] \, \mathrm{d}s + O\big(t^{\beta/2}\big).
\end{equation*}
When $\beta \leq 2$, $A$ is bounded and thus we get $\E[U_t] = U_0 + O(t^{\beta/2})$. For $\beta \ge 2$, if we assume by induction that for all $\beta' < \beta$ the expansion holds, then
\begin{equation*}
    \E[A_s] = \sum_{i=0}^{k - 1} a_i s^i + O(s^{\beta/2 - 1}).
\end{equation*}
Integrating yields
\begin{equation*}
    \E[U_t] = U_0 \sum_{i=1}^{k} \frac{a_{i-1}}{i} s^i + O(t^{\beta/2})
\end{equation*}
and we can conclude.
\end{proof}

\begin{proof}[Proof of {Step 2}]
By definition, we can write
\[
U_t=\wt U_t+R_t t^{\beta/2},\qquad
\wt U_t=U_0+\int_0^t A_s\,\mathrm ds+\sum_{i=1}^d\int_0^t B_s^{(i)}\,\mathrm dW_s^{(i)},
\]
where $A\in\bbI^{\beta-2}$ and $B^{(i)}\in\bbI^{\beta-1}$. Since
\begin{equation*}
\E[|Y_t^\bfgm|^p]\lesssim t^{p|\bfgm|_1/2},
\end{equation*}
the remainder contributes $O(t^{(\beta+|\bfgm|_1)/2})$. It is therefore enough to work with $\wt U$, and we write it again as $U$. Since $dY_t^{(i)}=\wb\chi(t)\,\mathrm dW_t^{(i)}$, Itô's formula gives
\[
d(Y_t^\bfgm)
=
\sum_{\gamma_i\ge1}\gamma_iY_t^{\bfgm-\bfe_i}\wb\chi(t)\,\mathrm dW_t^{(i)}
+
\frac12\sum_{\gamma_i\ge2}\gamma_i(\gamma_i-1)
Y_t^{\bfgm-2\bfe_i}\wb\chi(t)^2\,\mathrm dt .
\]
Thus
\begin{align*}
d(U_tY_t^\bfgm)
={}&A_tY_t^\bfgm\,\mathrm dt
+Y_t^\bfgm B_t\cdot\mathrm dW_t
+U_t\sum_{\gamma_i\ge1}\gamma_iY_t^{\bfgm-\bfe_i}\wb\chi(t)\,\mathrm dW_t^{(i)}
\\
&+\frac12U_t\sum_{\gamma_i\ge2}\gamma_i(\gamma_i-1)
Y_t^{\bfgm-2\bfe_i}\wb\chi(t)^2\,\mathrm dt
+\sum_{\gamma_i\ge1}\gamma_i B_t^{(i)}
Y_t^{\bfgm-\bfe_i}\wb\chi(t)\,\mathrm dt .
\end{align*}
The last term is the cross-variation term and has no factor $1/2$. Taking expectations, after the usual localization of the stochastic integrals, gives
\begin{align*}
\E[U_tY_t^\bfgm]
={}&
\int_0^t\E[A_sY_s^\bfgm]\,\mathrm ds
+\frac12\sum_{\gamma_i\ge2}\gamma_i(\gamma_i-1)
\int_0^t\wb\chi(s)^2\E[U_sY_s^{\bfgm-2\bfe_i}]\,\mathrm ds
\\
&+\sum_{\gamma_i\ge1}\gamma_i
\int_0^t\wb\chi(s)\E[B_s^{(i)}Y_s^{\bfgm-\bfe_i}]\,\mathrm ds .
\end{align*}
We now argue by induction on the total index $\beta+|\bfgm|_1$, with the case $|\bfgm|_1=0$ supplied by Step 1. In the first integral, the process $A$ has index $\beta-2$; in the second integral the power of $Y$ has been reduced by $2$; and in the third integral the process $B^{(i)}$ has index $\beta-1$ while the power of $Y$ has been reduced by $1$. Hence every expectation under the integral sign is covered by the induction hypothesis with total index $\beta+|\bfgm|_1-2$. Integration in time raises every power by one, and the factors $\wb\chi$ and $\wb\chi^2$ only produce the deterministic functions $\wb\chi^{(\bfj)}_{(\bfp)}$ already introduced above. The integrated remainder is therefore $O(t^{(\beta+|\bfgm|_1)/2})$, and the displayed expansion in Lemma~\ref{lem:k:iterated_ito} follows. If $|\bfgm|_1>0$, the initial term is zero because $Y_0^\bfgm=0$.
\end{proof}

\subsection{Proof of Proposition~\ref{prop:reg:M}}
\label{sec:proof:prop:reg:M}

We first prove the following auxiliary result.
\begin{lemma}
\label{lem:proof:reg:M}
    For each $y$, consider a stochastic process $(A_t(y))_t$ and define $(X_t(y))_t$ by
    \begin{equation*}
        X_t(y) = \int_0^t A_s(y) \, \mathrm{d}W_s.
    \end{equation*}
    Suppose in addition that $(t,y) \mapsto A_t(y)$ is of class $\mathcal{C}^{\beta}$ in $y$ for some integer $\beta \geq 1$. Then there exists a modification of $(X_t(y))_t$ which is of class $\mathcal{C}^{\beta}$ in $y$ and moreover for all multi-indices $\bfgm$ with $|\bfgm|_1 \leq \lfloor \beta \rfloor$, we have
    \begin{equation*}
        \partial^\bfgm X_t(y) = \int_0^t \partial^\bfgm  A_s(y) \, \mathrm{d}W_s.
    \end{equation*}
\end{lemma}
\begin{proof}
Since all continuous modifications of $(t, y) \mapsto X_t(y)$ are the same, it is enough to prove that the same result holds for $(t, y) \in [0, T] \times [-K, K]^d$ for some $T > 0$ and some $K > 0$. Moreover, by localization, we only need to prove the result for $(t,y) \mapsto I_{t \wedge \tau}(y)$ where $\tau$ is a stopping time. Therefore, there is no loss of generality in assuming that $A$ and all its derivatives are uniformly bounded by a deterministic constant $K'$. Note also that it is enough to prove the result for $m=1$. The result for a general $m$ follows directly by induction and the fact that all continuous modifications of a stochastic process are the same.\\

Let $0 \leq t_1 \leq t_2 \leq T$ and $(y_1, y_2) \in [-K,K]^d \times [-K,K]^d$. By the Burkholder--Davis--Gundy Inequality, we have for all $p \geq 2$
\begin{align*}
\E [(X_{t_1}(y_1) - X_{t_2}(y_2))^p]
&\lesssim
\E\Big[\Big(\int_0^{t_1} (A_s(y_1) - A_s(y_2))^2  \, \mathrm{d}s + \int_{t_1}^{t_2} A_s(y_2)^2 \, \mathrm{d}s \Big)^{p/2} \Big]
\\
&\lesssim
\E\Big[\Big(\int_0^{t_1 \wedge \tau} (A_s(y_1) - A_s(y_2))^2  \, \mathrm{d}s\Big)^{p/2} + \Big(\int_{t_1 \wedge \tau}^{t_2} A_s(y_2)^2 \, \mathrm{d}s\Big)^{p/2} \Big].
\end{align*}
Using that $A_s$ is bounded and differentiable with bounded derivatives, we know that $|A_s(y_1) - A_s(y_2)| \lesssim |y_1 - y_2|$ and $|A_s(y_2)| \lesssim 1$. Therefore, we get
\begin{equation*}
    \E [(X_{t_1}(y_1) - X_{t_2}(y_2))^p] \lesssim |t_1 - t_2|^{p/2} + |y_1 - y_2|^{p}.
\end{equation*}
Since $p>1$ is arbitrary, multidimensional Kolmogorov's continuity theorem (see for instance Corollary 5-2.3.1 in~\cite{khoshnevisan2006multiparameter}) ensures that there exists a continuous version of $(t,y) \mapsto X_t(y)$. We denote this version by $\wt X$. \\

We then prove that $y \mapsto \wt X_t(y)$ admits weak derivatives in each direction at each point. Let $1 \leq j \leq d$. We define for each $y = (y_1, \dots, y_d) \in \mathbb{R}^d$ and each $u \in \mathbb{R}$
\begin{equation*}
Y_{j}(y,u) = (y_1, \dots, y_{j-1}, u, y_{j+1}, \dots, y_d).
\end{equation*}
For all $y \in \mathbb{R}^d$, we then define
\begin{equation*}
J_{t}(y) = \int_0^t \partial_j A_s(y) \cdot \mathrm{d}W_s.
\end{equation*}
Following the same approach as previously, we can find a version of $(t,y) \mapsto J_t(y)$ which is continuous. We denote this version by $(t,y) \mapsto \wt J_t(y)$ and we set
\begin{equation*}
\wh X_{t}(y) = \wt X_{t}(Y_{j}(y,0)) + \int_0^{y_j} 
\wt J_{t}(Y_{j}(y,u)) \, \mathrm{d}u.
\end{equation*}
First note that $(t, y) \mapsto \wh I_{t}(y)$ is continuous because $(t, y) \mapsto \wt J_{t}(y)$ is continuous and moreover, $\wh X_{t}$ admits a derivative in the $j$-th coordinate given by 
\begin{equation*}
\partial_j \wh X_{t}(y) = \wt J_{t}(y).
\end{equation*}
We now prove that for all $t$ and $y$, $\wh X_{t}(y) = X_t(y)$ almost surely. Note that when $t$ and $y$ are fixed, we have
\begin{equation*}
\int_0^{y_j} \wt J_{t}(Y_{j}(y,u)) \, \mathrm{d}u
=
\int_0^{y_j} J_{t}(Y_{j}(y,u)) \, \mathrm{d}u
\end{equation*}
almost surely. This follows from the fact that 
\begin{equation*}
\E\Big[\int_0^t |\wt J_{t}(Y_{j}(y,u)) - J_{t}(Y_{j}(y,u))|\, \mathrm{d}u\Big] = \int_0^t \E\big[ |\wt J_{t}(Y_{j}(y,u)) - J_{t}(Y_{j}(y,u))| \big]\, \mathrm{d}u = 0.
\end{equation*}
Then, we have
\begin{align*}
\wh X_{t}(y) 
&= \wt X_{t}(Y_{j}(y,0)) + \int_0^{y_j} 
 J_{t}(Y_{j}(y,u))
\, \mathrm{d}u
\\
&= \int_0^t X_s(Y_{j}(y,0))\cdot \mathrm{d}W_s + \int_0^{y_j} 
 \int_0^t \partial_j X_s(Y_{j}(y,u)) \cdot \mathrm{d}W_s \mathrm{d}u
\\
&= \int_0^t \Big(A_s(Y_{j}(y,0)) + \int_0^{y_j} 
 \int_0^t \partial_j A_s(Y_{j}(y,u)) \mathrm{d}u\Big) \cdot \mathrm{d}W_s
 \\
&= X_{t}(y)  
\end{align*}
almost surely. Note that the interchange of the stochastic integral and the Lebesgue integral is justified by the following Fubini's theorem for stochastic integrals.
\begin{theorem}[Theorem 2.2 in~\cite{verrar2012fubini}]
	Let $(X, \Sigma, \mu)$ be a  $\sigma$-finite measure space. Let $S=M+A$ be a continuous semimartingale. Let $\psi: X \times [0, T] \rightarrow \mathrm{R}$ be progressively measurable and such that we have
\begin{equation*}
    \int_X\left(\int_0^T|\psi(x, t)|^2 \mathrm{d}\langle M\rangle_t \right)^{\frac{1}{2}} \mathrm{d} \mu(x)<\infty, 
    \quad \text{ and }\quad \int_X \int_0^T|\psi(x, t)| \mathrm{d}A_t \mathrm{d} \mu(x) < \infty
\end{equation*}
almost surely. Then, for all $t \in[0, T]$, one has that almost surely
$$
\int_X \int_0^t \psi(r,x) \: \mathrm{d}S_r \: \mathrm{d} \mu(x)=\int_0^t \int_X \psi(r,x) \: \mathrm{d} \mu(x) \: \mathrm{d}S_r.
$$
\end{theorem}

Thus $(t, y) \mapsto \wh X_{t}(y)$ is a continuous modification of  $(t, y) \mapsto I_{t}(y)$, which ensures in particular that almost surely, for all $t$ and $y$, $\wh I_{t}(y) = \wt I_{t}(y)$. By construction, this ensures that $\wt I_{t}$ has a derivative in the $j$-th coordinate. We eventually obtain that $\wt X_{t}$ is of class $C^1$ by applying the following Kolmogorov--Chentsov theorem, due to~\cite{andreev2014kolmogorov}.
\begin{theorem}[Theorem 3.1 of~\cite{andreev2014kolmogorov}]
    Let $D \subset \mathbb{R}^d$ be a bounded domain of cone type and let $X: \Omega \times D \rightarrow \mathbb{R}$ be a random field on $D$. Assume that there exist $m \geq1$, $p>1$, $\epsilon \in(0, p]$, and $C>0$ such that the weak derivatives $\partial^\beta X$ are in $L^p(\Omega \times D)$ and
\begin{equation*}
    \mathbb{E}\left(\left|\partial^\beta X(x)-\partial^\beta X(y)\right|^p\right) \leq C|x-y|^{d+\epsilon},
\end{equation*}
for all $x, y \in D$ and any multi-index $\beta \in \mathbb{N}^n$ with $|\beta| \leq d$. Then $X$ has a modification that is locally of class $\bar{C}^t$ for all $t<d+\min \{\epsilon / p, 1-d / p\}$.
\end{theorem}

\end{proof}

We now show how Lemma~\ref{lem:proof:reg:M} directly implies Proposition~\ref{prop:reg:M}. Recall that $(M_t(y))_t$ is defined by
\begin{equation*}
M_t(y) =
\exp \Big(
\int_0^t  b(y +  W_s) \cdot \mathrm{d}W_s
- \frac{1}{2} \int_0^t \left| b(y +  W_s)\right|^2 \, \mathrm{d}s 
\Big).
\end{equation*}
Since the exponential function is smooth, the fact that $(t,y) \mapsto M_t(y)$ is of class $\mathcal{C}^{\beta - 1}$ directly follows from the regularity of the processes
\begin{equation*}
(t,y) \mapsto \int_0^t  b(y +  W_s) \cdot \mathrm{d}W_s
\quad \text{ and } \quad 
(t,y) \mapsto \int_0^t \left| b(y +  W_s)\right|^2 \, \mathrm{d}s.
\end{equation*}
Using Lemma~\ref{lem:proof:reg:M} and the fact that $b$ is of class $\mathcal{C}^{\beta - 1}$, we see that the first one is $\mathcal{C}^{\beta - 1}$. For the second one, we just use the usual differentiation under the integral sign for Lebesgue integral with parameters.
Therefore, there exists a version  of $(t, y) \mapsto M_t(y)$ which is of class $\mathcal{C}^{\beta - 1}$ in $y$ in $L^p$, for $p\geq1$. We denote this version by $(t, y) \mapsto \wt M_t(y)$. It remains to prove that
\begin{equation*}
    \partial^{\bfgm} \wt M_t(y) = \int_0^t \sum_{\bfu  +\bfv = \bfgm}
    {\bfgm \choose \bfu}
    \partial^{\bfu}b(y+W_s) \partial^{\bfv} \wt M_s(y) \cdot \mathrm{d}W_s.
\end{equation*}
In fact, we know from Lemma~\ref{lem:proof:reg:M} that $(t,y) \mapsto U_t(y)$ defined as
\begin{equation*}
    U_t(y) = \int_0^t b(y+W_s) \wt M_s(y) \, \mathrm{d}W_s
\end{equation*}
is of class $\mathcal{C}^{\beta - 1}$ in $y$ in $L^p$ and its derivatives are given by
\begin{equation*}
    \partial^{\bfgm} U_t(y) = \int_0^t \sum_{\bfu  +\bfv = \bfgm} {\bfgm \choose \bfu} \partial^{\bfu}b(y+W_s) \partial^{\bfv} \wt M_s(y) \cdot \mathrm{d}W_s.
\end{equation*}
Moreover, for each $t,y$, we have
\begin{align*}
    \E[(U_t(y) - M_t(y))^2]
    &=
    \E\Big[\Big(\int_0^t \big(b(y+W_s) \wt M_s(y) - b(y+W_s) M_s(y)\big) \, \mathrm{d}W_s\Big)^2\Big]
\\
    &=
    \E\Big[\int_0^t b(y+W_s)^2\big( \wt M_s(y) - M_s(y)\big)^2 \, \mathrm{d}s \Big]
    \\
    &=
    \int_0^t \E\Big[b(y+W_s)^2\big( \wt M_s(y) - M_s(y)\big)^2 \Big]\, \mathrm{d}s
    \\
    &=
    0,
\end{align*}
and thus $U_t(y) = M_t(y)$ almost surely, ensuring that $(t,y) \mapsto U_t(y)$ is a modification of $(t,y) \mapsto M_t(y)$. Since $(t,y) \mapsto U_t(y)$ is continuous, we deduce that almost surely, $U_t(y) = \wt M_t(y)$ for all $t,y$ and thus we can conclude.

\subsection{Proof of Proposition~\ref{prop:expansion_hidden_measure:mu}}
\label{sec:proof:prop:expansion_hidden_measure:mu}

Fix a compact set $K\subset\mathbb R^d$. We have
\[
m^h(W)=Y_h^h=h^{1/2}Z,\qquad Z\sim N(0,q_mI_d),
\qquad q_m=\int_0^1\chi(u)^2\,\mathrm du .
\]
Let $n_\mu=\lceil\beta\rceil-1$. Taylor's formula for $\invar$ at the point $x$ gives
\[
\invar(x-h^{1/2}Z)
=
\sum_{|\bfgm|_1\le n_\mu}
\frac{(-1)^{|\bfgm|_1}}{\bfgm!}
\partial^\bfgm\invar(x)\,h^{|\bfgm|_1/2}Z^\bfgm
+R_h(x,Z).
\]
On the event $\{h^{1/2}|Z|\le1\}$, the local isotropic $C^\beta$ regularity gives $|R_h(x,Z)|\le C_Kh^{\beta/2}|Z|^\beta$ uniformly for $x\in K$.
On the complementary event, boundedness of $\invar$ and boundedness of the finitely many derivatives on $K$ give a contribution $O_K(h^N)$ for every $N>0$, by Gaussian tails. Hence $\E[R_h(x,Z)]=O_K(h^{\beta/2})$.

Moreover, $\E[Z^\bfgm]=0$ if at least one component $\gamma_i$ is odd. Thus only even multi-indices contribute. For $1\le j\le k=\lfloor\beta/2\rfloor$, set
\[
u_j(x)=
\sum_{\substack{|\bfgm|_1=2j\\ \gamma_i\ \mathrm{even}\ \forall i\\ 2j\le n_\mu}}
\frac{1}{\bfgm!}\partial^\bfgm\invar(x)\,\E[Z^\bfgm],
\]
with the convention that the empty sum is zero. Then
\[
\E\big[\invar(x-m^h(W))\big]
=
\invar(x)+\sum_{j=1}^{k}u_j(x)h^j+r_h(x)h^{\beta/2},
\]
where $\sup_{0<h<1}\sup_{x\in K}|r_h(x)|<\infty$. The coefficients depend on $m$ through $q_m$ and the Gaussian moments of $Z$, but they do not depend on $h$.

\subsection{Proof of Proposition~\ref{prop:expansion_hidden_measure:F}}
\label{sec:proof:prop:expansion_hidden_measure:F}

Recall from Proposition~\ref{prop:reg:M} that, under the weighted local H\"older and polynomial-growth conditions in Assumption~\ref{assumption:isotropic}, $(t,y)\mapsto M_t(y)$ is locally of class $\mathcal C^{\beta-1}$ in $y$ in $L^p$. Hence $F_t(y)=(M_t(y)-1)\invar(y)$ has the same spatial regularity. Let
\[
n_F=\lceil\beta-1\rceil-1.
\]
Taylor's formula gives
\[
F_h(x-Y_h)
=
\sum_{|\bfgm|_1\le n_F}
\frac{(-1)^{|\bfgm|_1}}{\bfgm!}
\partial^\bfgm F_h(x)Y_h^\bfgm
+\mathcal R_h(x).
\]
The term $|\bfgm|_1=0$ vanishes in expectation, since
\[
\E[F_h(x)]
=
\E\left[\int_0^h b(x+W_s)M_s(x)\invar(x)\cdot\mathrm dW_s\right]
=0
\]
after localization and $L^2$ convergence. Recall that $k=\lfloor\beta/2\rfloor$. The nonzero Taylor terms and the remainder are controlled by the following two lemmas.

\begin{lemma}
    \label{lem:expansion_hidden_measure:F:bfgm}
    For every compact set $K\subset\mathbb R^d$ and every $\bfgm$ with $1 \leq |\bfgm|_1 \leq n_F$, there exist deterministic functions $u_1, \dots, u_k$ depending on $b$ and $m$, but not on $h$, and functions $r_h$ such that $\sup_{0<h<1}\sup_{y\in K}|r_h(y)|<\infty$ and
    \begin{equation*}
    \E\Big[
    \partial^\bfgm F_h(y) 
    Y_h^\bfgm
    \Big]
    =
    \sum_{j=1}^k
    u_j(y) h^{j}
    + r_h(y) h^{\beta/2}
\end{equation*}
\end{lemma}

\begin{lemma}
    \label{lem:expansion_hidden_measure:F:beta}
    For every compact set $K\subset\mathbb R^d$, the Taylor remainder above satisfies, uniformly for $y\in K$,
\begin{equation*}
\E\Big[\mathcal R_h(y)\Big]
=
O_K(h^{\beta/2}).
\end{equation*}
\end{lemma}

The proofs of two lemmas rely on the following regularity result.
\begin{lemma}
\label{lem:I:M}
    Under Assumption~\ref{assumption:isotropic}, for every compact set $K\subset\mathbb R^d$ and each multi-index $\bfgm$ with $0 \leq |\bfgm|_1 \leq \lfloor \beta-1 \rfloor$, the family $(\partial^{\bfgm} M_t(y))_{0\le t\le h,\ y\in K}$ is of class $\bbI_K^{\beta - |\bfgm|_1}$.
\end{lemma}

\begin{proof}[Proof of Lemma~\ref{lem:I:M}]
    Fix $K\subset\mathbb R^d$. We first prove this result when $|\bfgm|_1 = 0$. Since $(M_t(y))_t$ has the required local moment bounds uniformly for $y\in K$, it is of class $\bbI_K^{0}$. Now suppose that $(M_t(y))_t$ is of class $\bbI_K^{\gamma}$ for some $\gamma > 0$. We know that
    \begin{equation*}
        M_t(y) = 1 + \int_0^t b(y+W_s) M_s(y) \cdot \mathrm{d}W_s,
    \end{equation*}
    and the process $(b(y+W_t))_t$ is $\bbI_K^{\beta - 1}$. Therefore, $(b(y+W_t) M_t(y))_t$ is $\bbI_K^{(\beta - 1)\wedge \gamma}$ and thus $(M_t(y))_t$ is  $\bbI_K^{\beta\wedge (\gamma+1)}$. Repeating this argument implies that $(M_t(y))_t$ is in the class $\bbI_K^{\beta}$.\\

    We now proceed to an induction on $|\bfgm|_1$ to prove this result when $|\bfgm|_1 > 0$. In that case, we have
    \begin{equation*}
        \partial^{\bfgm} M_t(y) 
        = \sum_{\bfu + \bfv = \bfgm}
        {\bfgm\choose\bfu}
        \int_0^t \partial^{\bfu} b(y+W_s) \partial^{\bfv}M_s(y) \cdot \mathrm{d}W_s
    \end{equation*}
    by Proposition~\ref{prop:reg:M}. For $\bfv \neq \bfgm$, we already have (by induction on $|\bfgm|_1$) that $(\partial^{\bfv}M_s(y))_t$ is $\bbI_K^{\beta - |\bfv|_1}$. Moreover, $(\partial^{\bfu} b(y+W_t))_t$ is $\bbI_K^{\beta - |\bfu|_1 - 1}$. Thus
    \begin{equation*}
        \Big( \int_0^t \partial^{\bfu} b(y+W_s) \partial^{\bfv}M_s(y) \, \mathrm{d}W_s \Big)_t
    \end{equation*}
    is $\bbI_K^{(\beta - |\bfv|_1 + 1) \wedge (\beta - |\bfu|_1)}$. Since $|\bfu|_1 \leq |\bfgm|_1$ and $|\bfv|_1 < |\bfgm|_1$, we deduce that this integral is $\bbI_K^{\beta - |\bfgm|_1}$. Then we get
    \begin{equation*}
        \partial^{\bfgm} M_t(y) 
        =
        \int_0^t b(y+W_s) \partial^{\bfgm}M_s(y) \cdot \mathrm{d}W_s + \text{an element of } \bbI_K^{\beta - |\bfgm|_1}
    \end{equation*}
    and we conclude exactly as in the case $\bfgm = 0$.
\end{proof}

\begin{proof}[Proof of Lemma~\ref{lem:expansion_hidden_measure:F:bfgm}]
Fix $K\subset\mathbb R^d$ and let $\bfgm$ be such that $1 \leq |\bfgm|_1 \leq n_F$. We know from the definition of $(F_t(y))_t$ and from Proposition~\ref{prop:reg:M} that
\begin{align*}
    \partial^\bfgm F_h(y) 
    &=
    \int_0^h \partial^{\bfgm} (b(y+W_s) M_s(y) \invar(y)) \cdot \mathrm{d}W_s
    \\
    &=
    \sum_{\bfu + \bfv + \bfw = \bfgm}
    \frac{\bfgm!}{\bfu!\bfv!\bfw!}
    \int_0^h  \partial^{\bfu} b(y+W_s) \partial^{\bfv} M_s(y) \partial^{\bfw} \invar(y) \cdot \mathrm{d}W_s.
\end{align*}
For each triple $(\bfu,\bfv,\bfw)$, the deterministic factor $\partial^{\bfw}\invar(y)$ is bounded on $K$, and Lemma~\ref{lem:I:M} gives
\[
\partial^{\bfu} b(y+W_\cdot)\,\partial^{\bfv}M_\cdot(y)\,\partial^{\bfw}\invar(y)
\in \bbI_K^{(\beta-|\bfu|_1-1)\wedge(\beta-|\bfv|_1)}
\subseteq \bbI_K^{\beta-|\bfgm|_1-1}.
\]
After stochastic integration, each summand in $\partial^\bfgm F_\cdot(y)$ belongs to $\bbI_K^{\beta-|\bfgm|_1}$. Lemma~\ref{lem:k:iterated_ito}, applied with this index and with the power $Y_h^\bfgm$, yields the stated expansion uniformly for $y\in K$. Since $|\bfgm|_1>0$, the constant term vanishes. Finally, Lemma~\ref{lem:control:chibar} rewrites every $\wb\chi$-coefficient evaluated at $h$ as a deterministic power of $h$ times a constant depending on $m$, so the coefficients are deterministic and independent of $h$.
\end{proof}

\begin{proof}[Proof of Lemma~\ref{lem:expansion_hidden_measure:F:beta}]
Fix $K\subset\mathbb R^d$ and put $\delta=\beta-1-n_F\in(0,1]$. The Taylor remainder is bounded by a finite sum of terms of the form
\[
C\int_0^1\E\left[
\left|\partial^\bfgm F_h(y-uY_h)-\partial^\bfgm F_h(y)\right|
|Y_h|^{|\bfgm|_1}
\right]\,\mathrm du,
\qquad |\bfgm|_1=n_F .
\]
We first verify the local random-seminorm bound needed for the Taylor remainder. By the derivative formula for $M_t(y)$, the product rule for $F_t(y)=(M_t(y)-1)\invar(y)$, and Assumption~\ref{assumption:isotropic}, for every such compact $K$ there are random variables $H_{h,K}$ and a number $q_K\ge0$ such that
\[
\left|\partial^\bfgm F_h(y-z)-\partial^\bfgm F_h(y)\right|
\le
h^{1/2}H_{h,K}(1+|z|^{q_K})|z|^\delta,
\qquad y\in K,\ z\in\mathbb R^d,
\]
and $\sup_{0<h<1}\|H_{h,K}\|_{L^2}<\infty$. To see this, expand the difference of $\partial^\bfgm F_h$ into finitely many terms involving either a weighted H\"older increment of a highest derivative of $b$ or $\invar$, or a lower-order derivative multiplied by an increment of a derivative of $M_h$. The latter increments satisfy the same weighted H\"older bound by the stochastic-integral equations of Proposition~\ref{prop:reg:M}, BDG inequalities, and Lemma~\ref{lem:I:M}. The prefactor $h^{1/2}$ comes from the final stochastic integration over $[0,h]$, and the polynomial weights are integrable uniformly for $0<h<1$ by Brownian Gaussian tails.

Using this estimate with $z=uY_h$ and Cauchy's inequality gives
\[
\E\left[
\left|\partial^\bfgm F_h(y-uY_h)-\partial^\bfgm F_h(y)\right|
|Y_h|^{|\bfgm|_1}
\right]
\le
C_K h^{1/2}\left\|(1+|Y_h|^{q_K})|Y_h|^{|\bfgm|_1+\delta}\right\|_{L^2}.
\]
Now, with the rescaled Brownian motion $\wt W_t=h^{-1/2}W_{ht}$,
\[
Y_h=h^{1/2}\wt Y_1,\qquad
\wt Y_1=\int_0^1\chi(s)\,\mathrm d\wt W_s,
\]
so Gaussian moments yield
\[
\left\|(1+|Y_h|^{q_K})|Y_h|^{|\bfgm|_1+\delta}\right\|_{L^2}
\le C_K h^{(|\bfgm|_1+\delta)/2}.
\]
Since $|\bfgm|_1+\delta=\beta-1$, the remainder is $O_K(h^{\beta/2})$.
\end{proof}

Combining the Taylor decomposition of $F_h(x-Y_h)$ with Lemmas~\ref{lem:expansion_hidden_measure:F:bfgm} and~\ref{lem:expansion_hidden_measure:F:beta}, and collecting equal powers of $h$, gives Proposition~\ref{prop:expansion_hidden_measure:F}. The resulting remainder is bounded uniformly for $x$ in each compact set $K$, with constants depending on $K$, $b$, and $m$, but not on $h$.

\section{Proof of Proposition~\ref{prop:bounds:frakp_p}}

Throughout this proof, we denote by $c, C, \widetilde{C}, \lambda, \widetilde{\lambda}$ positive constants whose value may change from line to line, but which depend only on the model parameters, on the dimension $d$, and, through Proposition~\ref{prop:bounds:p}, on the fixed short-time horizon $T_0$. All uses of Proposition~\ref{prop:bounds:p} below are restricted to $0<t\le T_0$.

\subsection{Outline of the proof}

We fix $x \in \R^d$, and we study the distribution of $m^{h}_t(X)$ given $X_0 = 0$. Since $X$ is a Markov process, it follows that $\mathfrak{p}_{h,t}(x, \cdot)$ can be seen as the density of $m^{h}(X)$ given $X_0 \sim p_t(x, \cdot)$. Following the same ideas as in the proof of Proposition~\ref{proposition:expansion_hidden_measure}, we see that
\begin{align*}
\mathfrak{p}_{h,t}(x, y)
&=
\E\big[
p_t(x, y - m^h(W)) M_{h} (y-m^h(W))
\big]
\leq 
\E\big[
p_t (x, y - m^h(W))^2 
\big]^{1/2}
\E\big[
 M_{h} (y-m^h(W))^2
\big]^{1/2}.
\end{align*}

Hence the desired bound for $\mathfrak{p}_{h,t}$ follows from the next two lemmas, which control separately the two factors on the right-hand side. 
\begin{lemma}
\label{lem:proof:frakp_p:1}
For every $T_0>0$, there exist constants $C>0$ and $\lambda > 0$, independent of $x,y$ and $p$, such that, whenever $0<h\le t\le T_0$,
\begin{equation*}
    \E\big[p_t \big(x, y-m^h(W)\big)^2\big]^{1/2} 
    \leq C\, g_{\lambda, d}\big(t, \theta_t(x) - y\big).
\end{equation*}
\end{lemma}

\begin{lemma}
\label{lem:proof:frakp_p:2}
There exist constants $c_0, \widetilde{c}_0, C >0 $ such that, if $ph \leq \widetilde{c}_0$, then for all $y\in\R^d$,
\begin{equation*}
\E\big[
 M_{h} (y-m^h(W))^2
\big]^{1/2}
\leq C\exp(c_0|y|^2).
\end{equation*}
\end{lemma}

The proof of Lemmas~\ref{lem:proof:frakp_p:1} and~\ref{lem:proof:frakp_p:2} are given in Sections~\ref{sec:proof:lem:proof:frakp_p:1} and~\ref{sec:proof:lem:proof:frakp_p:2} respectively. Using these lemmas under the restrictions $0<h\le t\le T_0$ and $ph\le\widetilde c_0$, we obtain
\begin{equation*}
\mathfrak{p}_{h,t}(x; y)
\leq 
\widetilde{C}_0\, g_{\widetilde{\lambda}_0, d}\big(t, \theta_t(x) - y\big)\, \exp\big(c_0 |y|^2\big),
\end{equation*}
for some constants $\widetilde{C}_0>0$ and $\widetilde{\lambda}_0\in(0,1]$, which proves the first estimate in Proposition~\ref{prop:bounds:frakp_p}. \medskip

We now derive the estimate for the spatial gradient $\nabla_x \mathfrak{p}_{h,t}(x;y)$ stated in Proposition~\ref{prop:bounds:frakp_p}. Using the representation formula,
\[
\mathfrak{p}_{h,t}(x;y)
=
\E\Big[
p_t\big(x, y - Z\big)\, M_h(y-Z)
\Big], 
\qquad Z = m^h(W),
\]
and the regularity of $p_t$ in $x$ (see Proposition~\ref{prop:bounds:p}), we may differentiate under the expectation for $0<t\le T_0$; dominated convergence is justified by the short-time Gaussian bounds in Proposition~\ref{prop:bounds:p}, by the condition $h\le t$, and by Lemma~\ref{lem:proof:frakp_p:2}. We obtain
\begin{equation*}
\nabla_x \mathfrak{p}_{h,t}(x;y)
=
\E\Big[
\nabla_x p_t\big(x, y - Z\big)\, M_h(y-Z)
\Big].
\end{equation*}
Applying again Cauchy--Schwarz gives
\begin{equation*}
\big|\nabla_x \mathfrak{p}_{h,t}(x;y)\big|
\leq
\E\big[|\nabla_x p_t(x, y-Z)|^2\big]^{1/2}\,
\E\big[ M_h(y-Z)^2\big]^{1/2}.
\end{equation*}
The second factor here is already controlled by Lemma~\ref{lem:proof:frakp_p:2}. It remains to bound the first factor by a Gaussian kernel. This is entirely analogous to Lemma~\ref{lem:proof:frakp_p:1}, but now using the gradient bound from Proposition~\ref{prop:bounds:p} for $0<t\le T_0$, we have
\begin{equation*}
|\nabla_x p_t(x;z)|
\leq C_1\, t^{-1/2}\, g_{\lambda_1,d}\big(t,\theta_t(x)-z\big),
\end{equation*}
for some $C_1>0$ and $\lambda_1\in(0,1]$. This estimate is used in place of the bound for $p_t(x;z)$ stated in Proposition~\ref{prop:bounds:p}. Keeping track of the extra factor $t^{-1/2}$, we get
\begin{equation*}
   \E\big[|\nabla_x p_t \big(x, y-m^h(W)\big)|^2\big]^{1/2}
    \leq C t^{-1/2} \, g_{\lambda, d}\big(t, \theta_t(x) - y\big)
\end{equation*}
by repeating the argument of Lemma~\ref{lem:proof:frakp_p:1}, which proves in turn the gradient bound in Proposition~\ref{prop:bounds:frakp_p}. The details are skipped for conciseness.

\subsection{Proof of Lemma~\ref{lem:proof:frakp_p:1}}
\label{sec:proof:lem:proof:frakp_p:1}

Recall from Section~\ref{sec:hidden:setting} that
\begin{equation*}
    Z := m^h(W) = \int_0^h \chi(s/h) \, \mathrm{d}W_s.
\end{equation*}
Thus $Z = (Z_1, \dots, Z_d)$ is a $d$-dimensional Gaussian random variable with independent centered components, each of variance $\sigma^2_m h$, where
\begin{equation*}
    \sigma^2_m = \int_0^1 \chi(s)^2 \, \mathrm{d}s.
\end{equation*}
By the definition of $\chi$, we have $\sigma^2_m \leq 1$.
Since $0<t\le T_0$, Proposition~\ref{prop:bounds:p} applies and the transition density $p_t$ satisfies
\begin{equation*}
    p_t(x,z) \leq C_0\, g_{\lambda_0, d}\big(t, \theta_t(x) - z\big),
\end{equation*}
for some constants $C_0>0$ and $\lambda_0 > 0$ (which can be chosen uniformly over the class of drifts $b$ considered in Proposition~\ref{prop:bounds:frakp_p}). In particular,
\begin{equation*}
    p_t(x,z)^2 \leq C_0^2\, t^{-d/2}\, g_{2\lambda_0, d}\big(t, \theta_t(x) - z\big).
\end{equation*}
Moreover, for all $z \in \R^d$,
\begin{equation*}
    g_{\lambda, d}(t, z) = \prod_{i=1}^d g_{\lambda, 1}(t, z_i),
\end{equation*}
so that
\begin{align*}
    \E\big[p_t(x,y-Z)^2\big] 
    &\leq
    C_0^2 t^{-d/2}
    \E\Big[g_{2\lambda_0, d}\big(t, \theta_t(x) - y + Z\big)\Big]
    \\
    &\leq
    C_0^2 t^{-d/2}
    \prod_{i=1}^d
    \E\Big[g_{2\lambda_0, 1}\big(t, \theta_t^{(i)}(x) - y_i + Z_i\big)\Big],
\end{align*}
where $\theta_t^{(i)}(x)$ denotes the $i$-th coordinate of $\theta_t(x)$. Each expectation in the product is computed using the following elementary Gaussian lemma.

\begin{lemma}
\label{lem:bound:gX}
Let $X$ be a centered Gaussian random variable with variance $\sigma^2$. Then for any $a \in \bbR$ and any $\lambda > 0$,
    \begin{equation*}
        \E\Big[g_{\lambda, 1}(t, a + X) \Big] =
        2^{-1/2}\,
        g_{\lambda, 1}\big(t+2\sigma^2 \lambda, a\big).
    \end{equation*}
\end{lemma}

\begin{proof}
By definition,
\begin{align*}
        \E\Big[ g_{\lambda, 1}(t, a+X) \Big] 
        &= 
        \int_{\bbR} \exp\Big(-\frac{|x|^2}{2\sigma^2} - \frac{\lambda|x-a|^2}{t}\Big) \, \frac{\mathrm{d}x}{(2\pi\sigma^2)^{1/2}}
        \\
        &= 2^{-1/2}\,
        (t + 2\lambda \sigma^2)^{-1/2}  \exp \Big( -\frac{ \lambda }{ t + 2 \sigma^2 \lambda } |a|^2 \Big),
\end{align*}
which is exactly $2^{-1/2}\, g_{\lambda,1}(t+2\sigma^2 \lambda, a)$.
\end{proof}

Applying Lemma~\ref{lem:bound:gX} with $X=Z_i$, $\sigma^2 = \sigma_m^2 h$ and $\lambda = 2\lambda_0$, we obtain
\begin{align*}
    \E\big[p_t(x,y-Z)^2\big] 
    &\leq
    C_0^2 (2t)^{-d/2}
    \prod_{i=1}^d
    g_{2\lambda_0, 1}\big(t + 4 \sigma_m^2 h \lambda_0, \theta_t^{(i)}(x) - y_i\big)
    \\
    &\leq
    C_0^2 (2t)^{-d/2}
    g_{2\lambda_0, d}\big(t + 4 \sigma_m^2 h \lambda_0, \theta_t(x) - y\big).
\end{align*}
Since $h \leq t$ and $\sigma_m^2 \leq 1$, we have $t + 4 \sigma_m^2 h \lambda_0 \leq t + 4h \leq 5t$, and thus
\begin{equation*}
    \E\big[p_t(x,y-Z)^2\big] 
    \leq
    C_0^2 \Big(\frac{5}{2t}\Big)^{d/2}
    g_{2 \lambda_0 / 5, d}\big(t, \theta_t(x) - y\big).
\end{equation*}
Taking the square root yields
\begin{equation*}
    \E\big[p_t(x,y-Z)^2\big] ^{1/2}
    \leq
    C_0 \Big(\frac{5}{2}\Big)^{d/4}
    g_{\lambda_0 / 5, d}\big(t, \theta_t(x) - y\big),
\end{equation*}
which proves Lemma~\ref{lem:proof:frakp_p:1} (up to a relabeling of the constants).

\subsection{Proof of Lemma~\ref{lem:proof:frakp_p:2}}
\label{sec:proof:lem:proof:frakp_p:2}

We again write
\begin{equation*}
    Z = m^h(W) = \int_0^h \chi(s/h) \, \mathrm{d}W_s.
\end{equation*}
First, almost surely and for all $y \in \R^d$,
\begin{equation*}
    M_h(y)
    = \exp \Big( \int_0^h b(y + W_s) \, \mathrm{d}W_s - \frac{1}{2} \int_0^h |b(y + W_s)|^2 \, \mathrm{d}s \Big)
    \leq \exp \Big( \int_0^h b(y + W_s) \, \mathrm{d}W_s \Big).
\end{equation*}
Therefore,
\begin{align*}
    \E \big[ M_h(y-Z)^2 \big] 
    &\leq \E\Big[\exp \Big( 2 \int_0^h b(y + W_s - Z) \, \mathrm{d}W_s \Big)\Big]
    \\
    &\leq \sum_{n=0}^{\infty} \frac{2^n}{n!} \E\Big[ \Big( \int_0^h b(y + W_s - Z) \, \mathrm{d}W_s \Big)^n \Big].
\end{align*}
We now prove that, for every integer $n\ge1$,
\begin{equation*}
    \E\Big[ \Big| \int_0^h b(y + W_s - Z) \, \mathrm{d}W_s \Big|^n \Big]
    \leq
    C^n K_{2n}^{1/2}(ph)^{n/2}
    \big(1+K_{2n}^{1/2}h^{n/2}+|y|^n\big),
\end{equation*}
where $K_q=\E[|N|^q]$ for a standard one-dimensional Gaussian random variable $N$, and where $C$ depends only on the model parameters and on the dimension. This estimate is trivial for $n=0$ after changing $C$.

Since $m$ is the discrete measure in~\eqref{eq:measure_preav},
\begin{equation*}
    Z=m^h(W)=\frac1p\sum_{j=0}^{p-1}W_{jh/p}.
\end{equation*}
In particular, $Z$ is measurable with respect to
\begin{equation*}
    \calG=\sigma(W_{kh/p}:0\le k\le p).
\end{equation*}
Set $t_k=kh/p$ and $\delta=h/p$. For $0\le r\le\delta$, define
\begin{equation*}
    I_k=\frac{W_{t_{k+1}}-W_{t_k}}{\delta},
    \qquad
    \beta_r^k=W_{t_k+r}-W_{t_k}-rI_k .
\end{equation*}
The Gaussian regression formula for Brownian motion gives that the family $(\beta^0,\ldots,\beta^{p-1})$ is independent of $\calG$, the bridges are mutually independent, and each $\beta^k$ is a $d$-dimensional Brownian bridge from $0$ to $0$ over $[0,\delta]$. Equivalently, conditionally on $\calG$, the law of each $\beta^k$ is the centered Gaussian law with covariance
\begin{equation*}
    \E\big[\beta_r^{k,i}\beta_u^{k,j}\mid\calG\big]
    =
    \mathbf 1_{\{i=j\}}\left(r\wedge u-\frac{ru}{\delta}\right),
    \qquad 0\le r,u\le\delta,
\end{equation*}
and this conditional law does not depend on the realized grid values.

Let $\mathbb F^{\beta,k}=(\mathcal F_r^{\beta,k})_{0\le r<\delta}$ be the natural completed filtration of $\beta^k$. In this filtration the canonical Brownian-bridge semimartingale decomposition is
\begin{equation*}
    \mathrm d\beta_r^k
    =
    -\frac{\beta_r^k}{\delta-r}\,\mathrm dr
    +
    \mathrm d\overline B_r^k,
    \qquad 0\le r<\delta,
\end{equation*}
where $\overline B^k$ is a $d$-dimensional Brownian motion in $\mathbb F^{\beta,k}$. After the initial enlargement by the independent sigma-field $\calG$, the same process $\overline B^k$ remains a Brownian motion. For
\begin{equation*}
    Y_r^k
    =
    y+W_{t_k}+rI_k+\beta_r^k-Z,
    \qquad 0\le r\le\delta,
\end{equation*}
the process $Y^k$ is adapted to $\calG\vee\mathcal F_r^{\beta,k}$. Hence, conditionally on $\calG$, semimartingale integration on $[0,\delta-\varepsilon]$ gives
\begin{equation*}
    \int_{t_k}^{t_k+\delta-\varepsilon}
    b(y+W_s-Z)\,\mathrm dW_s
    =
    A_{k,\varepsilon}+S_{k,\varepsilon}+M_{k,\varepsilon},
\end{equation*}
with
\begin{equation*}
    A_{k,\varepsilon}
    =
    I_k\int_0^{\delta-\varepsilon}b(Y_r^k)\,\mathrm dr,
    \qquad
    S_{k,\varepsilon}
    =
    -\int_0^{\delta-\varepsilon}
    b(Y_r^k)\frac{\beta_r^k}{\delta-r}\,\mathrm dr,
    \quad\text{ and }\quad
    M_{k,\varepsilon}
    =
    \int_0^{\delta-\varepsilon}b(Y_r^k)\,\mathrm d\overline B_r^k .
\end{equation*}
The estimates below, applied on terminal subintervals $(\delta-\eta,\delta-\varepsilon)$, show that $S_{k,\varepsilon}$ is Cauchy in $L^n$ as $\varepsilon\to 0$. The same argument, using BDG for the martingale increments and the vanishing of $\int_{\delta-\eta}^{\delta}|b(Y_r^k)|^2\,\mathrm dr$ in every $L^{n/2}$, shows that $M_{k,\varepsilon}$ converges in $L^n$. We may therefore write
\begin{equation*}
    \int_{t_k}^{t_{k+1}} b(y+W_s-Z)\,\mathrm dW_s
    =
    A_k+S_k+M_k,
\end{equation*}
where
\begin{equation*}
    A_k
    =
    I_k\int_0^\delta b(Y_r^k)\,\mathrm dr,
    \qquad
    S_k
    =
    -\int_0^\delta b(Y_r^k)\frac{\beta_r^k}{\delta-r}\,\mathrm dr,
    \qquad
    M_k
    =
    \int_0^\delta b(Y_r^k)\,\mathrm d\overline B_r^k .
\end{equation*}

We first record the Gaussian input used in all three estimates. For every $0\le r\le\delta$, the random vector $W_{t_k+r}-Z$ is centered Gaussian and satisfies
\begin{equation*}
    \E\big[|W_{t_k+r}-Z|^2\big]\le Ch .
\end{equation*}
Indeed, $0\le t_k+r\le h$ and $Z=p^{-1}\sum_{j=0}^{p-1}W_{t_j}$, so this follows from $\E|W_s|^2\le Ch$ and $\E|Z|^2\le Ch$. Since $b$ has at most linear growth, for every $q\ge1$,
\begin{equation*}
    \sup_{0\le k\le p-1}\sup_{0\le r\le\delta}
    \|b(Y_r^k)\|_{L^q}
    \le
    C\big(1+|y|+K_q^{1/q}h^{1/2}\big).
\end{equation*}
Also,
\begin{equation*}
    \|I_k\|_{L^q}
    \le
    C K_q^{1/q}\delta^{-1/2}.
\end{equation*}

For $A_k$, Hölder's inequality and Minkowski's integral inequality yield, for $n\ge1$,
\begin{align*}
    \|A_k\|_{L^n}
    &\le
    \|I_k\|_{L^{2n}}
    \int_0^\delta \|b(Y_r^k)\|_{L^{2n}}\,\mathrm dr
    \\
    &\le
    C K_{2n}^{1/(2n)}\delta^{1/2}
    \big(1+|y|+K_{2n}^{1/(2n)}h^{1/2}\big).
\end{align*}
Consequently,
\begin{equation*}
    \E[|A_k|^n]
    \le
    C^nK_{2n}^{1/2}\delta^{n/2}
    \big(1+|y|^n+K_{2n}^{1/2}h^{n/2}\big).
\end{equation*}

For the bridge-drift term $S_k$, use again Minkowski and Hölder:
\begin{align*}
    \|S_k\|_{L^n}
    &\le
    \int_0^\delta
    \left\|
    |b(Y_r^k)|\,\left|\frac{\beta_r^k}{\delta-r}\right|
    \right\|_{L^n}\,\mathrm dr
    \\
    &\le
    \int_0^\delta
    \|b(Y_r^k)\|_{L^{2n}}
    \left\|\frac{\beta_r^k}{\delta-r}\right\|_{L^{2n}}
    \,\mathrm dr .
\end{align*}
Since $\beta_r^k$ is Gaussian with covariance $(r(\delta-r)/\delta)I_d$, the Brownian-bridge moment estimate gives
\begin{equation*}
    \left\|\frac{\beta_r^k}{\delta-r}\right\|_{L^{2n}}
    \le
    C K_{2n}^{1/(2n)}
    \left(\frac{r}{\delta(\delta-r)}\right)^{1/2}.
\end{equation*}
The singularity is integrable after taking the $L^{2n}$-norm:
\begin{equation*}
    \int_0^\delta
    \left(\frac{r}{\delta(\delta-r)}\right)^{1/2}
    \,\mathrm dr
    =
    \delta^{1/2}\int_0^1\left(\frac{u}{1-u}\right)^{1/2}\,\mathrm du
    =
    \frac{\pi}{2}\delta^{1/2}.
\end{equation*}
It follows that
\begin{equation*}
    \E[|S_k|^n]
    \le
    C^nK_{2n}^{1/2}\delta^{n/2}
    \big(1+|y|^n+K_{2n}^{1/2}h^{n/2}\big).
\end{equation*}
Thus the bridge-drift term has the same per-interval scaling $\delta^{n/2}=(h/p)^{n/2}$ as the endpoint and martingale terms; no stronger condition than $ph\le\widetilde c_0$ is needed here.

It remains to bound $M_k$. Conditionally on $\calG$, $M_{k,t}:=\int_0^t b(Y_r^k)\,\mathrm d\overline B_r^k$, $0\le t<\delta$, is a continuous local martingale in the bridge filtration enlarged by $\calG$, with quadratic variation
\begin{equation*}
    \langle M_k\rangle_t
    =
    \int_0^t |b(Y_r^k)|^2\,\mathrm dr .
\end{equation*}
We use the following version of the Burkholder--Davis--Gundy inequality.
\begin{lemma}[Theorem A in~\cite{carlen1991Lp}]
For any continuous bounded martingale $M$ and any $t \geq 0$,
\begin{equation*}
    \E\big[ |M_t|^p \big]^{1/p} \leq 2 p^{1/2} \E\big[ \langle M\rangle_t^{p/2} \big]^{1/p}, \qquad p \geq 1,
\end{equation*}
where $\langle M\rangle_t$ denotes the quadratic variation of $M$.
\end{lemma}
Applying the Burkholder--Davis--Gundy inequality conditionally on $\calG$ to $M_{k,t}$, $t<\delta$, then letting $t\to\delta$ and integrating over $\calG$, gives, for $n\ge1$,
\begin{equation*}
    \E[|M_k|^n]
    \le
    (2\sqrt n)^n
    \E\left[
    \left(\int_0^\delta |b(Y_r^k)|^2\,\mathrm dr\right)^{n/2}
    \right].
\end{equation*}
For $n\ge2$, Jensen--Minkowski gives
\begin{equation*}
    \left\|
    \left(\int_0^\delta |b(Y_r^k)|^2\,\mathrm dr\right)^{1/2}
    \right\|_{L^n}
    \le
    \left(\int_0^\delta
    \|b(Y_r^k)\|_{L^n}^2\,\mathrm dr\right)^{1/2}
    \le
    C\delta^{1/2}\big(1+|y|+K_n^{1/n}h^{1/2}\big).
\end{equation*}
The same bound for $n=1$ follows from Cauchy--Schwarz. Since Stirling's formula implies $n^{n/2}\le C^nK_{2n}^{1/2}$ and $K_n\le K_{2n}^{1/2}$, we get
\begin{equation*}
    \E[|M_k|^n]
    \le
    C^nK_{2n}^{1/2}\delta^{n/2}
    \big(1+|y|^n+K_{2n}^{1/2}h^{n/2}\big).
\end{equation*}

Combining the three estimates and using
\begin{equation*}
    J_h(y)=\sum_{k=0}^{p-1}(A_k+S_k+M_k),
\end{equation*}
we obtain, for $n\ge1$,
\begin{align*}
    \E[|J_h(y)|^n]
    &\le
    (3p)^{n-1}
    \sum_{k=0}^{p-1}
    \E\big[|A_k|^n+|S_k|^n+|M_k|^n\big]
    \le
    C^nK_{2n}^{1/2}(ph)^{n/2}
    \big(1+|y|^n+K_{2n}^{1/2}h^{n/2}\big),
\end{align*}
because $p^n\delta^{n/2}=(ph)^{n/2}$.

We now sum the exponential series. Since $K_{2n}=(2n)!/(2^n n!)\le 2^n n!$, and since $h\le ph$ for $p\ge1$, the previous estimate gives, after changing constants,
\begin{equation*}
    \E \big[ M_h(y-Z)^2 \big]
    \le
    C\sum_{n=0}^{\infty}
    \frac{(C\sqrt{ph})^n}{(n!)^{1/2}}
    \big(1+|y|^n\big)
    +
    C\sum_{n=0}^{\infty}
    (C\sqrt{ph})^n .
\end{equation*}
Choose $\widetilde c_0\in(0,1]$ so small that $C\sqrt{ph}\le1/2$ whenever $ph\le\widetilde c_0$. The second series is then bounded by a constant. Moreover, by Cauchy--Schwarz,
\begin{equation*}
    \sum_{n=0}^{\infty}\frac{z^n}{(n!)^{1/2}}
    =
    \sum_{n=0}^{\infty}2^{-n/2}
    \frac{(2z^2)^{n/2}}{(n!)^{1/2}}
    \le
    \left(\sum_{n=0}^{\infty}2^{-n}\right)^{1/2}
    \left(\sum_{n=0}^{\infty}\frac{(2z^2)^n}{n!}\right)^{1/2}
    \le
    C\exp(Cz^2).
\end{equation*}
Applying this with $z=C\sqrt{ph}$ and $z=C\sqrt{ph}\,|y|$, and using $ph\le\widetilde c_0\le1$, yields
\begin{equation*}
    \E \big[ M_h(y-Z)^2 \big]
    \le
    C\exp(c|y|^2).
\end{equation*}
Taking square roots and changing $C,c$ once more gives the same right-hand-side in the previous inequality, which proves Lemma~\ref{lem:proof:frakp_p:2}.

\section{Proof of Corollaries~\ref{corrolary:bounds:1} to~\ref{corrolary:bounds:4}}
\label{appendix:proof:corollaries}

\subsection{Proof of Corollary~\ref{corrolary:bounds:1}}

We only detail the proof of the first inequality; the gradient bound is obtained 
by the same arguments, starting from the corresponding estimate in 
Proposition~\ref{prop:bounds:frakp_p} and tracking an additional factor $t^{-1/2}$.

By Proposition~\ref{prop:bounds:frakp_p} (applied with $T_0=t_0^*$), for every $c_0>0$ there exist 
constants $\wt c_0>0$, $\wt C_0>0$ and $\wt\lambda_0\in(0,1]$ such that, 
whenever $ph\le\wt c_0$, we have for all $h\le t\le T_0$ and $x,y\in\R^d$,
\begin{equation}\label{eq:start-frakp}
    \mathfrak{p}_{h,t}(x; y)
    \le 
    \wt C_0\,
    g_{\wt \lambda_0,d}\big(t,\theta_t(x)-y\big)\,
    \exp\!\big(c_0 |y|^2\big),
\end{equation}
where $\theta_t$ is the deterministic flow defined in~\eqref{eq:theta}.\medskip

\noindent\textit{Step 1: move the exponential weight from $y$ to $x$.}
Use the elementary inequality
\[
    |y|^2 \le 2\,|y-\theta_t(x)|^2 + 2\,|\theta_t(x)|^2,
\]
which yields
\begin{equation}\label{eq:split-y}
    \exp\!\big(c_0 |y|^2\big)
    \le 
    \exp\!\big(2 c_0 |y-\theta_t(x)|^2\big)\,
    \exp\!\big(2 c_0 |\theta_t(x)|^2\big).
\end{equation}
Plugging~\eqref{eq:split-y} into~\eqref{eq:start-frakp} and using 
$g_{\lambda,d}(t,z)=t^{-d/2}\exp(-\lambda|z|^2/t)$ gives
\[
\begin{aligned}
    \mathfrak{p}_{h,t}(x;y)
    &\le 
    \wt C_0\,t^{-d/2}
    \exp\!\Big(-\wt \lambda_0\frac{|\theta_t(x)-y|^2}{t}\Big)\,
    \exp\!\big(2 c_0 |\theta_t(x)-y|^2\big)\,
    \exp\!\big(2 c_0 |\theta_t(x)|^2\big)\\
    &= \wt C_0\,t^{-d/2}
    \exp\!\Big(-\Big(\tfrac{\wt \lambda_0}{t}-2c_0 \Big)\,|\theta_t(x)-y|^2\Big)\,
    \exp\!\big(2 c_0 |\theta_t(x)|^2\big).
\end{aligned}
\]
For $t \leq \wt \lambda_0 / (4c_0)$, we then see that
\[
    \frac{\wt\lambda_0}{t}-2c_0 \ge \frac{\widetilde{\lambda}_0}{2t} = \frac{\widehat{\lambda}_0}{t}
\]
with $\widehat{\lambda}_0\in(0,1]$. Hence
\begin{equation}\label{eq:after-y-theta}
    \mathfrak{p}_{h,t}(x;y)
    \le 
    C_0\,
    g_{\widehat{\lambda}_0,d}\big(t,\theta_t(x)-y\big)\,
    \exp\!\big(2 c_0 |\theta_t(x)|^2\big),
\end{equation}
for a suitable constant $C_0>0$.\\

Next, we bound $|\theta_t(x)|$ in terms of $|x|$. By 
Corollary~1 in~\cite{amorino2021minimax}, we have for all $t$,
\begin{equation}\label{eq:amorino}
    |\theta_t(x) - x| \leq C t \,(1 + |x|),
\end{equation}
for some constant $C>0$ depending only on  $\mathfrak{b}$. It follows that
\[
    |\theta_t(x)|
    \le |x| + C t(1+|x|)
\]
for a constant $C'>0$ and all $t \leq \wt \lambda_0 / (4c_0)$. Consequently,
\[
    \exp\!\big(2 c_0 |\theta_t(x)|^2\big)
    \le 
    \,\exp\!\big(C' c_0 |x|^2\big)
\]
for a suitable constant $C'>0$. Absorbing $C_1$ into $C_0$ and renaming 
constants, we can rewrite~\eqref{eq:after-y-theta} as
\begin{equation}\label{eq:theta-to-x-exp}
    \mathfrak{p}_{h,t}(x;y)
    \le 
    \wt C_0 \,
    g_{\widehat{\lambda}_0,d}\big(t,\theta_t(x)-y\big)\,
    \exp\!\big(C' c_0 |x|^2\big ),
\end{equation}
for all $t \leq \wt \lambda_0 / (4c_0)$.\medskip

\noindent\textit{Step 2: compare Gaussian kernels centered at $\theta_t(x)$ and at $x$.}
Write
\[
    x-y = (\theta_t(x)-y) - (\theta_t(x)-x).
\]
Set $u:=\theta_t(x)-y$ and $a:=\theta_t(x)-x$. Then $|u|^2 \ge \tfrac12|u-a|^2 - |a|^2$, which gives
\[
    |\theta_t(x)-y|^2 
    \ge \frac12|x-y|^2 - |\theta_t(x)-x|^2.
\]
Hence
\[
\begin{aligned}
    g_{\widehat{\lambda}_0,d}\big(t,\theta_t(x)-y\big)
    &= t^{-d/2}\exp\!\Big(-\widehat{\lambda}_0\frac{|\theta_t(x)-y|^2}{t}\Big)\\
    &\le t^{-d/2}\exp\!\Big(-\frac{\widehat{\lambda}_0}{2t}|x-y|^2\Big)\,
    \exp\!\Big(\widehat{\lambda}_0\frac{|\theta_t(x)-x|^2}{t}\Big).
\end{aligned}
\]

Using again~\eqref{eq:amorino}, we have
\[
    \frac{|\theta_t(x)-x|^2}{t} 
    \le C^2 t(1+|x|)^2
\]
and thus 
\[
    \exp\!\Big(\widehat{\lambda}_0\frac{|\theta_t(x)-x|^2}{t}\Big)
    \le \exp\!\big(C^2 t |x|^2\big).
\]
Therefore
\[
    g_{\widehat{\lambda}_0,d}\big(t,\theta_t(x)-y\big)
    \le 
    C_4\,
    g_{\widehat{\lambda}_0 /2,d}(t,x-y)\,
    \exp\!\big(C^2 t |x|^2\big),
\]
with $\widehat{\lambda}_0 /2 \in (0,1]$.

\medskip

\noindent\textit{Step 3: conclusion.}
Combining this last estimate with~\eqref{eq:theta-to-x-exp}, we obtain
\[
    \mathfrak{p}_{h,t}(x;y)
    \le 
    \widetilde{C}_0\,
    g_{\widetilde{\lambda}_0,d}(t,x-y)\,
    \exp\!\big((C'c_0+C^2t)|x|^2\big)
\]
for a suitable constant $C'>0$ and $C^2$ and 
all $h\le t \leq \wt \lambda_0 / (4c_0)$.
Renaming the constants and taking $t$ small enough yields the first inequality in the statement. \medskip

The gradient bound follows exactly the same steps, starting from the gradient 
estimate in Proposition~\ref{prop:bounds:frakp_p} and carrying along the 
additional factor $t^{-1/2}$.

\subsection{Proof of Corollary~\ref{corrolary:bounds:2}}

Let $\tau = h + t / 2$ be the midpoint of the interval $(h, t+h)$. Assume $t\ge2h$; then $h\le t/2$ and the two transition-density factors below fall within the range of Corollary~\ref{corrolary:bounds:1}. By the Markov property of $X$ at time $\tau$ and using Bayes' formula, we can write
\begin{equation*}
    \mathfrak{q}_{h, t+h}(x; y) = 
    \frac{1}{m^{h}(\wb \mu^{b})(x)}
    \int_{\mathbb{R}^d} 
    \mathfrak{p}_{h, t + h - \tau}(z; y) 
    \mathfrak{p}_{h, \tau}^{\leftarrow}(z; x) \wb \mu^{b}(z) \, \mathrm{d}z,
\end{equation*}
where $\mathfrak{p}_{h, \tau}^{\leftarrow}(x)$ is the density of $m_{h}(X)$ conditionally on $X_{\tau} = z$, and $m^{h}(\wb \mu^{b})$ is the density of $m_{h}(X)$ under $\mathbb{P}^{b}_{\wb \mu^b}$.
Indeed, given $X_\tau=z$, the future segment $(X_{\tau+u})_{u \geq 0}$ is independent of the past
and distributed as $(X_u)_{0\le u\le t}$ started from $z$, so that
$m^h_{t,t+h}(X)$ has the same conditional law as $m^h_{t-\tau,t-\tau + h}(X)$ started from $z$.

By Corollary~\ref{corrolary:bounds:1}, for every $c_0>0$ there exist $\wt c_0>0$,
$\wt\lambda_0\in(0,1]$ and $\wt C_0>0$ such that, if $ph\le\wt c_0$ and
$2h\le t\le t_0^*$, then for all $z,y\in\R^d$,
\[
\mathfrak{p}_{h,t+h-\tau}(z;y)
\le \wt C_0\,g_{\wt\lambda_0,d}(t/2,z-y)\,\exp\!\big(c_0|z|^2\big),
\]
since $h\le t/2\le t_0^*/2$ after decreasing $t_0^*$ and changing constants if necessary. On the other hand, we know that the time-reversed process $X^{\leftarrow}_u = X_{\tau - u}$ is a stochastic diffusion with dynamics 
\begin{equation*}
    dX^{\leftarrow}_u = b^{\leftarrow}(X^{\leftarrow}_u) \mathrm{d}u + dW^{\leftarrow}_u,
\end{equation*}
for some Brownian motion $W^{\leftarrow}$ and with 
\begin{equation*}
    b^{\leftarrow}(x) = - b(x) + \nabla \log(\wb \mu^b)(x) = b(x) - 2 \wt b(x).
\end{equation*}
Therefore, this process satisfies the assumptions needed to apply~\cite{menozzi2020density}, and for all $z,x\in\R^d$,
\begin{equation*}
\mathfrak{p}_{h, t/2}^{\leftarrow}(z; x)
\le \wt C_0\,g_{\wt\lambda_0,d}(t/2,z-x)\,\exp\!\big(c_0|x|^2\big).
\end{equation*}
Combining these bounds and using also that $\wb \mu^b$ is bounded yields
\begin{equation}\label{eq:q-upper-gauss}
\mathfrak{q}_{h,t+h}(x;y)
\le 
C
\frac{\exp\!\big(c_0|x|^2\big)}{m^{h}(\wb \mu^{b})(x)}
\,\int_{\R^d}
g_{\wt\lambda_0,d}(t/2,z-y)\,
g_{\wt\lambda_0,d}\big(t/2,x-z\big)\,
\exp\!\big(c_0|z|^2\big)\,\mathrm{d}z,
\end{equation}
for a constant $C>0$ that depends only on $\mathfrak{b}$ and $c_0$. \medskip

We now bound the right-hand side of~\eqref{eq:q-upper-gauss}.  
Set $s := t / 2 > 0$. Using the inequality $|z| \le |x| + |z-x|$, we have
\[
|z|^2 \le 2|x|^2 + 2|z-x|^2,
\qquad
\exp(c_0|z|^2)
\le 
\exp(2c_0|x|^2)\,\exp(2c_0|z-x|^2).
\]
Plugging this into~\eqref{eq:q-upper-gauss} yields
\begin{align}
I
:=&\int_{\R^d}
g_{\wt\lambda_0,d}(s,z-y)\,
g_{\wt\lambda_0,d}(s,x-z)\,
\exp(c_0|z|^2)\, \mathrm{d}z \notag \\
\le&\;
\exp(2c_0|x|^2)
\int_{\R^d}
g_{\wt\lambda_0,d}(s,z-y)\,
s^{-d/2}\exp\!\Big(-\Big(\tfrac{\wt\lambda_0}{s}-2c_0\Big)|x-z|^2\Big)\, \mathrm{d}z.
\label{eq:conv-start}
\end{align}

Choose $t_0^*$ small enough so that $s=t / 2 \le \wt\lambda_0/(4c_0)$.
Then
\[
\tfrac{\wt\lambda_0}{s}-2c_0 \ge \tfrac{\wt\lambda_0}{2s},
\qquad
s^{-d/2}\exp\!\Big(-\Big(\tfrac{\wt\lambda_0}{s}-2c_0\Big)|x-z|^2\Big)
\le 
g_{\wt\lambda_0/2,d}(s,x-z).
\]
Thus from~\eqref{eq:conv-start},
\begin{equation}\label{eq:conv-gaussian}
I
\le
\exp(2c_0|x|^2)
\int_{\R^d}
g_{\wt\lambda_0,d}(s,z-y)\,
g_{\wt\lambda_0/2,d}(s,x-z)\, \mathrm{d}z.
\end{equation}

We now compute the convolution of these two Gaussian kernels.
Write $A=\wt\lambda_0$ and $B=\wt\lambda_0/2$. A standard completion-of-the-square argument gives
\[
A|z-y|^2 + B|z-x|^2
= (A+B)|z-m|^2 + \frac{AB}{A+B}|x-y|^2,
\qquad
m := \frac{Ay + Bx}{A+B}.
\]
Therefore,
\begin{align*}
&\int_{\R^d} g_{\wt\lambda_0,d}(s,z-y)\,g_{\wt\lambda_0/2,d}(s,x-z)\, \mathrm{d}z  \\
&\qquad
= s^{-d}\exp\!\Big(-\frac{AB}{A+B}\frac{|x-y|^2}{s}\Big)
\int_{\R^d}\exp\!\Big( -(A+B)\frac{|z-m|^2}{s}\Big)\, \mathrm{d}z  \\
&\qquad
= C_1\,s^{-d/2}\exp\!\Big(-\wt\lambda_* \frac{|x-y|^2}{s}\Big)
= C_1\,g_{\wt\lambda_*,d}(s,x-y),
\end{align*}
for some $C_1>0$ and $\wt\lambda_* := AB/(A+B) \in (0,1]$. Inserting this into~\eqref{eq:conv-gaussian} gives
\[
I
\le
C_1\,\exp(2c_0|x|^2)\,
g_{\wt\lambda_*,d}(s,x-y).
\]
Since $s=t / 2\in(0,t]$ and $t\le t_0^*$, we have $s$ and $t$ comparable.
Hence there exist constants $\widetilde C_0>0$ and $\widetilde\lambda_0\in(0,1]$
such that
\[
g_{\lambda_*,d}(s,x-y)
\le
\widetilde C_0\,g_{\widetilde\lambda_0,d}(t,x-y).
\]
Combining this with~\eqref{eq:q-upper-gauss} completes the estimate:
\[
\mathfrak{q}_{h,t+h}(x;y)
\le 
\widetilde{C}_0\,
g_{\widetilde{\lambda}_0,d}(t,x-y)\,
\frac{\exp(c_0|x|^2)}{m^{h}(\wb \mu^{b})(x)}.
\]
This proves the stated bound.

\subsection{Proof of Corollary~\ref{corrolary:bounds:3}}

    Fix $r_\star\in(0,t_0^*/2]$ and decrease $\wt c_0$ so that $ph\le\wt c_0$ implies $h\le r_\star$. Using the Markov decomposition at time $t+h-r_\star$, we have
\begin{equation*}
    \mathfrak{q}_{h, t+h}(x; y) = 
    \frac{1}{m^{h}(\wb \mu^{b})(x)}
    \int_{\mathbb{R}^d} 
    \mathfrak{p}_{h, r_\star}(z; y) 
    \mathfrak{p}_{h, t+h-r_\star}^{\leftarrow}(z; x)
    \wb \mu^{b}(z) \, \mathrm{d}z.
\end{equation*}
By Corollary~\ref{corrolary:bounds:1}, applied at the fixed short time $r_\star$ rather than at a large time, we have
\begin{equation*}
     \mathfrak{p}_{h, r_\star}(z; y) \leq \widetilde{C}_0 \, g_{\widetilde{\lambda}_0, d}(r_\star, z - y) \exp(c_0 |y|^2) \leq \widetilde{C}_0' \exp(c_0 |y|^2),
\end{equation*}
for some constants $\widetilde{C}_0, \widetilde{C}_0'$. Therefore, we get
\begin{align*}
    \mathfrak{q}_{h, t+h}(x; y) 
    &\leq  
    \widetilde{C}_0' \exp(c_0 |y|^2)
    \frac{1}{m^{h}(\wb \mu^{b})(x)}
    \int_{\mathbb{R}^d} 
    \mathfrak{p}_{h, t+h-r_\star}^{\leftarrow}(z; x)
    \wb \mu^{b}(z) \, \mathrm{d}z
\\
&=
    \widetilde{C}_0' \exp(c_0 |y|^2).
\end{align*}

\subsection{Proof of Corollary~\ref{corrolary:bounds:4}}

Under the same condition $ph\le\wt c_0$, this follows directly from Corollary~\ref{corrolary:bounds:3}, choosing any fixed $t\ge t_0^*$, and from the identity
\begin{equation*}
    m^h(\wb \mu^b)(y) = \int_{\mathbb{R}^d} m^h(\wb \mu^b)(x) \mathfrak{q}_{h, t+h}(x; y) \, \mathrm{d}x.
\end{equation*}

\section{Proof of the results of Section~\ref{sec:construction}}

\subsection{Proof of Equation~\eqref{eq:bound:bias:kernel}}
\label{sec:proof:eq:bound:bias:kernel}

Recall the notation
\[
\mu_{p,n}^b = \overline{\mu}^b * \varphi_{\tau_n / p^{1/2}},
\]
where $\varphi_{\tau}$ denotes the density of a centred Gaussian with variance $\tau^2 I_d$.
Since $\overline{\mu}^b \in \mathcal{H}_d(\boldsymbol{\alpha},\boldsymbol{\mathcal{L}})$ and convolution with a Gaussian preserves anisotropic H\"older regularity, we have
\[
\mu_{p,n}^b \in \mathcal{H}_d(\boldsymbol{\alpha},\boldsymbol{\mathcal{L}}).
\]
Expanding the convolution with $\bs{K}_{\bs{h}}$, we write
\begin{equation*}
    \mu_{p,n}^b * \bs{K}_{\bs{h}}(x) 
    = \int_{\mathbb{R}^d}
    \mu_{p,n}^b(z)\,
    \bs{K}_{\bs{h}}(x-z)\,\mathrm{d}z 
    = \int_{\mathbb{R}^d}
    \mu_{p,n}^b(x-\bs{h}\odot z)\,
    \bs{K}(z)\,\mathrm{d}z,
\end{equation*}
where we use the shorthand $\bs{h} \odot z = (h_1 z_1,\dots,h_d z_d)$ and $\bs{K}(z) = \prod_{i=1}^d K(z_i)$.
Using $\int_{\mathbb{R}} K(y)\,\mathrm{d}y = 1$, we obtain
\begin{equation*}
    \mu_{p,n}^b * \bs{K}_{\bs{h}}(x) - \mu_{p,n}^b(x)
    =
    \int_{\mathbb{R}^d}
    \big(
    \mu_{p,n}^b(x-\bs{h}\odot z)
    -
    \mu_{p,n}^b(x)
    \big)\,
    \bs{K}(z)\,\mathrm{d}z.
\end{equation*}

We now use a Taylor expansion of $\mu_{p,n}^b$ in each coordinate combined with its anisotropic H\"older regularity and the moment conditions on $K$. Since the kernel $K$ is of order $l \ge \lceil \alpha_d\rceil$, the contributions of the derivatives of order $1,\dots,\lfloor \alpha_i\rfloor-1$ vanish by~\eqref{eq:def:K}. The remainder term in the $i$-th direction is controlled by the $\alpha_i$-H\"older property and gives, for each $i$,
\[
\big|\mu_{p,n}^b(x-\bs{h}\odot z) - \mu_{p,n}^b(x)\big|
\lesssim \sum_{i=1}^d h_i^{\alpha_i} |z_i|^{\alpha_i},
\]
up to a constant depending only on $\boldsymbol{\mathcal{L}}$ and $\boldsymbol{\alpha}$. Integrating with respect to $\bs{K}(z)$, we obtain
\begin{equation*}
\Big| \mu_{p,n}^b * \bs{K}_{\bs{h}}(x) - \mu_{p,n}^b(x)\Big|
\lesssim  
\sum_{i=1}^d 
h_i^{\alpha_i}
\int_{\mathbb{R}^d}
|z_i|^{\alpha_i}\,|\bs{K}(z)|\,\mathrm{d}z
\lesssim \sum_{i=1}^d h_i^{\alpha_i}.
\end{equation*}
This yields Equation~\eqref{eq:bound:bias:kernel}.

\subsection{Invertibility of $A$}
\label{sec:appendix:A}
We prove that the matrix $A$ defined in Section~\ref{sec:construction} is invertible by computing its determinant explicitly and showing that it is non-zero.

Recall that for $0 \le k,i \le \ell$ the entries of $A$ are
\[
a_{k,i} = \sum_{j=0}^k a_{k,i,j},
\qquad
a_{k,i,j} = {k \choose j} (-1)^j m_j\, i^{k-j},
\]
where $m_j$ is the $j$-th moment of a standard Gaussian variable. By multilinearity of the determinant with respect to the rows, we get
\begin{align*}
\det(A) 
&= 
\sum_{j_0=0}^0 
\sum_{j_1=0}^1
\cdots 
\sum_{j_\ell=0}^\ell
\det\big((a_{k,i,j_k})_{0\le k,i\le\ell}\big) \\
&=
\sum_{j_0=0}^0 
\sum_{j_1=0}^1
\cdots 
\sum_{j_\ell=0}^\ell
\det\big( ( i^{\,k-j_k})_{0\le k,i\le\ell}\big)\,
\prod_{k=0}^{\ell}
\Big(
{k \choose j_k} (-1)^{j_k} m_{j_k} \Big).
\end{align*}

Observe that 
\[
\det\big( ( i^{\,k-j_k})_{0\le k,i\le\ell}\big) = 0
\]
whenever two rows coincide, which happens if $k-j_k = k'-j_{k'}$ for some $k\neq k'$. Thus, the determinant is non-zero only if all the integers $k-j_k$ ($0\le k\le\ell$) are distinct. Since $0\le j_k\le k$, we must have $k-j_k\in\{0,\dots,\ell\}$ and hence the only possible choice is
\[
j_0 = j_1 = \cdots = j_\ell = 0.
\]
(One can check this by a simple induction on $k$.) Therefore,
\begin{equation*}
\det(A) 
=
\det\big( ( i^{\,k})_{0\le k,i\le\ell}\big)
\prod_{k=0}^{\ell}
\Big(
{k \choose 0} (-1)^{0} m_{0} \Big).
\end{equation*}
By definition $m_0 = 1$, and ${k\choose 0}=1$, so the product is equal to $1$ and
\[
\det(A) 
=
\det\big( ( i^{\,k})_{0\le k,i\le\ell}\big).
\]
The matrix $((i^{\,k})_{0\le k,i\le\ell})$ is a Vandermonde matrix with nodes $0,1,\dots,\ell$ (up to a transposition); its determinant is
\[
\det\big( ( i^{\,k})_{0\le k,i\le\ell}\big)
= \prod_{0 \le i < j \le \ell} (j-i) \neq 0.
\]
Hence $\det(A)\neq 0$ and $A$ is invertible.

\subsection{Proof of Proposition~\ref{propo:small:bias:mu}}
\label{sec:propo:bias:mu}
Recall that, for all $x\in\mathbb{R}^d$ the estimator $\widehat\mu_{n,\bs{h},p}(x)$ is defined in~\eqref{eq:def:est:mu} by
\begin{equation*}
	\widehat\mu_{n,\bs{h},p}(x) 
	= \sum_{\bs{\gamma}\in\{0,\dots,\ell\}^d} 
	\bs{u}_{\bs{\gamma}}\,
	\widehat{\nu}_{n,\bs{h},p}^{(R)}\big(x+\tau_n \bs{\gamma}/p^{1/2}\big),
\end{equation*}
where $\bs{\gamma} = (\gamma_1,\dots,\gamma_d)$ and $\bs{u}_{\bs{\gamma}} = \prod_{i=1}^d u_{\gamma_i}$. For $p=1$, the convention above and stationarity give $\mathbb{E}_{\overline{\mu}^b}^b[\widehat{\nu}^R_{n,\bs{h},1}(x)]=(\overline{\mu}^b*\varphi_{\tau_n}*\bs{K}_{\bs{h}})(x)$, so the Richardson term is absent; for $p\geq 2$, by Lemma~\ref{lemma:bias:Richardson}, the only remaining term in the bias of $\widehat\mu_{n,\bs{h},p}(x)$ to control is
\begin{equation}
\label{eq:goal:debias}
\Big|
	 \sum_{\bs{\gamma}\in\{0,\dots,\ell\}^d} \bs{u}_{\bs{\gamma}}\,
	\mathbb{E}\big[\overline{\mu}^b\big(x+ \tfrac{\tau_n}{p^{1/2}} (\bs{\gamma}-\zeta)\big)\big] 
	- \overline{\mu}^b(x)\Big|,
\end{equation}
where $\zeta$ is a standard $d$-dimensional Gaussian vector.

We first prove a one-dimensional auxiliary lemma.

\begin{lemma}
\label{lemma:debias:d1}
Let $f : \mathbb{R}\to\mathbb{R}$ be an $\alpha$-H\"older function with H\"older constant $\mathcal{L}$, and let $\zeta$ be a standard real Gaussian variable. Then, for any $x\in\mathbb{R}$ and $\tau>0$,
\begin{equation*}
	\bigg|  \sum_{\gamma = 0}^l u_\gamma\, \mathbb{E}\big[f\big(x + \tau (\gamma - \zeta)\big)\big]- f(x)\bigg| 
	\leq \tau^\alpha \frac{\mathcal{L}}{\lfloor \alpha\rfloor!}\sum_{\gamma=0}^l\mathbb{E}\big[|\gamma-\zeta|^{\alpha}\big].
\end{equation*}
\end{lemma}

\begin{proof}
By Taylor’s formula with integral remainder and using that $f$ is $\alpha$-H\"older, we have for some $\varepsilon\in[0,1]$:
\begin{align*}
 \sum_{\gamma = 0}^l u_\gamma\, \mathbb{E}\big[ f(x + \tau (\gamma - \zeta)) \big] 
 &= \sum_{\gamma = 0}^l u_\gamma \sum_{k = 0}^{\lfloor \alpha\rfloor -1} \frac{f^{(k)}(x)}{k!}\,\mathbb{E}[(\tau(\gamma - \zeta))^k] + R_{\alpha, l}(\varepsilon) \\
 &= \sum_{k = 0}^{\lfloor \alpha \rfloor -1} \frac{f^{(k)}(x)}{k!}
 \sum_{\gamma = 0}^l u_\gamma \sum_{\beta = 0}^k \frac{k!}{\beta!(k-\beta)!} (-1)^\beta m_\beta \gamma^{k-\beta} +R_{\alpha, l}(\varepsilon),
\end{align*}
where $m_\beta$ is the $\beta$-th moment of a standard Gaussian and
\[
R_{\alpha, l}(\varepsilon) 
= \mathbb{E}\Big[\sum_{\gamma=0}^l u_\gamma
\frac{f^{(\lfloor\alpha\rfloor)}(x + \varepsilon\tau(\gamma-\zeta))}{(\lfloor \alpha\rfloor)!} 
(\tau(\gamma-\zeta))^{\lfloor \alpha\rfloor}\Big].
\]
By the defining property~\eqref{eq:def:u} of the coefficients $(u_\gamma)_{0\le\gamma\le l}$, the finite sum over $k$ reduces to $f(x)$. Moreover, since $f$ is $\alpha$-H\"older,
\[
|f^{(\lfloor\alpha\rfloor)}(x+\varepsilon\tau(\gamma-\zeta)) - f^{(\lfloor\alpha\rfloor)}(x)|
\lesssim \mathcal{L} |\tau(\gamma-\zeta)|^{\alpha-\lfloor\alpha\rfloor},
\]
and hence
\[
|R_{\alpha,l}(\varepsilon)| \lesssim 
\tau^{\alpha} \frac{\mathcal{L}}{\lfloor \alpha\rfloor!}\sum_{\gamma=0}^l\mathbb{E}\big[|\gamma-\zeta|^{\alpha}\big].
\]
This yields the desired bound (up to a change of the constant), and the lemma follows.
\end{proof}

We now extend Lemma~\ref{lemma:debias:d1} to the $d$-dimensional setting by induction over the coordinates. Set
\[
\theta(x,\bs{\gamma},n,p) 
= x + \frac{\tau_n}{p^{1/2}}\bs{\gamma} \in \mathbb{R}^d,
\qquad
\bs{\gamma}=(\gamma_1,\dots,\gamma_d).
\]
We claim that for each $0\le I\le d$ and any fixed indices $0\le \gamma_{I+1},\dots,\gamma_d\le l$,
\begin{align*}
&\Big|
	\sum_{\gamma_I = 0}^l \cdots \sum_{\gamma_1 = 0}^l 
	\prod_{i=1}^I u_{\gamma_i}\,
	\mathbb{E}\big[\overline{\mu}^b(\theta(x,\bs{\gamma},n,p) -\tfrac{\tau_n}{p^{1/2}}\zeta)\big] \\
&\hspace{4em}- 
	\mathbb{E}\big[\overline{\mu}^b\big(x_1,\dots,x_I,
	\theta_{I+1} -\tfrac{\tau_n}{p^{1/2}}\zeta_{I+1},\dots,
	\theta_d -\tfrac{\tau_n}{p^{1/2}}\zeta_d\big)\big]
\Big|
\;\lesssim\; \sum_{i=1}^I \Big(\frac{\tau_n}{p^{1/2}}\Big)^{\alpha_i},
\end{align*}
where $\theta_j$ denotes the $j$-th coordinate of $\theta(x,\bs{\gamma},n,p)$.

For $I=0$ this is trivial. For $I=d$ the second expectation no longer contains any random variable and equals $\overline{\mu}^b(x)$, so the claim with $I=d$ is exactly~\eqref{eq:goal:debias}.

Assume now that the claim holds for some $I<d$ and fix $0 \le \gamma_{I+2}, \dots, \gamma_{d} \leq l$. We have
\begin{align*}
& \sum_{\gamma_{I+1} = 0}^l \cdots \sum_{\gamma_1 = 0}^l 
	\prod_{i=1}^{I+1} u_{\gamma_i}\,
	\mathbb{E}\big[\overline{\mu}^b(\theta(x,\bs{\gamma},n,p)-\tfrac{\tau_n}{p^{1/2}}\zeta)\big] \\
&\qquad=
\sum_{\gamma_{I+1} = 0}^l u_{\gamma_{I+1}} 
\sum_{\gamma_{I} = 0}^l\cdots \sum_{\gamma_1 = 0}^l 
	\prod_{i=1}^{I}u_{\gamma_i}\,
	\mathbb{E}\big[\overline{\mu}^b(\theta(x,\bs{\gamma},n,p)-\tfrac{\tau_n}{p^{1/2}}\zeta)\big].
\end{align*}
By the induction hypothesis, the inner sum is close to
\[
\mathbb{E}\Big[\overline{\mu}^b\big(x_1,\dots,x_I,
\theta_{I+1} -\tfrac{\tau_n}{p^{1/2}}\zeta_{I+1},\dots,
\theta_d -\tfrac{\tau_n}{p^{1/2}}\zeta_d\big)\,\Big|\,\zeta_{I+2},\dots,\zeta_d\Big],
\]
up to an error bounded by $\sum_{i=1}^I (\tau_n/p^{1/2})^{\alpha_i}$. 

By Assumption~\ref{assumption:isotropic}, $\overline{\mu}^b$ is $\alpha_{I+1}$-H\"older in its $(I+1)$-th coordinate. Conditioning on $(\zeta_{I+2},\dots,\zeta_d)$ and applying Lemma~\ref{lemma:debias:d1} in the $(I+1)$-th variable, we obtain
\begin{align*}
&\Big| \sum_{\gamma_{I+1} = 0}^l u_{\gamma_{I+1}} 
	\mathbb{E}\big[\overline{\mu}^b(x_1, \dots, x_I, 
	\theta_{I+1} -\tfrac{\tau_n}{p^{1/2}}\zeta_{I+1}, \dots ,
	\theta_{d} -\tfrac{\tau_n}{p^{1/2}}\zeta_{d})\big| \zeta_{I+2}, \dots, \zeta_{d}\big] \\
&\hspace{4em}- \overline{\mu}^b\big(x_1, \dots, x_{I+1}, 
\theta_{I+2} -\tfrac{\tau_n}{p^{1/2}}\zeta_{I+2}, \dots ,
\theta_{d} -\tfrac{\tau_n}{p^{1/2}}\zeta_{d}\big)
\Big| \\
&\qquad\lesssim \Big(\frac{\tau_n}{p^{1/2}}\Big)^{\alpha_{I+1}},
\end{align*}
with a constant that is uniform in $n,p$ and $b$. Combining this with the induction hypothesis yields the desired bound at level $I+1$.

By induction on $I$, we obtain the claim for $I=d$, which is precisely~\eqref{eq:goal:debias}. This completes the proof of Proposition~\ref{propo:small:bias:mu}.

\section{Proof of Proposition~\ref{prop:variance}}
\label{sec:appendix:prop:variance}

Throughout this proof, we use the notation $A \lesssim B$ when there exists a constant $C$ depending only on $\bs{\mathfrak{b}}$ such that $A \leq CB$.

\subsection{Structure and completion of the proof}

First note that it is enough to prove Proposition~\ref{prop:variance} for $\widehat\nu_{n,\bs{h},p}(x)$. The passage to $\widehat\mu_{n,\bs{h},p}(x)$ follows by the finite Richardson/deconvolution linear combination in~\eqref{eq:def:est:mu}; the corresponding variance is bounded by the square of the $\ell^1$ coefficient norm times the largest variance of the shifted $\widehat\nu$ terms.
We now introduce the main notations used in this proof. Let $n_p = \lfloor n/p \rfloor$. For each $k$, we write
\begin{equation*}
    \widetilde\xi_k = \frac{1}{\sqrt{p}}\sum_{\ell=0}^{p-1}\xi_{kp+\ell},
\end{equation*}
so that $(\widetilde\xi_k)_k$ is a sequence of i.i.d.\ standard $d$-dimensional Gaussian variables. Then we write
\begin{equation*}
	\mathfrak{C}^b_p(k ; x) = \mathrm{Cov}_{\overline{\mu}^b}^b\Big( \bs{K}_{\bs{h}}\Big(x -  p^{-1}\sum_{\ell=0}^{p-1}X_{(kp + \ell)\Delta_n} + \frac{\tau_n}{\sqrt{p}} \widetilde\xi_{k} \Big)\: , \: \bs{K}_{\bs{h}}\Big(x -  p^{-1}\sum_{\ell=0}^{p-1}X_{ \ell\Delta_n} + \frac{\tau_n}{\sqrt{p}} \widetilde\xi_{0} \Big)\Big),
\end{equation*}
for each $k$ and $x$, and 
\begin{equation*}
        \mathfrak{D}^b_p(k;x,u,v) = \E^b_{\overline{\mu}^b}\Big[\bs{K}_{\bs{h}}\Big(x -  \frac{1}{p}\sum_{\ell=0}^{p-1}X_{(kp + \ell)\Delta_n} + u\Big)\bs{K}_{\bs{h}}\Big(x -  \frac{1}{p}\sum_{\ell=0}^{p-1}X_{\ell\Delta_n} + v\Big)\Big],
\end{equation*}
for each $k$, $x$, $u$ and $v$. By stationarity of 
\begin{equation*}
    \Big(p^{-1}\sum_{\ell=0}^{p-1}X_{(kp + \ell)\Delta_n} - \frac{\tau_n}{\sqrt{p}} \widetilde\xi_{k}\Big)_k
\end{equation*}
we see that
\begin{equation*}
	\mathfrak{C}^b_p(k ; x) \leq \E^b_{\overline{\mu}^b}\Big[\bs{K}_{\bs{h}}\Big(x -  \frac{1}{p}\sum_{\ell=0}^{p-1}X_{(kp + \ell)\Delta_n} + \frac{\tau_n}{\sqrt{p}} \widetilde\xi_{k}\Big)\bs{K}_{\bs{h}}\Big(x -  \frac{1}{p}\sum_{\ell=0}^{p-1}X_{\ell\Delta_n} + \frac{\tau_n}{\sqrt{p}} \widetilde\xi_{0}\Big)\Big].
\end{equation*}
Moreover, for $k>0$, the noises $\widetilde\xi_{0}$ and $\widetilde\xi_{k}$ are independent and therefore
\begin{equation}
    \label{eq:bound:C:Duv}
	\mathfrak{C}^b_p(k;x) \leq \int_{\R^{d}}\int_{\R^{d}}\mathfrak{D}^b_p(k;x,u,v)  \phi_{\tau_n / p^{1/2}}(u) \phi_{\tau_n / p^{1/2}}(v) \, \mathrm{d}u \, \mathrm{d}v.
\end{equation}

By definition of $\widehat\nu_{n,\bs{h},p}(x)$, we have
\begin{align*}
	& \mathrm{Var}_{\overline{\mu}^b}^b(\widehat\nu_{n,\bs{h},p}(x)) = \mathrm{Var}_{\overline{\mu}^b}^b\Big( \frac{1}{n_p}\sum_{k=0}^{n_p-1} \bs{K}_{\bs{h}}\Big(x - \frac{1}{p}\sum_{\ell=0}^{p-1}Y_{kp + \ell,n}\Big)\Big) \\
	& = \frac{1}{n_p^2}\sum_{i, j=0}^{n_p-1}
 \mathrm{Cov}_{\overline{\mu}^b}^b\Big( \bs{K}_{\bs{h}}\Big(x -  \frac{1}{p}\sum_{\ell=0}^{p-1}X_{(ip + \ell)\Delta_n} + \frac{\tau_n}{\sqrt{p}} \widetilde\xi_{i}\Big) \, , \bs{K}_{\bs{h}}\Big(x -  \frac{1}{p}\sum_{\ell=0}^{p-1}X_{(jp + \ell)\Delta_n} + \frac{\tau_n}{\sqrt{p}} \widetilde\xi_{j} \Big)\Big).
\end{align*}

Since $X$ and $\xi$ are both ergodic and mutually independent, we deduce that
\begin{equation}
\label{eq:variancedevelopped}
		\mathrm{Var}_{\overline{\mu}^b}^b\Big( \frac{1}{n_p}\sum_{k=0}^{n_p-1} \bs{K}_{\bs{h}}\Big(x - \frac{1}{p}\sum_{\ell=0}^{p-1}Y_{kp + \ell,n}\Big)\Big) = \frac{1}{n_p}\mathfrak{C}^b_p(0;x) + \frac{2}{n_p^2}\sum_{k=1}^{n_p-1} (n_p-k) \mathfrak{C}^b_p(k ; x).
\end{equation}

We now want to bound the coefficients $ \mathfrak{C}^b_p(k ; x)$. The  bounds rely crucially on the results of Section~\ref{sec:hidden}. In particular, let $\mathfrak{q}_{p, k}^*(x; \cdot)$ denote the probability density function of $\frac{1}{p}\sum_{\ell = 0}^{p-1} X_{(kp+\ell)\Delta_n}$ given  $\frac{1}{p}\sum_{\ell = 0}^{p-1} X_{\ell\Delta_n} = x$, under $\mathbb{P}^{b}_{\wb \mu^b}$. This means that the joint distribution of $(\frac{1}{p}\sum_{\ell = 0}^{p-1} X_{\ell\Delta_n}, \frac{1}{p}\sum_{\ell = 0}^{p-1} X_{(kp+\ell)\Delta_n})$ is given by 
\begin{equation*}
    (x,y)
    \mapsto
    m_p^{p\Delta_n}(\wb \mu^b)(x)
    \mathfrak{q}_{p, k}^*(x; y).
\end{equation*}
Using Corollaries~\ref{corrolary:bounds:2} and~\ref{corrolary:bounds:3}, we know that for each $c_0 > 0$, there exist constants $\wt c_0 >0$, 
$\wt \lambda_i \in (0, 1]$ and $\wt C_i > 0$ for $i \in \{0, 1\}$ and $t_0^*$ such that 
if $p^2 \Delta_n \leq \wt c_0$, then for any $k\ge3$ satisfying $(k-1)p\Delta_n \leq t_0^*$ and any $x, y \in \mathbb{R}^d$, 
we have
\begin{equation}
\label{eq:qstar:bound:short}
    \mathfrak{q}_{p, k}^*(x; y)
    \leq 
    \widetilde{C}_0 \, g_{\widetilde{\lambda}_0, d}((k-1)p\Delta_n, x - y) 
    \frac{\exp\!\big(c_0|x|^2\big)}{m_p^{p\Delta_n}(\wb \mu^b)(x)}.
\end{equation}
The cases $k=1$ and $k=2$ are not covered by the short-time Gaussian estimate because the separating gaps are respectively $0$ and $p\Delta_n$, respectively. We handle these terms using the instantaneous-correlation bound of Lemma~\ref{lem:corr:inst} and for all $(k-1)p\Delta_n \geq t_0^*$, we have
\begin{equation}
\label{eq:qstar:bound:long}
   \mathfrak{q}_{p, k}^*(x; y)
    \leq \widetilde{C}_0 \exp\!\big(c_0|y|^2\big).
\end{equation}
In the following, we fix $c_0 > 0$ small enough so that
\begin{equation*}
    3c_0 < 1 / (2\tau^{*2}) \quad \text{ and } \quad 3c_0 < c_K.
\end{equation*}

We now present several bounds on the coefficients $ \mathfrak{C}^b_p(k ; x)$.

\begin{lemma}[Instantaneous correlations]
\label{lem:corr:inst}
    For each $k \geq 0$, we have
    \begin{equation*}
        \mathfrak{C}^b_p(k ; x) \lesssim \prod_{i=1}^d h_i^{-1}.
    \end{equation*}
\end{lemma}

\begin{lemma}[Short-term correlations]
    \label{lem:corr:short}
    For any $1 \leq k_1 \leq d$ and any $k \geq 3$ such that $(k-1)p\Delta_n \leq t_0^*$, we have
    \begin{equation*}
        \mathfrak{C}^b_p(k ; x)  \lesssim (kp\Delta_n)^{-k_1/2} \prod_{i=k_1+1}^d h_i^{-1}.
    \end{equation*}
\end{lemma}

\begin{lemma}[Mid-term correlations]
    \label{lem:corr:mid}
    For any $k\geq 1$ such that $(k-1)p\Delta_n \geq t_0^*$, we have $\mathfrak{C}^b_p(k ; x)  \lesssim 1$.
\end{lemma}

\begin{lemma}[Long-term correlations]
    \label{lem:corr:long}
    For any $k\geq 1$, we have
\begin{equation*}
    \mathfrak{C}^b_p(k;x) \lesssim e^{-C_{PI}^{-1}kp\Delta_n}\Big( \prod_{i=1}^d h_i^{-1}\Big)^2.
\end{equation*}
\end{lemma}

Following~\cite{amorino2023estimation}, the rest of the proof consists in identifying the best way to combine these lemmas with~\eqref{eq:variancedevelopped}, depending on $d$ and $k_0$. The finitely many edge indices $k=1,2$ are included in the instantaneous part; this only adds a term of order $n_p^{-1}\prod_i h_i^{-1}$, already present in the displayed bounds below. Recall that $k_0$ is such that $\alpha_1 = \dots = \alpha_{k_0} < \alpha_{k_0+1}\leq \dots \leq \alpha_d$. We now introduce 
$0 \leq j_1 \leq j_2 \leq j_3 \leq j_4$ such that $j_3 = \lceil t_0^* (p \Delta_n)^{-1} \rceil$ and 
\begin{equation*}
    j_4 =  \max \left( \min\left(\left\lfloor\frac{-2C_{PI} \log\prod_{i=1}^d h_i }{p\Delta_n}\right\rfloor, n_p \right), j_3\right).
\end{equation*}
Then we have
\begin{align*}
	\frac{1}{n_p}\mathfrak{C}_p^b(0;x) +& \frac{2}{n_p^2} \sum_{k=1}^{n_p-1} (n_p - k) \mathfrak{C}_p^b(k ; x) \\
    & = \frac{1}{n_p}\mathfrak{C}_p^b(0;x) + \frac{2}{n_p^2}\Big( \sum_{k=1}^{j_1} + \sum_{k=j_1 + 1}^{j_2} +  \sum_{k=j_2 + 1}^{j_3} +  \sum_{k=j_3 + 1}^{j_4} +  \sum_{k=j_4 + 1}^{n_p - 1}\Big) (n_p-k)\mathfrak{C}_p^b(k ; x).
\end{align*}
For the first sum, we use Lemma~\ref{lem:corr:inst} so that we have
\begin{equation*}
    \frac{1}{n_p}\mathfrak{C}_p^b(0;x) + \frac{2}{n_p^2}\sum_{k=1}^{j_1} (n_p - k) \mathfrak{C}_p^b(k ; x) \lesssim \frac{1}{n_p \prod_{i=1}^d h_i} (j_1 + 1).
\end{equation*}
For the second and third sums, we use Lemma~\ref{lem:corr:short} for some $1 \leq k_1 \leq d$ and $1 \leq k_2 \leq d$ to be chosen later. In that case, we get
\begin{empheq}[left=\hspace*{3cm}\empheqlbrace]{flalign}
  &n_p^{-2} \sum_{k=j_1+1}^{j_2} (n_p - k) \mathfrak{C}_p^b(k ; x) \lesssim n_p^{-1}  \sum_{k=j_1+1}^{j_2} (kp\Delta_n)^{-k_1/2} \prod_{i=k_1+1}^d h_i^{-1},& \nonumber\\
  &n_p^{-2}  \sum_{k=j_2+1}^{j_3} (n_p - k) \mathfrak{C}_p^b(k ; x) \lesssim n_p^{-1} \sum_{k=j_2+1}^{j_3}  (kp\Delta_n)^{-k_2/2} \prod_{i=k_2+1}^d h_i^{-1}.\nonumber
\end{empheq}
Since $j_3p\Delta_n\ge t_0^*$, we use Lemma~\ref{lem:corr:mid} for the fourth sum and we get
\begin{equation*}
        \frac{1}{n_p^2} \sum_{k=j_3+1}^{j_4} (n_p - k) \mathfrak{C}_p^b(k ; x) \lesssim \frac{1}{n_p} (j_4 - j_3) .
\end{equation*}
Using the definition of $j_4$, we have
\begin{equation*}
   \frac{1}{n_p^2} \sum_{k=j_3+1}^{j_4} (n_p - k) \mathfrak{C}_p^b(k ; x) 
   \lesssim  \min\left(\frac{-2C_{PI} \log\prod_{i=1}^d h_i }{n_p p\Delta_n}, 1 \right).
\end{equation*}
Using also that $n_p p \Delta_n \gtrsim T_n$, we get
\begin{equation}
\label{eq:bound:sum4}
   \frac{1}{n_p^2} \sum_{k=j_3+1}^{j_4} (n_p - k) \mathfrak{C}_p^b(k ; x) 
   \lesssim  \min\left(\frac{\sum_{i=1}^d | \log h_i | }{T_n}, 1 \right).
\end{equation}

For the last sum, we use Lemma~\ref{lem:corr:long} and we have
\begin{align*}
\frac{1}{n_p^2} \sum_{k=j_4+1}^{n_p-1} (n_p - k) \mathfrak{C}_p^b(k ; x) 
&\lesssim \frac{1}{n_p} \sum_{k=j_4+1}^{n_p-1} \mathfrak{C}_p^b(k ; x)
\\
&\lesssim \frac{1}{n_p} \sum_{k=j_4+1}^{n_p-1}
e^{-C_{PI}^{-1}(k-1)p\Delta_n}\Big( \prod_{i=1}^d h_i^{-1}\Big)^2\\
&\lesssim \frac{1}{T_n}
e^{-C_{PI}^{-1}j_4p\Delta_n}\Big( \prod_{i=1}^d h_i^{-1}\Big)^2.\end{align*}
Using the definition of $j_4$, we have
\begin{equation}
\label{eq:bound:sum5}
    \frac{1}{n_p^2} \sum_{k=j_4+1}^{n_p-1} (n_p - k) \mathfrak{C}_p^b(k ; x) 
    \lesssim 
    \frac{1}{T_n}
    +
     \frac{1}{T_n}
    e^{-C_{PI}^{-1} T_n}\Big( \prod_{i=1}^d h_i^{-1}\Big)^2. 
\end{equation}
\\

\textit{Case $d=1$ or $d=2$}. Taking $j_1 = 0$, $j_2 = j_3$ and $k_1 = d$, we obtain
\begin{align*}
     \frac{2}{n_p^2}\sum_{k=j_1+1}^{j_2} \mathfrak{C}_p^b(k ; x) 
     &\lesssim \frac{2}{n_p^2}\sum_{k=1}^{j_2} (n_p - k) \frac{1}{kp\Delta_n}
     \\
    &\lesssim \frac{\log(j_2)}{n_p p \Delta_n} = \frac{|\log(p\Delta_n)|}{T_n}.
\end{align*}
Therefore
\begin{equation*}
	\mathrm{Var}_{\overline{\mu}^b}^b(\widehat\nu_{n,\bs{h},p}(x))
 \lesssim 
 \frac{p\Delta_n}{T_n}\prod_{i=1}^d h_i^{-1} + \frac{|\log(p\Delta_n)|}{T_n} + \sum_{i=1}^d\frac{|\log(h_i)|}{T_n} + \frac{1}{T_n} e^{-CT_n} \left( \prod_{i=1}^d h_i^{-1}\right)^{2}. 
\end{equation*}

\textit{Case $d \geq 3$, $k_0 = 1$ and $\alpha_2 < \alpha_3$ or $k_0 = 2$. } We take $k_1=2$ and $k_2 = d$ and 
\begin{equation*}
    j_1 = \left\lfloor \frac{h_1h_2}{p\Delta_n} \right\rfloor
    \;\;
    \text{ and }
    \;\;
	j_2 = \left\lfloor \frac{(\prod_{i \geq 3} h_i)^{\frac{2}{d-2}}}{p\Delta_n}\right\rfloor.
\end{equation*}
Therefore we have
\begin{align*}
    \frac{1}{n_p^2} \sum_{k=0}^{j_1} (n_p - k) \mathfrak{C}_p^b(k ; x) 
    &\lesssim \frac{1}{n_p \prod_{i=1}^d h_i} \frac{h_1h_2}{p\Delta_n} + \frac{1}{n_p}\prod_{i=1}^dh_i^{-1}
\\
    &\lesssim \frac{1}{T_n \prod_{i=3}^d h_i} + \frac{1}{n_p}\prod_{i=1}^dh_i^{-1}.
\end{align*}
Moreover,
\begin{align*}
n_p^{-2} \sum_{k=j_1+1}^{j_2} (n_p - k) \mathfrak{C}_p^b(k ; x) &\lesssim n_p^{-1}  \sum_{k=j_1+1}^{j_2} (kp\Delta_n)^{-1} \prod_{i=3}^d h_i^{-1}
\\
&\lesssim n_p^{-1} \log(j_2/j_1)(p\Delta_n)^{-1} \prod_{i=3}^d h_i^{-1}
\\
&\lesssim T_n^{-1} \sum_{i=1}^d |\log(h_i) \prod_{i=3}^d h_i^{-1}
\end{align*}
and
\begin{align*}
    n_p^{-2}  \sum_{k=j_2+1}^{j_3} (n_p - k) \mathfrak{C}_p^b(k ; x) &\lesssim n_p^{-1} \sum_{k=j_2+1}^{j_3}  (kp\Delta_n)^{-d/2}
    \\
    &\lesssim n_p^{-1} j_2^{1-d/2} (p\Delta_n)^{-d/2}
    \\
    &\lesssim T_n^{-1} (\prod_{i \geq 3} h_i)^{-1}.
\end{align*}
Combining these three bounds with~\eqref{eq:bound:sum4} and~\eqref{eq:bound:sum5}, we get
\begin{equation*}
	\mathrm{Var}_{\overline{\mu}^b}^b(\widehat\nu_{n,\bs{h},p}(x))
 \lesssim 
 \frac{1}{T_n \prod_{i=3}^d h_i} + \frac{1}{n_p}\prod_{i=1}^dh_i^{-1}
 +
  \frac{1}{T_n} \sum_{i=1}^d |\log(h_i)| \prod_{i=3}^d h_i^{-1},
\end{equation*}
which proves Proposition~\ref{prop:variance} in the case $d=1,2$.
\\

\textit{Case $d \geq 3$, $k_0 \geq 3$. } We take $k_1=k_0$ and $k_2=d$ and
\begin{equation*}
    j_1 = \left\lfloor \frac{\Big(\prod_{i=1}^{k_0} h_i\Big)^{2/k_0}}{p\Delta_n} \right\rfloor
    \;\;
    \text{ and }
    \;\;
    j_2 = \left\lfloor \frac{1}{p\Delta_n}\right\rfloor.
\end{equation*}
Therefore, we have
\begin{align*}
    \frac{1}{n_p^2} \sum_{k=0}^{j_1} (n_p - k) \mathfrak{C}_p^b(k ; x) 
    &\lesssim \frac{1}{n_p \prod_{i=1}^d h_i} \frac{\Big(\prod_{i=1}^{k_0} h_i\Big)^{2/k_0}}{p\Delta_n} + \frac{1}{n_p}\prod_{i=1}^dh_i^{-1}
\\
&\lesssim \frac{1}{T_n} \prod_{i=1}^{k_0} (h_i)^{(2-k_0)/k_0}  \prod_{i=k_0+1}^d h_i^{-1} + \frac{1}{n_p}\prod_{i=1}^dh_i^{-1}
\end{align*}
and
\begin{align*}
n_p^{-2} \sum_{k=j_1+1}^{j_2} (n_p - k) \mathfrak{C}_p^b(k ; x) &\lesssim n_p^{-1}  \sum_{k=j_1+1}^{j_2} (kp\Delta_n)^{-k_0/2} \prod_{i=k_0+1}^d h_i^{-1}
\\
&\lesssim n_p^{-1} j_1^{1-k_0/2} (p\Delta_n)^{-k_0/2} \prod_{i=k_0+1}^d h_i^{-1}
\\
&\lesssim T_n^{-1} \prod_{i=1}^{k_0} (h_i)^{(2-k_0)/k_0}  \prod_{i=k_0+1}^d h_i^{-1}
\end{align*}
and using that $j_3 \leq j_2 + 1$
\begin{equation*}
    n_p^{-2}  \sum_{k=j_2+1}^{j_3} (n_p - k) \mathfrak{C}_p^b(k ; x)
    \lesssim n_p^{-1} (j_2+1)^{-d/2} (p\Delta_n)^{-d/2}
\lesssim \frac{1}{n_p}.
\end{equation*}

Combining these three bounds with~\eqref{eq:bound:sum4} and~\eqref{eq:bound:sum5}, we get
\begin{equation*}
	\mathrm{Var}_{\overline{\mu}^b}^b(\widehat\nu_{n,\bs{h},p}(x))
 \lesssim 
T_n^{-1} \prod_{i=1}^{k_0} (h_i)^{(2-k_0)/k_0}  \prod_{i=k_0+1}^d h_i^{-1} + \frac{1}{n_p}\prod_{i=1}^dh_i^{-1} + \frac{\sum_{i=1}^d | \log h_i | }{T_n},
\end{equation*}
which proves Proposition~\ref{prop:variance} in the case $d \geq 3$, $k_0 \geq 3$.\\

\textit{Case $d \geq 3$, $k_0 = 1$ and $\alpha_2 = \alpha_3$}. We take $k_1=1$ and $k_2=3$,
\begin{equation*}
    j_1 = 0 
    \;\;
    \text{ and }
    \;\;
    j_2 = \left\lfloor \frac{h_2h_3}{p\Delta_n}\right\rfloor.
\end{equation*}
Therefore, we have
\begin{equation*}
    \frac{1}{n_p^2} \sum_{k=0}^{j_1} (n_p - k) \mathfrak{C}_p^b(k ; x) 
    \lesssim \frac{1}{n_p \prod_{i=1}^d h_i}.
\end{equation*}
Then
\begin{align*}
n_p^{-2} \sum_{k=j_1+1}^{j_2} (n_p - k) \mathfrak{C}_p^b(k ; x) &\lesssim n_p^{-1}  \sum_{k=j_1+1}^{j_2} (kp\Delta_n)^{-1/2} \prod_{2}^d h_i^{-1}
\\
&\lesssim n_p^{-1} j_2^{1/2} (p\Delta_n)^{-1/2} \prod_{i=2}^d h_i^{-1}
\\
&\lesssim T_n^{-1} (h_2 h_3)^{-1/2} \prod_{i=4}^d h_i^{-1}
\end{align*}
and
\begin{align*}
n_p^{-2} \sum_{k=j_2+1}^{j_3} (n_p - k) \mathfrak{C}_p^b(k ; x) &\lesssim n_p^{-1}  \sum_{k=j_2+1}^{j_3} (kp\Delta_n)^{-3/2} \prod_{i=4}^d h_i^{-1}
\\
&\lesssim n_p^{-1} j_2^{-1/2} (p\Delta_n)^{-3/2} \prod_{i=4}^d h_i^{-1}
\\
&\lesssim T_n^{-1} (h_2 h_3)^{-1/2} \prod_{i=4}^d h_i^{-1}
\end{align*}

Combining these three bounds with~\eqref{eq:bound:sum4} and~\eqref{eq:bound:sum5}, we get
\begin{equation*}
	\mathrm{Var}_{\overline{\mu}^b}^b(\widehat\nu_{n,\bs{h},p}(x))
 \lesssim 
 \frac{\sum_{i=1}^d | \log h_i | }{T_n}
    +
    \frac{1}{n_p\prod_{i=1}^d h_i}
    +
    \frac{1}{T_n\sqrt{h_2 h_3}}\prod_{i=4}^d h_i^{-1},
\end{equation*}
which proves Proposition~\ref{prop:variance} in the case $d \geq 3$, $k_0 = 1$ and $\alpha_2 = \alpha_3$.

\subsection{Proof of Lemma~\ref{lem:corr:inst}}

Since $X$ and $\xi$ are both ergodic and mutually independent, we have by Cauchy--Schwarz inequality
\begin{align*}
	\mathfrak{C}^b_p(k ; x)
 & \leq \mathrm{Var}_{\overline{\mu}^b}^b\Big(\bs{K}_{\bs{h}}\Big(x -  p^{-1}\sum_{\ell=0}^{p-1}X_{(kp + \ell)\Delta_n} + \frac{\tau_n}{\sqrt{p}} \widetilde\xi_{k}\Big)\Big)  
 \\
 &\leq \mathbb{E}_{\overline{\mu}^b}^b \Big[\bs{K}_{\bs{h}}\Big(x -  p^{-1}\sum_{\ell=0}^{p-1}X_{(kp + \ell)\Delta_n} + \frac{\tau_n}{\sqrt{p}} \widetilde\xi_{k}\Big)^2 \Big]
 \\
 &\leq \int_{\mathbb{R}^d} \mathbb{E}_{\overline{\mu}^b}^b \Big[\bs{K}_{\bs{h}}\Big(x -  p^{-1}\sum_{\ell=0}^{p-1}X_{(kp + \ell)\Delta_n} + u\Big)^2 \Big] \phi_{\tau_n / p^{1/2}}(u) \, \mathrm{d}u.
\end{align*}
Using the notations of Sections~\ref{sec:hidden} and~\ref{sec:estimation}, we see that
\begin{equation*}
    \mathbb{E}_{\overline{\mu}^b}^b \Big[\bs{K}_{\bs{h}}\Big(x -  p^{-1}\sum_{\ell=0}^{p-1}X_{(kp + \ell)\Delta_n} + u\Big)^2 \Big]
    =
    \int_{\mathbb{R}^d} m^{p\Delta_n}(\overline{\mu}^b)(z) \bs{K}_{\bs{h}}(x + u -  z)^2 \, \mathrm{d}z.
\end{equation*}
A change of variables and Corollary~\ref{corrolary:bounds:4} then yield
\begin{align*}
    \mathbb{E}_{\overline{\mu}^b}^b \Big[\bs{K}_{\bs{h}}\Big(x -  p^{-1}\sum_{\ell=0}^{p-1}X_{(kp + \ell)\Delta_n} + u\Big)^2 \Big]
    &=
    \Big( \prod_{i=1}^d h_i \Big)^{-1}
    \int_{\mathbb{R}^d} m^{p\Delta_n}(\overline{\mu}^b)(x + u - \bs{h} \odot z) \bs{K}_{\bs{h}}(z)^2 \, \mathrm{d}z.
    \\
    &\leq
    \Big( \prod_{i=1}^d h_i \Big)^{-1}
    \int_{\mathbb{R}^d} \exp(c_0|x + u - \bs{h} \odot z|^2) \bs{K}(z)^2 \, \mathrm{d}z,
\end{align*}
where we use the shorthand $\bs{h} \odot z = (h_1 z_1,\dots,h_d z_d)$. Using also
\begin{equation*}
    |x + u - \bs{h} \odot z|^2 \leq 3 |x|^2 + 3 |u|^2 + 3 |\bs{h} \odot z|^2,
\end{equation*}
and using that $h_i \leq 1$ for all $i$, we see that
\begin{equation*}
        \mathbb{E}_{\overline{\mu}^b}^b \Big[\bs{K}_{\bs{h}}\Big(x -  p^{-1}\sum_{\ell=0}^{p-1}X_{(kp + \ell)\Delta_n} + u\Big)^2 \Big]
    \leq
    \exp(3c_0|x|^2 + 3c_0|u|^2)
    \Big( \prod_{i=1}^d h_i \Big)^{-1}
    \int_{\mathbb{R}^d} \exp(3c_0|\bs{h} \odot z|^2) \bs{K}(z)^2 \, \mathrm{d}z.
\end{equation*}
Moreover, we have
\begin{equation*}
    \int_{\mathbb{R}^d} \exp(3c_0|\bs{h} \odot z|^2) \bs{K}(z)^2 \, \mathrm{d}z \lesssim 1
\end{equation*}
because $3c_0 < c_K$ and thus 
\begin{equation*}
	\mathfrak{C}^b_p(k ; x)
  \lesssim
     \Big( \prod_{i=1}^d h_i \Big)^{-1}
\int_{\mathbb{R}^d} \exp(3c_0 |u|^2) \phi_{\tau_n / p^{1/2}}(u) \, \mathrm{d}u 
\end{equation*}
because $3c_0 < 1/(2\tau^{*2})$.
We then compute explicitly
\begin{equation*}
    \int_{\R^{d}} \exp(3c_0|u|^2) \phi_{\tau_n / p^{1/2}}(u) \, \mathrm{d}u
    =
    \Big(1-\frac{6c_0\tau_n^2}{p}\Big)^{-d/2}.
\end{equation*}
Since $6c_0\tau_n^2/p\leq 6c_0\tau^{*2}<1$,
\begin{equation}
\label{eq:bound:exp:phi}
\int_{\R^{d}} \exp(3c_0|u|^2) \phi_{\tau_n / p^{1/2}}(u) \, \mathrm{d}u
    =
    \Big(1-\frac{6c_0\tau_n^2}{p}\Big)^{-d/2}
    \leq
    \Big(1-6c_0\tau^{*2}\Big)^{-d/2}
    <
    \infty.
\end{equation}
We eventually obtain
\begin{equation*}
	\mathfrak{C}^b_p(k ; x)
  \lesssim
     \Big( \prod_{i=1}^d h_i \Big)^{-1}
\int_{\mathbb{R}^d} \exp(3c_0 |u|^2) \phi_{\tau_n / p^{1/2}}(u) \, \mathrm{d}u \lesssim \Big( \prod_{i=1}^d h_i \Big)^{-1},
\end{equation*}
which proves Lemma~\ref{lem:corr:inst}.



\subsection{Proof of Lemma~\ref{lem:corr:short}}

Using the explicit distribution of 
\begin{equation*}
    (\frac{1}{p}\sum_{\ell = 0}^{p-1} X_{\ell\Delta_n}, \frac{1}{p}\sum_{\ell = 0}^{p-1} X_{(kp+\ell)\Delta_n})
\end{equation*}
and Equation~\eqref{eq:qstar:bound:short}, we have
\begin{align*}
	&\mathfrak{D}^b_p(k;x,u,v) \\
	& = 
    \int_{\R^d} \int_{\R^d} 
    \bs{K}_{\bs{h}}(x - y + u) 
    \bs{K}_{\bs{h}}(x - z + v)
    m_p^{p\Delta_n}(\wb \mu^b)(y)
    \mathfrak{q}_{p, k}^*(y; z)
    \, \mathrm{d}z 
    \, \mathrm{d}y
    \\
	& \leq 
    \widetilde{C}_0
    \int_{\R^d} \int_{\R^d} 
    \bs{K}_{\bs{h}}(x - y + u) 
    \bs{K}_{\bs{h}}(x - z + v)
    m_p^{p\Delta_n}(\wb \mu^b)(y)
    g_{\widetilde{\lambda}_0, d}((k-1)p\Delta_n, y - z)
    \frac{\exp\!\big(c_0|y|^2\big)}{m_p^{p\Delta_n}(\wb \mu^b)(y)}
    \, \mathrm{d}z 
    \, \mathrm{d}y
    \\
	& \lesssim
    \int_{\R^d} \int_{\R^d} 
    \bs{K}_{\bs{h}}(x - y + u) 
    \bs{K}_{\bs{h}}(x - z + v)
    g_{\widetilde{\lambda}_0, d}(kp\Delta_n, y - z)
    \exp\!\big(c_0|y|^2\big)
    \, \mathrm{d}z 
    \, \mathrm{d}y
\end{align*}
where we used $k\ge3$ to replace $(k-1)p\Delta_n$ by a comparable multiple of $kp\Delta_n$. We now fix $y$ and focus on the integral in $z$. Note that we can rewrite
\begin{align*}
    \int_{\R^d} 
    \bs{K}_{\bs{h}}(x - z + v)
    g_{\widetilde{\lambda}_0, d}(kp\Delta_n, y - z)
    \, \mathrm{d}z 
    &=
    \frac{1}{\lambda^{d/2}s^{d/2}}
    \int_{\R^d} 
    \bs{K}_{\bs{h}}(x - z + v)
    \exp(-\frac{|y-z|^2}{s})
    \, \mathrm{d}z 
\\
    &=
    \frac{1}{\lambda^{d/2}}
    \prod_{i=1}^{d}
    \Big\{
    \frac{1}{s^{1/2}h_i}
    \int_{\R} 
    K\Big(\frac{x_i - z_i + v_i}{h_i}\Big)
    \exp(-\frac{(y_i-z_i)^2}{s})
    \, \mathrm{d}z_i
    \Big\}
\end{align*}
with 
$s = kp\Delta_n / \widetilde{\lambda}_0$.
Using that $K$ is integrable, we obtain for any $i \in \{1, \dots, d\}$,
\begin{equation}
\begin{aligned}\label{eq:variance:T2:1}
    \frac{1}{s^{1/2}h_i}
\int_{\mathbb{R}} 
 \Big| K\Big(\frac{x_i - z_i + v_i}{h_i}\Big) \Big| \exp\Big(-\frac{|y_i-z_i|^2}{s}\Big) \, \mathrm{d}z_i
&\lesssim 
    \frac{1}{s^{1/2}h_i}
\int_{\mathbb{R}} 
\Big| K\Big(\frac{x_i - z_i + v_i}{h_i}\Big) \Big| \, \mathrm{d}z_i
\\
&\lesssim 
\frac{1}{\sqrt{s}}
\int_{\mathbb{R}} 
\Big| K(z_i) \Big| \, \mathrm{d}z_i
\\
&\lesssim 
\frac{1}{\sqrt{s}}.
\end{aligned}
\end{equation}
Using that $K$ is bounded, we also have
\begin{equation}
\begin{aligned}\label{eq:variance:T2:2}
    \frac{1}{s^{1/2}h_i}
\int_{\mathbb{R}} 
\Big| K\Big(\frac{x_i - z_i + v_i}{h_i}\Big) \Big| \exp\Big(-\frac{|z_i - y_i|^2}{s}\Big) \, \mathrm{d}z_i
\lesssim 
\frac{1}{s^{1/2}h_i}
\int_{\mathbb{R}} 
\exp\Big(-\frac{|z_i-y_i|^2}{s}\Big) \, \mathrm{d}z_i
\lesssim 
\frac{1}{h_i}.
\end{aligned}
\end{equation}
Combining the bounds given by Equations~\eqref{eq:variance:T2:1} and~\eqref{eq:variance:T2:2}, we obtain that for any $1 \leq k_1 \leq d$,
\begin{equation*}
    \int_{\R^d} 
    \bs{K}_{\bs{h}}(x - z + v)
    g_{\widetilde{\lambda}_0, d}(kp\Delta_n, y - z)
    \, \mathrm{d}z 
    \lesssim 
    \frac{1}{s^{k_1/2} \prod_{i=k_1+1}^d h_i},
\end{equation*}
and therefore we have
\begin{equation*}
\mathfrak{D}^b_p(k;x,u,v)
\lesssim
    \frac{1}{s^{k_1/2} \prod_{i=k_1+1}^d h_i} 
    \int_{\R^d} 
    \bs{K}_{\bs{h}}(x - y + u) 
    \exp\!\big(c_0|y|^2\big)
    \, \mathrm{d}y
\end{equation*}
Then a change of variables ensures that
\begin{align*}
    \int_{\R^d} 
    \bs{K}_{\bs{h}}(x - y + u) 
    \exp\!\big(c_0|y|^2\big)
    \, \mathrm{d}y
&=
\prod_{i=1}^{d}
\Big \{
    \frac{1}{h_i}
    \int_{\R} 
    K \Big( \frac{x_i - y_i + u_i}{h_i}\Big) 
    \exp\!\big(c_0y_i^2\big)
    \, \mathrm{d}y
\Big \}
\\
&=
\prod_{i=1}^{d}
\Big \{
    \int_{\R} 
    K ( y_i ) 
    \exp\!\big(c_0(x_i + u_i - h_i y_i)^2\big)
    \, \mathrm{d}y
\Big \}
\\
&\leq
\prod_{i=1}^{d}
\Big \{
\exp(3c_0(x_i^2 + u_i^2))
    \int_{\R} 
    K (  y_i ) 
    \exp\!\big(3 c_0 h_i^2 y_i^2\big)
    \, \mathrm{d}y
\Big \}.
\end{align*}
Therefore, using that $h_i \leq 1$ for all $i$ and that $K$ is sub-Gaussian, we obtain 
\begin{equation*}
        \int_{\R^d} 
    \bs{K}_{\bs{h}}(x - y + u) 
    \exp\!\big(c_0|y|^2\big)
    \, \mathrm{d}y
\lesssim 
\exp(3c_0|u|^2)
\end{equation*}
and thus
\begin{equation*}
\mathfrak{D}^b_p(k;x,u,v)
\lesssim \frac{1}{(kp\Delta_n)^{k_1/2} \prod_{i=k_1+1}^d h_i} \exp(3c_0|u|^2).
\end{equation*}
The results on $\mathfrak{C}^b_p(k;x)$ follow by integrating~\eqref{eq:bound:C:Duv} using~\eqref{eq:bound:exp:phi}.

\subsection{Proof of Lemma~\ref{lem:corr:mid}}

Using the explicit distribution of 
\begin{equation*}
    (\frac{1}{p}\sum_{\ell = 0}^{p-1} X_{\ell\Delta_n}, \frac{1}{p}\sum_{\ell = 0}^{p-1} X_{(kp+\ell)\Delta_n})
\end{equation*}
and Equation~\eqref{eq:qstar:bound:long}, we have
\begin{align*}
	&\mathfrak{D}^b_p(k;x,u,v) \\
	& = 
    \int_{\R^d} \int_{\R^d} 
    \bs{K}_{\bs{h}}(x - y + u) 
    \bs{K}_{\bs{h}}(x - z + v)
    m_p^{p\Delta_n}(\wb \mu^b)(y)
    \mathfrak{q}_{p, k}^*(y; z)
    \, \mathrm{d}z 
    \, \mathrm{d}y
    \\
	& \lesssim
    \int_{\R^d} \int_{\R^d} 
    \bs{K}_{\bs{h}}(x - y + u) 
    \bs{K}_{\bs{h}}(x - z + v)
    \exp\!\big(c_0|z|^2\big)
    \, \mathrm{d}z 
    \, \mathrm{d}y.
\end{align*}
Using the usual change of variables and the fact that $|K|$ is integrable, we have
\begin{equation*}
    \int_{\R^d} 
    \bs{K}_{\bs{h}}(x - y + u) 
    \, \mathrm{d}y \lesssim 1,
\end{equation*}
and using the same arguments as in the proof of Lemma~\ref{lem:corr:short}, we have
\begin{equation*}
    \int_{\R^d} 
    \bs{K}_{\bs{h}}(x - z + v)
    \exp\!\big(c_0|z|^2\big)
    \, \mathrm{d}z
    \lesssim 
    \exp(3c_0 |v|^2).
\end{equation*}
The result on $\mathfrak{C}^b_p(k;x)$ follows by integrating~\eqref{eq:bound:C:Duv},  using~\eqref{eq:bound:exp:phi}.

\subsection{Proof of Lemma~\ref{lem:corr:long}}

For $\bs{h} = (h_1, \dots, h_d) \in (0,1)^d$ and $y \in \mathbb{R}^d$, we define
\begin{equation*}
	\bs{\mathcal{K}}_{\bs{h}}(y) = \int_{\R^d} \bs{K}_{\bs{h}}(x-y+u) \phi_{\tau_n / p^{1/2}}(u) \, \mathrm{d}u 
\quad
\text{ and }
\quad
\bs{\mathcal{K}}_{\bs{h}}^c(y) = \bs{\mathcal{K}}_{\bs{h}}(y) - \E_{\overline{\mu}^b}[\bs{\mathcal{K}}_{\bs{h}}(\frac{1}{p}\sum_{\ell=0}^{p-1}X_{\ell\Delta_n})]. 
\end{equation*}
Then we have by definition
\begin{align*}
	 \mathfrak{C}^b_p(k;x) & \leq \int_{\R^d}\int_{\R^{d}} \mathfrak{C}^b_p(k;x,u,v)
     \phi_{\tau_n / p^{1/2}}(u) \phi_{\tau_n / p^{1/2}}(v) \, \mathrm{d}u \, \mathrm{d}v \\
	 & \leq \E^b_{\overline{\mu}^b}\Big[\bs{\mathcal{K}}_{\bs{h}}\Big(\frac{1}{p}\sum_{\ell=0}^{p-1}X_{(kp + \ell)\Delta_n}\Big)\bs{\mathcal{K}}_{\bs{h}}\Big(\frac{1}{p}\sum_{\ell=0}^{p-1}X_{\ell\Delta_n}\Big)\Big] - \E^b_{\overline{\mu}^b}\Big[\bs{\mathcal{K}}_{\bs{h}}\Big(\frac{1}{p}\sum_{\ell=0}^{p-1}X_{\ell\Delta_n}\Big)\Big]^2, \nonumber
\\ 
&\leq \E_{\overline{\mu}^b}\Big[\bs{\mathcal{K}}^c_{\bs{h}}\Big(\frac{1}{p}\sum_{\ell=0}^{p-1}X_{(kp + \ell)\Delta_n}\Big)\bs{\mathcal{K}}^c_{\bs{h}}\Big(\frac{1}{p}\sum_{\ell=0}^{p-1}X_{\ell\Delta_n}\Big)\Big].
\end{align*}
Moreover, we write $
\{ \mathcal{A}_k, \: k \geq 0\}$ for the natural filtration induced  by the lifted Markov chain 
\begin{equation*}
    \Big(p^{-1} \sum_{\ell=0}^{p-1} Y_{kp + \ell,n}, X_{(k+1)p\Delta_n}\Big)_{k\in\mathbb{N}}.
\end{equation*}
Then using Cauchy--Schwarz inequality, we have
\begin{align*}
	\Big|\E^b_{\overline{\mu}^b}\Big[& \bs{\mathcal{K}}^c_{\bs{h}}\Big(\frac{1}{p}\sum_{\ell=0}^{p-1}X_{(kp + \ell)\Delta_n}\Big)\bs{\mathcal{K}}^c_{\bs{h}}\Big(\frac{1}{p}\sum_{\ell=0}^{p-1}X_{\ell\Delta_n}\Big)\Big]\Big| \\
	& = \Big|\E^b_{\overline{\mu}^b}\Big[\E^b_{\overline{\mu}^b}\Big[\bs{\mathcal{K}}^c_{\bs{h}}\Big(\frac{1}{p}\sum_{\ell=0}^{p-1}X_{(kp + \ell)\Delta_n}\Big)\Big| \mathcal{A}_0\Big]\bs{\mathcal{K}}^c_{\bs{h}}\Big(\frac{1}{p}\sum_{\ell=0}^{p-1}X_{\ell\Delta_n}\Big)\Big]\Big|\\
	& \leq \E_{\overline{\mu}^b}^b\Big[\E^b_{\overline{\mu}^b}\Big[\bs{\mathcal{K}}^c_{\bs{h}}\Big(\frac{1}{p}\sum_{\ell=0}^{p-1}X_{(kp + \ell)\Delta_n}\Big)\Big| \mathcal{A}_0\Big]^2\Big]^{1/2} \E_{\overline{\mu}^b}^b\Big[\bs{\mathcal{K}}^c_{\bs{h}}\Big(\frac{1}{p}\sum_{\ell=0}^{p-1}X_{\ell\Delta_n}\Big)^2\Big]^{1/2}.
\end{align*}
The second expectation is treated as in Lemma~\ref{lem:corr:inst} and we have
\begin{equation*}
    \E_{\overline{\mu}^b}^b\Big[\bs{\mathcal{K}}^c_{\bs{h}}\Big(\frac{1}{p}\sum_{\ell=0}^{p-1}X_{\ell\Delta_n}\Big)^2\Big]^{1/2} \lesssim \prod_{i=1}^d h_i^{-1}.
\end{equation*}
For the first one, we use that $X$ is a stationary Markov process to get
\begin{align*}
	 \E^b_{\overline{\mu}^b}\Big[\E^b & \Big[\bs{\mathcal{K}}^c_{\bs{h}}\Big(\frac{1}{p}\sum_{\ell=0}^{p-1}X_{(kp + \ell)\Delta_n}\Big) \Big| \mathcal{A}_0\Big]^2\Big]
  =
  \E^b_{\overline{\mu}^b}\Big[\E^b\Big[\bs{\mathcal{K}}^c_{\bs{h}}\Big(\frac{1}{p}\sum_{\ell=0}^{p-1}X_{(kp + \ell)\Delta_n}\Big)\Big| \mathcal{F}_{p\Delta_n}\Big]^2\Big]
  \\
  &= \int_{\R^d}  \E^b\Big[\bs{\mathcal{K}}^c_{\bs{h}}\Big(\frac{1}{p}\sum_{\ell=0}^{p-1}X_{(kp + \ell)\Delta_n}\Big) | X_{p\Delta_n} = x \Big]^2 \overline{\mu}^b(\mathrm{d}x) \\
	 & = \int_{\R^d}  \E^b\Big[ \E^b\Big[ \bs{\mathcal{K}}^c_{\bs{h}}\Big(\frac{1}{p}\sum_{\ell=0}^{p-1}X_{(kp + \ell)\Delta_n}\Big)\Big| X_{kp\Delta_n}\Big] | X_{p\Delta_n} = x \Big]^2 \overline{\mu}^b(\mathrm{d}x)
\end{align*}
Define $\bs{\varphi} : y \mapsto  \E^b\Big[ \bs{\mathcal{K}}^c_{\bs{h}}\Big(\frac{1}{p}\sum_{\ell=0}^{p-1}X_{(kp + \ell)\Delta_n}\Big)\Big| X_{kp\Delta_n} =y \Big]$. Then for any $x \in \R^d$,
\begin{equation*}
	\E^b[\bs{\varphi}(X_{kp\Delta_n}) | X_{p\Delta_n} = x]  = \E^b_x[\bs{\varphi}(X_{(k-1)p\Delta_n})]  = P^b_{(k-1)p\Delta_n}\bs{\varphi}(x),
\end{equation*}
where $(P^b_t)$ is defined in Equation~\eqref{eq:defSG}. 
Moreover, from the definition of $\bs{\mathcal{K}}^c_{\bs{h}}$, we get $\E^b_{\overline{\mu}^b}[\bs{\varphi}(X_0)] = 0$. Then, using Equation~\eqref{eq:ergodicity}, we have
\begin{align*}
	\E^b_{\overline{\mu}^b}\Big[\E^b\Big[\bs{\mathcal{K}}^c_{\bs{h}}\Big(\frac{1}{p}\sum_{\ell=0}^{p-1}X_{(kp + \ell)\Delta_n}\Big)\Big| \mathcal{F}_{p\Delta_n}\Big]^2\Big] & = \mathrm{Var}_{\overline{\mu}^b}^b[ P^b_{(k-1)p\Delta_n}\bs{\varphi}(X_0)] \\
	& \leq e^{-2C_{PI}^{-1}(k-1)p\Delta_n} \E^b_{\overline{\mu}^b}[\bs{\varphi}(X_0)^2]. 
\end{align*}
Moreover, for all $x\in \R^d$, $|\bs{\varphi}(x)| \leq 2\|\bs{\mathcal{K}}^c_{\bs{h}}\|_{\infty} \leq 2\|\bs{K}_{\bs{h}}\|_{\infty}$. Then, we have $\| \bs{\varphi}\|_{\infty} \lesssim \prod_{i=1}^d h_i^{-1}$. Combining all previous inequalities, we get
\begin{equation*}
    \mathfrak{C}^b_p(k;x) \lesssim e^{-C_{PI}^{-1}(k-1)p\Delta_n}\Big( \prod_{i=1}^d h_i^{-1}\Big)^2.
\end{equation*}
We conclude using that $p\Delta_n$ is bounded, which allows us to replace $(k-1)p\Delta_n$ by $k\Delta_n$.

\section{Proof of Theorem~\ref{th:bernstein:loc}}
\label{sec:proof:bernstein}

We apply the regeneration Bernstein inequality of Lema\'nczyk~\cite[Theorem~1.1, made explicit in Theorem~4.3 and the following remark]{lemanczyk2021general}, which is recalled here for completeness.

\begin{theorem}[Theorem 1.1 in~\cite{lemanczyk2021general}]
Let $\Upsilon$ be a geometrically ergodic Markov chain with state space $\mathcal{X}$, and let $\pi$ be its unique stationary probability measure. Moreover, let $f: \mathcal{X} \rightarrow \mathbb{R}$ be a bounded measurable function such that $\mathbb{E}_\pi f=0$. Furthermore, let $x \in \mathcal{X}$. Then, we can find constants $K, \tau>0$ depending only on $x$ and the transition probability $P(\cdot, \cdot)$ such that for all $t>0$,
\[
\mathbb{P}_x\left(\frac{1}{N}\sum_{i=0}^{N-1} f\left(\Upsilon_i\right)>t\right) \leq K \exp \left(-\frac{N^2t^2}{32 N \sigma_{M r v}^2+\tau t\|f\|_{\infty} N\log N}\right),
\]
where
\[
\sigma_{M r v}^2=\operatorname{Var}_\pi\left(f\left(\Upsilon_0\right)\right)+2 \sum_{i=1}^{\infty} \operatorname{Cov}_\pi\left(f\left(\Upsilon_0\right), f\left(\Upsilon_i\right)\right)
\]
denotes the asymptotic variance of the process $\left(f\left(\Upsilon_i\right)\right)_i$.
\end{theorem}

We apply this result to the noisy lifted chain
\begin{equation*}
	\Big(p^{-1} \sum_{\ell=0}^{p-1} Y_{kp + \ell,n}, X_{(k+1)p\Delta_n}\Big)_{k\in\mathbb{N}}.
\end{equation*}
This process is Markov because, conditionally on the endpoint $X_{(k+1)p\Delta_n}$ and on the independent future noises, the next block is independent of the preceding blocks. Equivalently, it is the augmented chain of Lemma~\ref{lem:Z-Markov} with step $h_p=p\Delta_n$, post-composed in the first coordinate with the Gaussian noise kernel of variance $\tau_n^2/p$. Proposition~\ref{prop:Z-geom-ergodic} gives geometric ergodicity of this lifted chain, uniformly in physical time; however, its discrete geometric rate is $\rho_{h_p}=e^{-ch_p}$ and therefore degenerates as $h_p\downarrow0$. Thus the constants in Lema\'nczyk's inequality cannot be used as black-box constants depending only on the transition kernel.

The explicit form needed below is the following. If a geometrically ergodic Markov chain $\boldsymbol\Upsilon$ with invariant law $\pi$ satisfies the minorization condition (MC) of~\cite[Section~3]{lemanczyk2021general} with data $(C,m,\nu,\delta)$ and if, for the corresponding split chain,
\[
D=\max\big(
\|\tau_1-\tau_0\|_{\psi_1},
\|\tau_0\|_{\psi_1,\mathbb P_{\mu^*}},
\|\tau_0\|_{\psi_1,\mathbb P_{\pi^*}}
\big)<\infty,
\]
then the proof of \cite[Theorem~4.3]{lemanczyk2021general}, with the initial distribution $x^*$ replaced by $\mu^*$ in the first-block Orlicz norm from \cite[Theorem~4.1]{lemanczyk2021general}, gives for every bounded centered $f$ and every $N\ge2$, $t>0$,
\[
\mathbb P_{\mu}\left(\left|\sum_{i=0}^{N-1}f(\Upsilon_i)\right|>t\right)
\le
K\exp\left(
-\frac{t^2}{32N\sigma_{Mrv}^2+433\,t\,\delta\pi(C)\|f\|_\infty D^2\log N}
\right),
\]
where
\[
K=e^{10}+2(\delta\pi(C))^{-1},
\qquad
\sigma_{Mrv}^2=\operatorname{Var}_{\pi}(f(\Upsilon_0))
     +2\sum_{i=1}^{\infty}\operatorname{Cov}_{\pi}(f(\Upsilon_0),f(\Upsilon_i)).
\]
Thus, in our application, the constants are controlled once we have a minorization whose $\delta\pi(C)$ is bounded below uniformly and whose regeneration-time $\psi_1$ norms are bounded with their exact $h_p$-dependence. This is the purpose of the following lemma.

\begin{lemma}[Uniform regeneration for the lifted physical skeleton]
\label{lem:uniform-lifted-regeneration}
Fix $h_0>0$. There exist constants
\[
T>h_0,\quad R_C<\infty,\quad \delta_0\in(0,1),\quad \pi_0\in(0,1),
\quad \theta_0>0,\quad A_0<\infty,
\]
depending only on $\bs{\mathfrak b}$ and $h_0$, with the following property.
For every $b\in\Sigma(\bs{\mathfrak b})$ and every $h\in(0,h_0]$, set
\[
m_h=\lceil T/h\rceil+1,
\qquad
C=\R^d\times B_{R_C}.
\]
Then
\[
T\le (m_h-1)h\le T+h_0,
\qquad
T\le m_hh\le T+2h_0,
\]
and there exists a probability measure $\nu_h^b$ on $\R^{2d}$, with $\nu_h^b(C)=1$, such that
\[
(\mathfrak P_h^b)^{m_h}(z,A)\ge \delta_0\nu_h^b(A),
\qquad z\in C,\quad A\in\mathcal B(\R^{2d}).
\]
The invariant measure of the lifted chain satisfies
\[
\inf_{b\in\Sigma(\bs{\mathfrak b})}\inf_{0<h\le h_0}\overline\pi^{h,b}(C)\ge \pi_0.
\]
Let $(\tau_i^{h,b})_{i\ge0}$ be the Nummelin regeneration times associated with this minorization, in the sense of the minorization condition (MC) in \cite[Section~3]{lemanczyk2021general}; these times are measured in the original lifted-chain steps, not in physical time. Then
\[
\sup_{b\in\Sigma(\bs{\mathfrak b})}\sup_{0<h\le h_0}
\mathbb E\exp\{\theta_0h(\tau_1^{h,b}-\tau_0^{h,b})\}
\le A_0,
\]
and
\[
\sup_{b\in\Sigma(\bs{\mathfrak b})}\sup_{0<h\le h_0}
\mathbb E_{\overline\pi^{h,b,*}}\exp\{\theta_0h\tau_0^{h,b}\}
\le A_0.
\]
Moreover, for every deterministic initial point $z=(u,x)$,
\[
\sup_{b\in\Sigma(\bs{\mathfrak b})}\sup_{0<h\le h_0}
\mathbb E_{z^*}\exp\{\theta_0h\tau_0^{h,b}\}
\le A_0V_X(x).
\]
Consequently, with
\[
D_h(z)=\max\Big(
\|\tau_1^{h,b}-\tau_0^{h,b}\|_{\psi_1},
\|\tau_0^{h,b}\|_{\psi_1,\mathbb P_{z^*}},
\|\tau_0^{h,b}\|_{\psi_1,\mathbb P_{\overline\pi^{h,b,*}}}
\Big),
\]
one has, uniformly in $b$ and $h$,
\[
D_h(z)\le C_Dh^{-1}\big(1+\log V_X(x)\big),
\qquad
D_h^{\mathrm{stat}}\le C_Dh^{-1}
\]
for the stationary-start value $D_h^{\mathrm{stat}}$ in which the deterministic-start norm is omitted or replaced by the stationary split-chain norm.
The same statements hold for the noisy lifted chain.
\end{lemma}

\begin{proof}
We write
\[
V_Z(u,x)=V_X(x)=\exp\big(a(1+|x|^2)^{1/2}\big),
\]
where $a$ is chosen as in the proof of Proposition~\ref{prop:Z-geom-ergodic}. That proof gives the uniform Foster--Lyapunov estimate
\[
\mathcal L^bV_X\le -\lambda V_X+B\1_{B_{R_0}},
\qquad b\in\Sigma(\bs{\mathfrak b}),
\]
see~\eqref{eq:diff:X} and the argument preceding~\eqref{eq:defSG}. Dynkin's formula yields
\begin{equation}\label{eq:uniform-regeneration:continuous-drift}
P_t^bV_X(x)\le e^{-\lambda t}V_X(x)+B/\lambda,
\qquad t\ge0.
\end{equation}

We next prove the minorization. Let $R_C>R_0$ be fixed for the moment and choose a smaller ball $B_r\subset B_{R_C}$. By the two-sided heat-kernel estimate of \cite[Theorem~1.2, estimate~(1.4)]{menozzi2020density}, applied on the compact time interval $[T,T+h_0]$ and on the compact spatial set $B_{R_C}\times B_r$, there is $\varepsilon_0>0$, depending only on $\bs{\mathfrak b},h_0,T,R_C,r$, such that
\[
\inf_{b\in\Sigma(\bs{\mathfrak b})}
\inf_{s\in[T,T+h_0]}
\inf_{x\in B_{R_C},\,y\in B_r}
p_s^b(x;y)\ge \varepsilon_0.
\]
Equivalently, if $\lambda_r$ is normalized Lebesgue measure on $B_r$, then
\begin{equation}\label{eq:uniform-regeneration:diffusion-minorization}
P_s^b(x,dy)\ge \varepsilon_0|B_r|\lambda_r(dy),
\qquad x\in B_{R_C},\quad s\in[T,T+h_0].
\end{equation}
For the lifted chain, the identity proved in Proposition~\ref{prop:Z-geom-ergodic} gives, for $k\ge1$,
\begin{equation}\label{eq:uniform-regeneration:lifted-identity}
(\mathfrak P_h^b)^k((u,x),\cdot)
=\big(P_{(k-1)h}^b(x,\cdot)\big)Q_h^b,
\end{equation}
where $Q_h^b(y,A)=\mathbb P_y^b((m^h_{0,h}(X),X_h)\in A)$. Since $(m_h-1)h\in[T,T+h_0]$,~\eqref{eq:uniform-regeneration:diffusion-minorization} and~\eqref{eq:uniform-regeneration:lifted-identity} imply, for $z=(u,x)\in C$,
\begin{equation}\label{eq:uniform-regeneration:raw-lifted-minorization}
(\mathfrak P_h^b)^{m_h}(z,A)
\ge
\varepsilon_0|B_r|\int_{B_r}Q_h^b(y,A)\lambda_r(dy).
\end{equation}
It remains to restrict the minorizing measure to $C$ with a uniform positive mass. By Assumption~\ref{assumption:boundedness},
\[
|b(y)|\le b_0+b_1|y|.
\]
If $y\in B_r$ and $\sup_{0\le t\le h_0}|W_t|\le1$, then Gronwall's lemma applied to
\[
X_t=y+\int_0^tb(X_s)ds+W_t
\]
gives
\[
\sup_{0\le t\le h_0}|X_t|
\le (r+1+b_0h_0)e^{b_1h_0}.
\]
Increasing $R_C$ if necessary, this upper bound is at most $R_C$. Hence
\[
q_0:=\inf_{b\in\Sigma(\bs{\mathfrak b})}\inf_{0<h\le h_0}\inf_{y\in B_r}
Q_h^b(y,C)
\ge
\mathbb P\Big(\sup_{0\le t\le h_0}|W_t|\le1\Big)>0.
\]
Define
\[
\nu_h^b(A)
=
\frac{\int_{B_r}Q_h^b(y,A\cap C)\lambda_r(dy)}
{\int_{B_r}Q_h^b(y,C)\lambda_r(dy)}.
\]
Then $\nu_h^b$ is a probability measure, $\nu_h^b(C)=1$, and~\eqref{eq:uniform-regeneration:raw-lifted-minorization} gives the minorization with
\[
\delta_0=\varepsilon_0|B_r|q_0.
\]
If independent Gaussian observation noise is added to the first coordinate, the lifted kernel is post-composed with a noise kernel which leaves the endpoint coordinate unchanged; the preceding construction applies with $Q_h^b$ replaced by this post-composed kernel and the same set $C$, $\delta_0$ and $q_0$.

We now prove the invariant-mass bound. Under $\overline\pi^{h,b}$, the endpoint marginal is $\overline\mu^b$ by the definition of $\overline\pi^{h,b}=\overline\mu^bQ_h^b$ in the proof of Proposition~\ref{prop:Z-geom-ergodic}. Integrating the Lyapunov inequality against $\overline\mu^b$ gives
\[
\sup_{b\in\Sigma(\bs{\mathfrak b})}\int V_X\,d\overline\mu^b\le B/\lambda.
\]
Thus
\[
\overline\pi^{h,b}(C)
=\overline\mu^b(B_{R_C})
\ge
1-\frac{B/\lambda}{\inf_{|x|>R_C}V_X(x)}.
\]
Increasing $R_C$ once more gives the asserted lower bound with, for instance, $\pi_0=1/2$.

It remains to control the Nummelin regeneration times. Let
\[
S_h^b=(\mathfrak P_h^b)^{m_h},
\qquad
W_j=Z_{jm_h}^h.
\]
The chain $W$ is the $m_h$-skeleton of the lifted chain, and the preceding minorization is a one-step Nummelin minorization for $W$ on $C$. Moreover,
\[
S_h^bV_Z(u,x)
=P_{m_hh}^bV_X(x)
\le e^{-\lambda T}V_Z(u,x)+B/\lambda
\]
by~\eqref{eq:uniform-regeneration:continuous-drift}. Choose $\lambda_\star\in(e^{-\lambda T},1)$ and enlarge $R_C$ so that outside $C$,
\[
e^{-\lambda T}V_Z+B/\lambda\le \lambda_\star V_Z.
\]
On $C$, set
\[
K_\star=\sup_{b,h}\sup_{z\in C}S_h^bV_Z(z)<\infty.
\]
Let $\sigma_C=\inf\{j\ge0:W_j\in C\}$ and $\sigma_C^+=\inf\{j\ge1:W_j\in C\}$. For any $r_0\in(1,\lambda_\star^{-1})$, the process
\[
r_0^{j\wedge\sigma_C}V_Z(W_{j\wedge\sigma_C})
\]
is a supermartingale until $\sigma_C$ when started outside $C$. Since $V_Z\ge1$, optional stopping and monotone convergence give
\begin{equation}\label{eq:uniform-regeneration:first-hit}
\mathbb E_z r_0^{\sigma_C}\le V_Z(z),
\qquad z\notin C.
\end{equation}
If $z\in C$, then by the Markov property and~\eqref{eq:uniform-regeneration:first-hit},
\begin{equation}\label{eq:uniform-regeneration:return}
\mathbb E_z r_0^{\sigma_C^+}
\le r_0\mathbb E_zV_Z(W_1)
\le r_0K_\star.
\end{equation}
This is the usual drift-to-a-small-set mechanism used in the quantitative renewal theory of Baxendale \cite[Assumptions~A1--A3 and Theorem~1.1]{baxendale2005renewal}; here the bound is written out because only the exponential moment, not Baxendale's convergence-rate constant, is needed.

Let $\overline\tau_i^{h,b}=\tau_i^{h,b}/m_h$ denote the corresponding regeneration times in $W$-skeleton units. The split-chain construction is the Athreya--Ney--Nummelin construction; see Athreya--Ney~\cite{athreya1978new}, Nummelin \cite[Chapter~4]{nummelin1984general}, and Meyn--Tweedie \cite[Chapters~5 and~17]{meyn2012markov}. Since $\nu_h^b(C)=1$, after each failed visit to $C$ the next trial begins after a return time whose $r_0$-moment is bounded by~\eqref{eq:uniform-regeneration:return}. Choose $s\in(1,r_0)$ so close to $1$ that
\[
(1-\delta_0)\sup_{z\in C}\mathbb E_zs^{\sigma_C^+}<1,
\]
which is possible by~\eqref{eq:uniform-regeneration:return} and dominated convergence. Summing the resulting geometric series over the number of failed split-chain trials gives
\[
\sup_{b,h}\mathbb E s^{\overline\tau_1^{h,b}-\overline\tau_0^{h,b}}\le A_1<\infty.
\]
The same argument combined with~\eqref{eq:uniform-regeneration:first-hit} yields
\[
\sup_{b,h}\mathbb E_{z^*}s^{\overline\tau_0^{h,b}}\le A_1V_Z(z),
\qquad
\sup_{b,h}\mathbb E_{\overline\pi^{h,b,*}}s^{\overline\tau_0^{h,b}}\le A_1\sup_{b}\overline\pi^{h,b}(V_Z)<\infty.
\]
Finally,
\[
h(\tau_i^{h,b}-\tau_{i-1}^{h,b})
=hm_h(\overline\tau_i^{h,b}-\overline\tau_{i-1}^{h,b})
\le (T+2h_0)(\overline\tau_i^{h,b}-\overline\tau_{i-1}^{h,b}).
\]
Taking $\theta_0=(\log s)/(T+2h_0)$ proves the exponential moment bounds in physical time.

If a non-negative random variable $Y$ satisfies $\mathbb E\exp(\theta Y)\le A$, then
\[
\|Y\|_{\psi_1}
\le
\theta^{-1}\max\{1,\log A/\log2\}.
\]
Applying this with $Y=h\tau_0^{h,b}$ and $Y=h(\tau_1^{h,b}-\tau_0^{h,b})$ gives the asserted Orlicz bounds. For deterministic $z=(u,x)$, the factor $A_0V_X(x)$ contributes only $\log V_X(x)$ to the Orlicz norm.
\end{proof}

We now apply the preceding explicit Bernstein inequality with $h=h_p$ and $N=n_p$. Lemma~\ref{lem:uniform-lifted-regeneration} gives, for the noisy lifted chain, a minorization after $m=m_{h_p}$ lifted-chain steps with $\delta\ge\delta_0$ and $\pi(C)\ge\pi_0$. The regeneration times are measured in lifted-chain steps; since the lemma controls $h_p\tau_0^{h_p,b}$ and $h_p(\tau_1^{h_p,b}-\tau_0^{h_p,b})$, the stationary-start Orlicz parameter in Lema\'nczyk's theorem satisfies
\[
D\le D_{h_p}^{\mathrm{stat}}\le C_Dh_p^{-1},
\]
uniformly in $b$, $p$, $\Delta_n$ and the observation-noise variance.

Let
\[
Z_k=\left(p^{-1}\sum_{\ell=0}^{p-1}Y_{kp+\ell,n},\,X_{(k+1)p\Delta_n}\right),
\qquad
g_{x,\bs h}(u,z)=\bs K_{\bs h}(x-u),
\]
and set
\[
f_{x,\bs h}(u,z)
=g_{x,\bs h}(u,z)
-\mathbb E_{\pi}\big[g_{x,\bs h}(Z_0)\big],
\]
where $\pi$ denotes the invariant law of the noisy lifted chain. Under $\mathbb P_{\overline\mu^b}^b$ with independent stationary observation noise, $(Z_k)_{k\ge0}$ is stationary with invariant law $\pi$, and therefore
\[
\widehat\nu_{n,\bs h,p}(x)
-\mathbb E_{\overline\mu^b}^b\widehat\nu_{n,\bs h,p}(x)
=\frac1{n_p}\sum_{k=0}^{n_p-1}f_{x,\bs h}(Z_k).
\]
Moreover, $\mathbb E_\pi f_{x,\bs h}=0$ and
\[
\|f_{x,\bs h}\|_\infty
\le 2\|\bs K_{\bs h}\|_\infty.
\]
Applying Lema\'nczyk's two-sided inequality with $t=n_p\varepsilon$ gives
\[
\mathbb P_{\overline\mu^b}^b\left(
\left|\widehat\nu_{n,\bs h,p}(x)
-\mathbb E_{\overline\mu^b}^b\widehat\nu_{n,\bs h,p}(x)\right|>\varepsilon
\right)
\le
K\exp\left(
-\frac{n_p^2\varepsilon^2}
{32n_p\sigma_{Mrv}^2
+866\,\delta\pi(C)D^2\varepsilon\|\bs K_{\bs h}\|_\infty n_p\log n_p}
\right).
\]
Since $\delta\pi(C)\le1$, $D\le C_Dh_p^{-1}$ and $\delta\pi(C)\ge\delta_0\pi_0$, the last display is bounded by~\eqref{eq:bernstein:nu:explicit} after choosing, for example,
\[
K_B=e^{10}+2(\delta_0\pi_0)^{-1},
\qquad
C_B=866C_D^2.
\]
The asymptotic variance $\sigma_{Mrv}^2$ in Lema\'nczyk's theorem is exactly the quantity denoted in Theorem~\ref{th:bernstein:loc} by $\mathbb V_\nu(\alpha,n,\bs h,p;x)$.
Indeed, subtracting the stationary mean from $g_{x,\bs h}$ does not change covariances, and the first coordinate of $Z_k$ is the preaveraged noisy observation appearing in the definition of $\mathbb V_\nu$.

The additional claim
\[
\mathbb V_\nu(\alpha,n,\bs h,p;x)\le C_Vv^2(\alpha,n,\bs h,p)
\]
is obtained by repeating the proof of Proposition~\ref{prop:variance} with the infinite covariance sum $\mathbb V_\nu$ in place of the finite sum $\operatorname{Var}_{\overline\mu^b}^b(n_p^{-1}\sum_{k=0}^{n_p-1}\bs K_{\bs h}(x-p^{-1}\sum_{\ell=0}^{p-1}Y_{kp+\ell,n}))$. The tail of the infinite sum is controlled by the $L^2(\overline\mu^b)$ contraction~\eqref{eq:ergodicity}, as in the proof of Lemma~\ref{lem:corr:long}.

\section{Proof of the results of Section~\ref{sec:upper_bounds}}

\subsection{Preliminary results}

We say that $A(\theta) \lesssim B(\theta)$ implies $C(\theta) \lesssim D(\theta)$ if for every $c_1 > 0$, there exists $c_2 > 0$ such that $A(\theta) \leq c_1 B(\theta)$ implies $C(\theta) \leq c_2 D(\theta)$. We similarly define the equivalence between $A(\theta) \lesssim B(\theta)$ and $C(\theta) \lesssim D(\theta)$. We can now state the main lemma of this section.

\begin{lemma}
    \label{lem:simplificationp}
For $p \geq 1$, the condition $p\Delta_n \lesssim w_n^{HF}$ is equivalent to 
$p\Delta_n \lesssim \widetilde{w}_{n,p}^{HF}$ with
\begin{equation*}
\widetilde{w}_{n,p}^{HF}
=
    \begin{cases} 
        \big( \frac{p}{n}\big)^{\frac{1}{2\alpha_1 + 1}}\log\big(\frac{n}{p}\big)^{\frac{2\alpha_1}{2\alpha_1 + 1}} & \, \text{ if } d=1,
        \\
        \big( \frac{p}{n}\big)^{\frac{1}{\overline{\alpha} + 1} } \log\big(\frac{n}{p}\big)^{\frac{\overline{\alpha}}{\overline{\alpha} + 1}} & \, \text{ if } d=2,
        \\
        \big(\frac{p}{n}\big)^{\frac{\overline{\alpha}}{2\overline{\alpha} + d} ( \frac{1}{\alpha_1} + \frac{1}{\alpha_2} )}
        \log\big(\frac{n}{p}\big)
        & \, \text{ if } d \geq 3 \text{ and } (\bs{\alpha}, k_0)\in D_1,
        \\
        \big(\frac{p}{n}\big)^{\frac{2\overline{\alpha}}{2\overline{\alpha} + d} \frac{1}{\alpha_1}} & \, \text{ if } d \geq 3 \text{ and } (\bs{\alpha}, k_0)\in D_2,
        \\
        \big(\frac{p}{n}\big)^{\frac{\overline{\alpha}}{(2\overline{\alpha} + d)} \big( \frac{1}{\alpha_1} + \frac{1}{\alpha_2} \big)} & \, \text{ if } d\geq 3 \text{ and } (\bs{\alpha}, k_0)\in D_3.
    \end{cases}
\end{equation*}
\end{lemma}

\begin{proof}
First note that since $T_n = \frac{np\Delta_n}{p} $, the condition
\begin{equation*}
p\Delta_n 
\lesssim 
T_n^{-u} \log(T_n)^v
\end{equation*}
for some $u,v \geq 0$ is equivalent to
\begin{equation*}
p\Delta_n
\lesssim 
\Big(\frac{p}{n}\Big)^{\frac{u}{1+u}} \log\Big(\frac{n}{p}\Big)^{\frac{v}{1+u}}.
\end{equation*}

If $u = \frac{\overline{\alpha}_3}{2\overline{\alpha}_3 + d - 2} \Big( \frac{1}{\alpha_1} + \frac{1}{\alpha_2} \Big)$, we have
\begin{equation*}
u = \frac{1}{2 + \frac{d - 2}{\overline{\alpha}_3}}
\frac{\alpha_1 + \alpha_2}{\alpha_1\alpha_2}
= \frac{1}{2 + \frac{d}{\overline{\alpha}} - \frac{1}{\alpha_1} - \frac{1}{\alpha_2}}
\frac{\alpha_1 + \alpha_2}{\alpha_1\alpha_2}
= \frac{\overline{\alpha}(\alpha_1 + \alpha_2)}{\alpha_1\alpha_2(2\overline{\alpha} + d) - \overline{\alpha}(\alpha_1+\alpha_2)},
\end{equation*}
and therefore
\begin{equation*}
\frac{u}{1+u} = \frac{\overline{\alpha}(\alpha_1 + \alpha_2)}{\alpha_1\alpha_2(2\overline{\alpha} + d)} = 
\frac{\overline{\alpha}}{2\overline{\alpha} + d} \Big( \frac{1}{\alpha_1} + \frac{1}{\alpha_2} \Big),
\end{equation*}
which proves the case $d \geq 3$ and $k_0 = 1$ with  $\alpha_2 < \alpha_3$ or $k_0 = 2$. The other cases are done similarly.
\end{proof}

\subsection{Proof of Proposition~\ref{prop:smallnoise:highsampling}}
Suppose that $p^{\ast}\Delta_n \leq w_n^{HF}$ and that $\tau_n^{2\alpha_1} \leq \Delta_n^\beta$. Recall that in that case, we take
\begin{equation*}
    p^* = 1 \;\;\text{ and }\;\; \bs{h}^* = \bs{h}^{*,HF}.
\end{equation*}
From Proposition~\ref{propo:small:bias:mu}, we see by plugging these values that
\begin{equation*}
    \mathrm{B}^b_{n,\bs{h}^*,p^*}(x)^2 \lesssim \tau_n^{2\alpha_1} + v_n^{HF}.
\end{equation*}
We now turn to the variance. Since $|\log(h_i^*)|$ is of the order of $\log(T_n)$, we have 
\begin{equation*}
    v_n^{HF} {\lesssim}
    \: T_n^{-1}\begin{cases}
        \sum_{i=1}^d|\log(h_i^*)|
	& 
	\, \text{ if } d=1 \text{ or } d=2,
        \\
         \sum_{i=1}^d |\log(h_i^*)| \prod_{i=3}^d (h_i^*)^{-1}
	& 
	\, \text{ if } d \geq 3 \text{ and } (\bs{\alpha}, k_0) \in D_1
        \\
         (\prod_{i=1}^{k_0} (h_i^*)^{(2-k_0)/k_0}  \prod_{i=k_0+1}^d (h_i^*)^{-1} 
	& 
	\, \text{ if } d \geq 3 \text{ and } (\bs{\alpha}, k_0) \in D_2
        \\
        (h_2^* h_3^*)^{-1/2} \prod_{i=4}^d (h_i^*)^{-1}
    	& 
	\, \text{ if } d \geq 3 \text{ and } (\bs{\alpha}, k_0) \in D_3.
    \end{cases}
\end{equation*}
Therefore, by factoring by $v_n^{HF}$ the variance bound from Proposition~\ref{prop:variance}, we have
\begin{equation*}
	\mathrm{V}^b_{n,\bs{h}^*,p^*}(x) \lesssim v_n^{HF}(1 + r_n^{HF}), 
\end{equation*}
with 
\begin{equation*}
r_n^{HF}
=
    \begin{cases}
        0
	& 
	\, \text{ if } d=1 \text{ or } d=2,
        \\
	\Delta_n (h_1^*h_2^*)^{-1} \log(T_n)^{-1}
	& 
	\, \text{ if } d \geq 3 \text{ and } (\bs{\alpha}, k_0)\in D_1,
        \\
	    \Delta_n  \prod_{i=1}^{k_0} (h_i^*)^{-2/k_0}  
	   +
	   \log(T_n)  \prod_{i=1}^{k_0} (h_i^*)^{-2/k_0}  \prod_{i=1}^d h_i^*
	& 
	\, \text{ if } d\geq 3 \text{ and } (\bs{\alpha}, k_0)\in D_2,
        \\
 	      \log(T_n) (h_2^* h_3^*)^{1/2} \prod_{i=4}^d h_i^*
    	+
    	\Delta_n h_1^{-1} (h_2^* h_3^*)^{-1/2}
    	& 
	\, \text{ if } d\geq 3 \text{ and } (\bs{\alpha}, k_0)\in D_3,
    \end{cases}
\end{equation*}
\normalsize
and it remains to prove that in each case, the remainder can be ignored so that $\mathrm{V}^b_{n,\bs{h},p}(x)
\lesssim
v_n^{HF}$. First note that the condition $\Delta_n \lesssim w_n^{HF}$ and the definition of $\bs{h}^*$ ensure that
\begin{equation*}
    \Delta_n \lesssim 
        \begin{cases}
        h_1^*\log(T_n)
	& 
	\, \text{ if } d=1,
        \\
        h_1^*h_2^*\log(T_n)
	& 
	\, \text{ if } d=2,
        \\
    h_1^*h_2^* \log(T_n)
	& 
	\, \text{ if } d \geq 3 \text{ and } (\bs{\alpha}, k_0) \in D_1,
        \\
        \prod_{i=1}^{k_0} (h_i^*)^{2/k_0}  
	& 
	\, \text{ if } d \geq 3 \text{ and } (\bs{\alpha}, k_0) \in D_2,
        \\
    	h_1^* (h_2^* h_3^*)^{1/2}
    	& 
	\, \text{ if } d \geq 3 \text{ and } (\bs{\alpha}, k_0) \in D_3,
    \end{cases}
\end{equation*}
and therefore
\begin{equation*}
r_n^{HF}
\lesssim
    \begin{cases}
        0
	& 
	\,\text{ if } d=1 \text{ or } d=2 \text{ or } d \geq 3 \text{ and } (\bs{\alpha}, k_0) \in D_1,
        \\
	   \log(T_n)  \prod_{i=1}^{k_0} (h_i^*)^{-2/k_0}  \prod_{i=1}^d h_i^*
	& 
	\, \text{ if } d\geq 3 \text{ and } (\bs{\alpha}, k_0) \in D_2,
        \\
 	      \log(T_n) (h_2^* h_3^*)^{1/2} \prod_{i=4}^d h_i^*
    	& 
	\, \text{ if } d\geq 3 \text{ and } (\bs{\alpha}, k_0) \in D_3.
    \end{cases}
\end{equation*}
Then we easily check that if $d\geq 3$ and $k_0 \geq 3$, we have
\begin{equation*}
    \log(T_n)  \prod_{i=1}^{k_0} (h_i^*)^{-2/k_0}  \prod_{i=1}^d h_i^*
    =
    \log(T_n) T_n^{-\frac{\bar\alpha_2}{\bar\alpha_3 + d -2}\sum_{j=k_0+1}^d 1/\alpha_j} \to 0.
\end{equation*}
Analogously, if $d\geq 3$ and $k_0 = 1$ with $\alpha_2 = \alpha_3$, we have
\begin{equation*}
    \log(T_n) (h_2^* h_3^*)^{-1/2} \prod_{i=4}^d h_i^*
    =
    \log(T_n)^{c+1} T_n^{-c} \to 0,
\end{equation*}
with 
\begin{equation*}
	c = \frac{\bar\alpha_3}{2\bar\alpha_3 + d -2}\left( \frac{1}{2\alpha_2} + \frac{1}{2\alpha_3} + \sum_{i=4}^d \frac{1}{\alpha_i}\right) > 0, 
\end{equation*}
so that 
\begin{equation*}
\mathrm{V}^b_{n,\bs{h}^*,p^*}(x)
\lesssim
v_n^{HF} \vee \tau_n^{2\alpha_1},
\end{equation*}
which concludes the proof of Proposition~\ref{prop:smallnoise:highsampling}.

\subsection{Proof of Proposition~\ref{prop:largenoise:highsampling}}

Suppose that $p^{\ast}\Delta_n \leq w_n^{HF}$ and that $\tau_n^{2\alpha_1} \geq \Delta_n^\beta$ and recall that in that case, we take
\begin{equation*}
        p^* = \Big\lceil \big(\tau_{n}^{2\alpha_1} \Delta_n^{-\beta} \big)^{1/(\beta+\alpha_1)} \Big\rceil \vee 1 \;\;\text{ and }\;\; \bs{h}^* = \bs{h}^{*,HF}.
\end{equation*}
Thus $p^*\Delta_n \asymp (\tau_n^2\Delta_n)^{\alpha_1/(\beta+\alpha_1)}$ and $(p^*\Delta_n)^\beta \asymp \tau_n^{2\alpha_1}/(p^*)^{\alpha_1} \asymp (\tau_n^2\Delta_n)^{\alpha_1\beta/(\beta+\alpha_1)}$.
From Proposition~\ref{propo:small:bias:mu}, we see by plugging these values that
\begin{equation*}
    \mathrm{B}^b_{n,\bs{h}^*,p^*}(x)^2 \lesssim \frac{\tau_n^{2\alpha_1}}{(p^*)^{\alpha_1}} + (p^*\Delta_n)^{\beta} + v_n^{HF} \lesssim \big(\tau_n^2 \Delta_n\big)^{\frac{\alpha_1 \beta}{\beta+\alpha_1}} + v_n^{HF}.
\end{equation*}
For the variance, we proceed as for Proposition~\ref{prop:largenoise:highsampling}: using that $p^*\Delta_n \lesssim w_n^{HF}$ and $\bs{h}^* = \bs{h}^{*, HF}$, we have
from Proposition~\ref{prop:variance} that
\begin{equation*}
    \mathrm{V}_{n,\bs{h}^*,p^*}^b(x) \lesssim v_n^{HF}.
\end{equation*}
Therefore we can conclude.

\subsection{Proof of Proposition~\ref{prop:smallnoise:lowsampling}}

Suppose that $\Delta_n \geq w_n^{HF}$ and that $\tau_n^{2\alpha_1} \leq \Delta_n^\beta$ and recall that in that case, we take
\begin{equation*}
        p^* = 1 \;\;\text{ and }\;\; \bs{h}^* = \bs{h}^{*,1} = \bs{h}^{*,LF}.
\end{equation*}
From Proposition~\ref{propo:small:bias:mu}, we see by plugging these values that
\begin{equation*}
    \mathrm{B}^b_{n,\bs{h},p}(x)^2 \lesssim \tau_n^{2\alpha_1} + n^{-2\overline{\alpha}/(2\overline{\alpha}+d)}.
\end{equation*}
We now study the variance. First, note that
\begin{equation*}
    T_n^{-1} \Delta_n \prod_{i=1}^d (h_i^*)^{-1} = n^{\frac{-2\overline{\alpha}}{2\overline{\alpha}+d}} = v_n^{LF}.
\end{equation*}
Therefore, by factoring by $v_n^{LF}$ the variance bound from Proposition~\ref{prop:variance} and using that $|\log(h_i^*)|$ is of the order of $\log(n)$, we have
\begin{equation*}
	\mathrm{V}^b_{n,\bs{h}^*,p^*}(x)\lesssim v_n^{LF} ( 1 + r_n^{LF}),
\end{equation*}
with
\begin{equation*}
r_n^{LF}
=
    \begin{cases}
        0
	& 
	\, \text{ if } d=1 \text{ or } d=2,
        \\
        \Delta_n^{-1} \log(n) h_1^* h_2^*
    	& 
	\, \text{ if } d \geq 3 \text{ and } (\bs{\alpha}, k_0) \in D_1,
        \\
        \Delta_n^{-1} \prod_{i=1}^{k_0} (h_i^*)^{2/k_0} 
        + 
        \log (n) \Delta_n^{-1} \prod_{i=1}^d h_i^*
	& 
	\, \text{ if } d\geq 3 \text{ and } (\bs{\alpha}, k_0) \in D_2,
        \\
         \log(n)  \Delta_n^{-1} \prod_{i=1}^d h_i^*
    +
    \Delta_n^{-1} h_1 (h_2^* h_3^*)^{1/2} 
    	& 
	\, \text{ if } d\geq 3 \text{ and } (\bs{\alpha}, k_0) \in D_3.
    \end{cases}
\end{equation*}
Using Lemma~\ref{lem:simplificationp} and the fact that $\Delta_n \gtrsim w_n^{HF}$, we know that $\Delta_n \gtrsim \widetilde{w}_n^{HF}$. Using also the definition of $\bs{h}^*$, we check that  
\begin{equation}
\label{eq:bound:delta_n:prop:smallnoise:lowsampling}
    \Delta_n \gtrsim 
        \begin{cases}
        h_1^*|\log(h_1^*)|
	& 
	\, \text{ if } d=1,
        \\
        h_1^*h_2^*|\log(h_1^*h_2^*)| 
	& 
	\, \text{ if } d=2,
        \\
    h_1h_2 \log(n)
	& 
	\, \text{ if } d \geq 3 \text{ and } (\bs{\alpha}, k_0) \in D_1,
        \\
        \prod_{i=1}^{k_0} (h_i^*)^{2/k_0}  
	& 
	\, \text{ if } d\geq 3 \text{ and } (\bs{\alpha}, k_0) \in D_2,
        \\
    	h_1^* (h_2^* h_3^*)^{1/2}
    	& 
	\, \text{ if } d\geq 3 \text{ and } (\bs{\alpha}, k_0) \in D_3,
    \end{cases}
\end{equation}
and therefore
\begin{equation*}
\mathrm{V}^b_{n,\bs{h}^*,p^*}(x)
\lesssim
v_n^{LF}
    \begin{cases}
        1 
	& 
	\, \text{ if } d=1 \text{ or } d=2, \text{ or } d \geq 3 \text{ and } (\bs{\alpha}, k_0) \in D_1,
        \\
         1
        + 
        \log (n) \Delta_n^{-1} \prod_{i=1}^d h_i^*
	& 
	\, \text{ if } d \geq 3 \text{ and } (\bs{\alpha}, k_0) \in D_2,
        \\
        1
        +
         \log(n)  \Delta_n^{-1} \prod_{i=1}^d h_i^*
    	& 
	\, \text{ if } d \geq 3 \text{ and } (\bs{\alpha}, k_0) \in D_3.
    \end{cases}
\end{equation*}
We conclude the proof of Proposition~\ref{prop:smallnoise:lowsampling} using~\eqref{eq:bound:delta_n:prop:smallnoise:lowsampling} to get 
\begin{equation*}
    \log(n)  \Delta_n^{-1} \prod_{i=1}^d h_i^* 
    \lesssim
    \begin{cases}
        \log(n)
        \prod_{i=1}^{k_0} (h_i^*)^{(2-k_0)/k_0}
        \prod_{i=k_0+1}^d h_i^*
          	& 
	\, \text{ if } d \geq 3 \text{ and } (\bs{\alpha}, k_0) \in D_2,
        \\
    	(h_2^* h_3^*)^{1/2} \prod_{i=4}^d h_i^* 
    	& 
	\, \text{ if } d \geq 3 \text{ and } (\bs{\alpha}, k_0) \in D_3,
    \end{cases}
\end{equation*}
so that
\begin{equation*}
    \log(n)  \Delta_n^{-1} \prod_{i=1}^d h_i^*  \to 0
\end{equation*}
as $n \to \infty$.

\subsection{Proof of Proposition~\ref{prop:largenoise:lowsampling}}

Suppose that $p^*\Delta_n \geq w_n^{HF}$ and that $\tau_n^{2\alpha_1} \geq \Delta_n^\beta$. Recall that in that case, we take
\begin{equation*}
        p^* = \Big\lceil \big(\tau_{n}^{2\alpha_1} \Delta_n^{-\beta} \big)^{1/(\beta+\alpha_1)} \Big\rceil \vee 1 \;\;\text{ and }\;\; \bs{h}^* = \bs{h}^{*,p^*}.
\end{equation*}
Thus $p^*\Delta_n \asymp (\tau_n^2\Delta_n)^{\alpha_1/(\beta+\alpha_1)}$ and $(p^*\Delta_n)^\beta \asymp \tau_n^{2\alpha_1}/(p^*)^{\alpha_1} \asymp (\tau_n^2\Delta_n)^{\alpha_1\beta/(\beta+\alpha_1)}$.
From Proposition~\ref{propo:small:bias:mu}, we see by plugging these values that
\begin{equation*}
    \mathrm{B}^b_{n,\bs{h}^*,p^*}(x)^2 \lesssim \big(\tau_n^2 \Delta_n\big)^{\frac{\alpha_1 \beta}{ \beta+\alpha_1}} + \Big(\frac{p^*}{n}\Big)^{\frac{2\overline{\alpha}}{2\overline{\alpha}+d}}.
\end{equation*}
For the variance, we proceed as for Proposition~\ref{prop:smallnoise:lowsampling}: using that $p^*\Delta_n \gtrsim w_n^{HF}$ (and therefore that $p^*\Delta_n \gtrsim \widetilde{w}_n^{HF}$ by Lemma~\ref{lem:simplificationp}) and $\bs{h}^* = \bs{h}^{*, p^*}$, we have from Proposition~\ref{prop:variance} that
\begin{equation*}
    \mathrm{V}_{n,\bs{h}^*,p^*}^b(x) \lesssim \Big(\frac{p^*}{n}\Big)^{\frac{2\overline{\alpha}}{2\overline{\alpha}+d}},
\end{equation*}
and thus
\begin{equation*}
\mathcal{R}(\widehat\mu_{n,\bs{h}^*,p^*}, b; \, x) 
\lesssim (\tau_n^{2}\Delta_n)^{\frac{\alpha_1  \beta}{ \beta +\alpha_1}} + \Big (\frac{p^*}{n}\Big )^{\frac{2\overline{\alpha}}{2\overline{\alpha}+d}}.
\end{equation*}
It remains to prove that $(\tau_n^{2}\Delta_n)^{\frac{\alpha_1\beta}{\beta+\alpha_1}}$ is always dominant here. This term is indeed of the order of $(p^* \Delta_n)^{\beta}$. By Lemma~\ref{lem:simplificationp}, we have $p^* \Delta_n \gtrsim \widetilde{w}_{n,p^*}^{HF}$ so it is enough to prove that
\begin{equation*}
    \Big (\frac{p^*}{n}\Big )^{\frac{2\overline{\alpha}}{2\overline{\alpha}+d}} \lesssim 
    (\widetilde{w}_{n,p^*}^{HF})^{\beta}.
\end{equation*}
Moreover, we easily check that
\begin{equation*}
        \begin{cases}
         \frac{\beta}{2\overline{\alpha} + 1} \leq \frac{2\overline{\alpha}}{2\overline{\alpha}+d} & \, \text{ if } d=1,
         \\
         \beta\frac{\overline{\alpha}}{2\overline{\alpha} + 2} ( \frac{1}{\alpha_1} + \frac{1}{\alpha_2} ) \leq \frac{2\overline{\alpha}}{2\overline{\alpha}+d} & \, \text{ if } d=2,
        \\
        \beta\frac{\overline{\alpha}}{2\overline{\alpha} + d} ( \frac{1}{\alpha_1} + \frac{1}{\alpha_2} ) \leq \frac{2\overline{\alpha}}{2\overline{\alpha}+d} 
        & \, \text{ if } d \geq 3 \text{ and } (\bs{\alpha}, k_0) \in D_1,
        \\
        \beta\frac{2\overline{\alpha}}{2\overline{\alpha} + d} \frac{1}{\alpha_1} \leq \frac{2\overline{\alpha}}{2\overline{\alpha}+d}
        & \, \text{ if } d \geq 3 \text{ and } (\bs{\alpha}, k_0) \in D_2,
        \\
        \beta\frac{\overline{\alpha}}{(2\overline{\alpha} + d)} \Big( \frac{1}{\alpha_1} + \frac{1}{\alpha_2} \Big) \leq \frac{2\overline{\alpha}}{2\overline{\alpha}+d}
        & \, \text{ if } d \geq 3 \text{ and } (\bs{\alpha}, k_0) \in D_3,
    \end{cases}
\end{equation*}
using the fact that $2 \leq \beta \leq \alpha_1 \leq \alpha_2$. We conclude from the definition of $\widetilde{w}_{n,p^*}^{HF}$.

\section{Proof of Proposition~\ref{prop:oracle}} 
\label{proof:prop:oracle} 

In this section, we will repeatedly use estimates of the form
\begin{equation*}
\int_\nu^{\infty} \exp \left(-z^r\right) \, \mathrm{d} z \leq 2 r^{-1} \nu^{1-r} \exp \left(-\nu^r\right), \quad \nu, r>0, \quad \nu \geq(2 / r)^{1 / r}
\end{equation*}
and
\begin{equation}\label{ineq:exp:2}
\exp \left(-\frac{a z^p}{b+c z^{p / 2}}\right) \leq \exp \left(-\frac{a z^p}{2 b}\right)+\exp \left(-\frac{a z^{p / 2}}{2 c}\right), \quad a,b,c,z>0.
\end{equation}
Before delving into the proof of Proposition~\ref{prop:oracle}, let us state and prove the following lemma.

\begin{lemma}\label{lemma:control:A}
Assume the conditions of Proposition~\ref{prop:oracle}. For every $\gamma>1$, if $\widebar\omega$ in~\eqref{eq:def:inflated:penalty} is large enough, then, uniformly in $b\in\Sigma(\bs{\mathfrak b})$, $p$ and $\bs h\in\mathcal H_n^p$,
\[
\E_{\widebar\mu^b}^b[A_n^p(\bs h)]
\lesssim
\widetilde V_n^p(\bs h)+\mathrm B_{n,\bs h,p}^b(x)^2
+\left(\frac pn\right)^\gamma
+(p\Delta_n)^\beta\mathds 1_{p\ge2}
+\frac{\tau_n^{2\alpha_1}}{p^{\alpha_1}}.
\]
\end{lemma}

\begin{proof}
Put $h_p=p\Delta_n$, $L_p=\log(n_p)$ and $\Lambda_p=L_p+\log(e/h_p)$. By Theorem~\ref{th:bernstein:loc}, applied to the finite Richardson--deconvolution linear combination defining $\widehat\mu$, there are constants $K_0,c_0,c_1<\infty$, depending only on $\bs{\mathfrak b}$, such that for every $\bs\eta\in\mathcal H_n^p$ and every $\varepsilon>0$,
\begin{equation}\label{eq:oracle:bernstein:centered}
\mathbb P_{\widebar\mu^b}^b\left(
\left|\widehat\mu_{n,\bs\eta,p}(x)-\mathbb E_{\widebar\mu^b}^b\widehat\mu_{n,\bs\eta,p}(x)\right|>\varepsilon
\right)
\le
K_0\exp\left(
-\frac{\varepsilon^2}
{c_0v_n^p(\bs\eta)+c_1h_p^{-2}\varepsilon\|\mathbb K_{\bs\eta}\|_\infty L_p/n_p}
\right).
\end{equation}
Indeed, the asymptotic variance in Theorem~\ref{th:bernstein:loc} is bounded by $C n_pv_n^p(\bs\eta)$, while the finite number of Richardson and deconvolution coefficients only changes the constants; the physical block lengths appearing in the finite combination are all comparable to $h_p$.

Let
\[
Z_{\bs\eta}
=
\widehat\mu_{n,\bs\eta,p}(x)-\mathbb E_{\widebar\mu^b}^b\widehat\mu_{n,\bs\eta,p}(x),
\qquad
M_{\bs\eta}=\|\mathbb K_{\bs\eta}\|_\infty,
\qquad
\widetilde V_{\bs\eta}=\widetilde V_n^p(\bs\eta).
\]
Using~\eqref{ineq:exp:2} in~\eqref{eq:oracle:bernstein:centered} and integrating the tail of $|Z_{\bs\eta}|^2$, we get
\begin{equation*}
\E_{\widebar\mu^b}^b\left[\{|Z_{\bs\eta}|^2-\widetilde V_{\bs\eta}\}_+\right]
=
\int_{\widetilde V_{\bs\eta}}^\infty
\mathbb P_{\widebar\mu^b}^b(|Z_{\bs\eta}|\ge z^{1/2})\,dz \le I_1(\bs\eta)+I_2(\bs\eta),
\end{equation*}
where
\[
I_1(\bs\eta)
\le
K_0\int_{\widetilde V_{\bs\eta}}^\infty
\exp\left(-\frac{z}{2c_0v_n^p(\bs\eta)}\right)\,dz
\lesssim
v_n^p(\bs\eta)
\exp\left(-c\frac{\widetilde V_{\bs\eta}}{v_n^p(\bs\eta)}\right)
\]
and
\[
I_2(\bs\eta)
\le
K_0\int_{\widetilde V_{\bs\eta}}^\infty
\exp\left(
-c\frac{h_p^2n_p z^{1/2}}{M_{\bs\eta}L_p}
\right)\,dz.
\]
Since $\widetilde V_{\bs\eta}\ge \widebar\omega L_pv_n^p(\bs\eta)$, the first integral satisfies
\begin{equation}\label{eq:oracle:I1}
I_1(\bs\eta)
\lesssim
v_n^p(\bs\eta)n_p^{-c\widebar\omega}.
\end{equation}
For the second integral, the elementary identity
\[
\int_a^\infty e^{-q\sqrt z}\,dz
=2(q^{-2}+q^{-1}\sqrt a)e^{-q\sqrt a},
\qquad a,q>0,
\]
with $q=ch_p^2n_p/(M_{\bs\eta}L_p)$ gives
\begin{equation}\label{eq:oracle:I2:first}
I_2(\bs\eta)
\lesssim
\left[
h_p^{-4}\frac{M_{\bs\eta}^2L_p^2}{n_p^2}
+h_p^{-2}\frac{M_{\bs\eta}L_p\sqrt{\widetilde V_{\bs\eta}}}{n_p}
\right]
\exp\left(
-c\frac{h_p^2n_p\sqrt{\widetilde V_{\bs\eta}}}{M_{\bs\eta}L_p}
\right).
\end{equation}
The added term in~\eqref{eq:def:inflated:penalty} implies
\[
\frac{h_p^2n_p\sqrt{\widetilde V_{\bs\eta}}}{M_{\bs\eta}L_p}
\ge c\widebar\omega^{1/2}\Lambda_p.
\]
Furthermore, the grid lower bound gives $M_{\bs\eta}\lesssim n_pL_p^{-3}$, $v_n^p(\bs\eta)\lesssim n_p/T_n$ and $M_{\bs\eta}\sqrt{v_n^p(\bs\eta)}\lesssim 1+\sqrt{n_p/T_n}$. Since
\[
\sqrt{\widetilde V_{\bs\eta}}
\lesssim_{\widebar\omega}
L_p^{1/2}v_n^p(\bs\eta)^{1/2}
+h_p^{-2}\frac{M_{\bs\eta}L_p\Lambda_p}{n_p},
\]
the bracket in~\eqref{eq:oracle:I2:first} is bounded, uniformly over the grid and for $n$ large enough, by
\[
C_{\widebar\omega}h_p^{-4}\Lambda_p\frac{n_p}{T_n}.
\]
Indeed, the terms containing $M_{\bs\eta}$ are controlled by $M_{\bs\eta}\lesssim n_pL_p^{-3}$, while the mixed term is controlled by $M_{\bs\eta}\sqrt{v_n^p(\bs\eta)}\lesssim1+\sqrt{n_p/T_n}$ and $T_n\le n_p$, which follows from $h_p\le1$. Therefore, if $a=c\widebar\omega^{1/2}$,~\eqref{eq:oracle:I2:first} gives
\[
I_2(\bs\eta)
\lesssim
h_p^{-4}\Lambda_p\frac{n_p}{T_n}
\exp(-a\Lambda_p)
\le
\frac{n_p}{T_n}n_p^{-a}h_p^{a-4}\Lambda_p.
\]
For $a>5$, the factor $h_p^{a-4}\log(e/h_p)$ is bounded uniformly in $0<h_p\le1$, and the remaining factor $L_p$ is absorbed by decreasing the power of $n_p$ by an arbitrarily small amount. Choosing $\widebar\omega$ so large that $a>\gamma+3$ yields
\begin{equation*}
\sup_{\bs\eta\in\mathcal H_n^p} I_2(\bs\eta)
\lesssim
T_n^{-1}n_p^{-\gamma}.
\end{equation*}
The same choice of $\widebar\omega$ and~\eqref{eq:oracle:I1} also give
\begin{equation*}
\sup_{\bs\eta\in\mathcal H_n^p} I_1(\bs\eta)
\lesssim
T_n^{-1}n_p^{-\gamma}.
\end{equation*}
Consequently, since $\#\mathcal H_n^p\le T_n$,
\begin{equation}\label{eq:oracle:onecentered}
\E_{\widebar\mu^b}^b\left[
\max_{\bs\eta\in\mathcal H_n^p}
\{|Z_{\bs\eta}|^2-\widetilde V_n^p(\bs\eta)\}_+
\right]
\lesssim n_p^{-\gamma}
=\left(\frac pn\right)^\gamma.
\end{equation}
The convolution estimator $\widehat\mu_{n,(\bs h,\bs\eta),p}$ is controlled in exactly the same way, because
\[
\|\mathbb K_{\bs h}*\mathbb K_{\bs\eta}\|_\infty
\le \|\mathbb K\|_1\|\mathbb K_{\bs\eta}\|_\infty,
\qquad
\|\mathbb K_{\bs h}*\mathbb K_{\bs\eta}\|_1
\le \|\mathbb K\|_1^2,
\]
so the same penalty $\widetilde V_n^p(\bs\eta)$ controls the Bernstein linear term.

It remains only to repeat the deterministic Goldenshluger--Lepski decomposition. For every $\bs\eta\in\mathcal H_n^p$,
\begin{align*}
&\left\{\left|\widehat{\mu}_{n,(\bs{h}, \bs{\eta}),p}\left(x\right)-\widehat{\mu}_{n,\bs{\eta},p}\left(x\right)\right|^2-\widetilde V^p_n(\bs{\eta})\right\}_{+}\\
&\qquad\lesssim
\{|Z_{\bs\eta}|^2-\widetilde V_n^p(\bs\eta)\}_+
+\{|Z_{(\bs h,\bs\eta)}|^2-\widetilde V_n^p(\bs\eta)\}_+\\
&\qquad\quad+
\left|
\mathbb E_{\widebar\mu^b}^b\widehat\mu_{n,\bs\eta,p}(x)
-\mathbb E_{\widebar\mu^b}^b\widehat\mu_{n,(\bs h,\bs\eta),p}(x)
\right|^2.
\end{align*}
By Proposition~\ref{propo:small:bias:mu}, the deterministic term is bounded by
\[
\mathrm B_{n,\bs h,p}^b(x)^2
+(p\Delta_n)^\beta\mathds 1_{p\ge2}
+\frac{\tau_n^{2\alpha_1}}{p^{\alpha_1}}.
\]
Combining this bound with~\eqref{eq:oracle:onecentered} proves the lemma.
\end{proof}

We are now ready to move on to the proof of Proposition~\ref{prop:oracle}.

\begin{proof}[Proof of Proposition~\ref{prop:oracle}]
 Fix $\bs h\in\mathcal H_n^p$. The grid is finite by construction, so the minimizer $\bs h^*$ in~\eqref{eq:def:hstar:inflated} may be chosen in $\mathcal H_n^p$. Since $\bs h\in(0,1]^d$, $p^2\Delta_n\le c_\star$ and $0<p\Delta_n\le1$, Proposition~\ref{prop:variance} applies. Together with the bias--variance decomposition~\eqref{eq:biasvariancedecompo}, and after increasing the lower bound $\widebar\omega_0$ if necessary so that the first summand $\widebar\omega L_pv_n^p(\bs h)$ of $\widetilde V_n^p(\bs h)$ dominates the variance constant in Proposition~\ref{prop:variance}, this gives, for all $n\ge n_0$,
\begin{equation}\label{eq:oracle:risk:fixedh}
\E_{\widebar\mu^b}^b\left[
\left|\widehat\mu_{n,\bs h,p}(x)-\widebar\mu^b(x)\right|^2
\right]
\lesssim
\mathrm B_{n,\bs h,p}^b(x)^2+\widetilde V_n^p(\bs h).
\end{equation}
 We next compare the selected estimator to the fixed one. The elementary inequality $|u-v|^2\le2|u-w|^2+2|w-v|^2$, applied with $w=\widehat\mu_{n,(\bs h,\bs h^*),p}(x)$, yields
\begin{align*}
\left|\widehat\mu_{n,\bs h^*,p}(x)-\widehat\mu_{n,\bs h,p}(x)\right|^2
&\lesssim
\left|\widehat\mu_{n,\bs h^*,p}(x)-\widehat\mu_{n,(\bs h,\bs h^*),p}(x)\right|^2\\
&\quad+
\left|\widehat\mu_{n,\bs h,p}(x)-\widehat\mu_{n,(\bs h,\bs h^*),p}(x)\right|^2.
\end{align*}
For the first term, take $\bs\eta=\bs h^*$ in the maximum defining $A_n^p(\bs h)$ and use $y\le \{y-a\}_+ +a$ for $y,a\ge0$. This gives
\[
\left|\widehat\mu_{n,\bs h^*,p}(x)-\widehat\mu_{n,(\bs h,\bs h^*),p}(x)\right|^2
\le
A_n^p(\bs h)+\widetilde V_n^p(\bs h^*).
\]
For the second term, the convolution kernel is symmetric in its two bandwidth arguments, hence $\widehat\mu_{n,(\bs h,\bs h^*),p}=\widehat\mu_{n,(\bs h^*,\bs h),p}$. Taking now $\bs\eta=\bs h$ in $A_n^p(\bs h^*)$ gives
\[
\left|\widehat\mu_{n,\bs h,p}(x)-\widehat\mu_{n,(\bs h,\bs h^*),p}(x)\right|^2
\le
A_n^p(\bs h^*)+\widetilde V_n^p(\bs h).
\]
 Since $\bs h^*$ minimizes $A_n^p(\cdot)+\widetilde V_n^p(\cdot)$ over $\mathcal H_n^p$,
\[
A_n^p(\bs h^*)+\widetilde V_n^p(\bs h^*)
\le
A_n^p(\bs h)+\widetilde V_n^p(\bs h).
\]
Combining the preceding three displays,
\[
\left|\widehat\mu_{n,\bs h^*,p}(x)-\widehat\mu_{n,\bs h,p}(x)\right|^2
\lesssim
A_n^p(\bs h)+\widetilde V_n^p(\bs h).
\]
Applying once more $|u-v|^2\le2|u-w|^2+2|w-v|^2$, now with $w=\widehat\mu_{n,\bs h,p}(x)$, and then taking expectations, \eqref{eq:oracle:risk:fixedh} and the previous display give
\begin{align*}
\E_{\widebar\mu^b}^b\left[
\left|\widehat\mu_{n,\bs h^*,p}(x)-\widebar\mu^b(x)\right|^2
\right]
&\lesssim
\mathrm B_{n,\bs h,p}^b(x)^2+\widetilde V_n^p(\bs h)
+\E_{\widebar\mu^b}^b[A_n^p(\bs h)]\\
&\lesssim
\mathrm B_{n,\bs h,p}^b(x)^2+\widetilde V_n^p(\bs h)
+\left(\frac pn\right)^\gamma
+(p\Delta_n)^\beta\mathds 1_{p\ge2}
+\frac{\tau_n^{2\alpha_1}}{p^{\alpha_1}},
\end{align*}
where the last line uses Lemma~\ref{lemma:control:A}. Because $\bs h\in\mathcal H_n^p$ was arbitrary and the final three terms are independent of $\bs h$, taking the infimum over $\mathcal H_n^p$ gives~\eqref{eq:oracle:inflated}.
\end{proof}

\end{document}